\documentclass[10pt]{article}
\usepackage[margin = 1in]{geometry}

\usepackage[utf8]{inputenc} 
\usepackage[T1]{fontenc}    
\usepackage{url}            
\usepackage{booktabs}       
\usepackage{amsfonts}       
\usepackage{nicefrac}       
\usepackage{microtype}      
\usepackage{xcolor}         
\usepackage[sort]{natbib}

\usepackage{algorithmic}

\usepackage{soul}

\usepackage{amsmath, amssymb}
\usepackage{amsthm}
\usepackage{bbm}
\usepackage{ulem}
\usepackage{mathbbol}
\usepackage{float}
\usepackage{dirtytalk}
\usepackage{mathrsfs}
\DeclareMathOperator*{\argminA}{arg\,min}
\usepackage[english]{babel}

\usepackage{booktabs}
\usepackage{empheq}
\usepackage{mdframed}
\usepackage{pifont}
\usepackage{enumitem}
\usepackage[verbose]{wrapfig}
\usepackage{placeins}
\usepackage{empheq}
\usepackage{bm}

\definecolor{mydarkblue}{rgb}{0,0.08,0.45}
\usepackage[colorlinks=true,
    linkcolor=mydarkblue,
    citecolor=mydarkblue,
    filecolor=mydarkblue,
    urlcolor=mydarkblue,
    pdfview=FitH]{hyperref}

\usepackage{tikz}
\usetikzlibrary{shapes}
\usetikzlibrary{decorations.markings}
\usetikzlibrary{plotmarks}
\usetikzlibrary{patterns}
\usepackage{ctable}
\usepackage{multirow}
\usepackage{multicol}

\usepackage{color, colortbl}
\definecolor{mydarkblue}{rgb}{0,0.08,0.45}
    
\usepackage{empheq}

\definecolor{myteal}{RGB}{27,158,119}
\definecolor{myorange}{RGB}{217,95,2}
\definecolor{myred}{RGB}{231,41,138}
\definecolor{mypurple}{RGB}{152,78,163}
\definecolor{myblue}{RGB}{55,126,184}
\definecolor{mygreen}{RGB}{0,100,0}

\usepackage{CJKutf8}
\usepackage{array}
\usepackage{mathtools}
\usepackage{cancel}
\usepackage{verbatim}
\usepackage{xparse}

\usepackage[ruled,vlined]{algorithm2e}
\usepackage{algorithmic}
\usepackage{bm}
\usepackage{bbm}
\usepackage{mathbbol}

\def\xx{{\boldsymbol x}}

\def\XX{{\boldsymbol X}}

\def\ZZ{{\boldsymbol Z}}
\def\aa{{\boldsymbol a}}
\def\bb{{\boldsymbol b}}

\def\ww{{\boldsymbol w}}

\def\II{{\boldsymbol I}}

\def\ww{{\boldsymbol w}}

\def\AA{{\boldsymbol A}}

\def\MM{{\boldsymbol M}}
\def\DD{{\boldsymbol D}}
\def\OO{{\boldsymbol O}}
\def\PP{{\boldsymbol P}}

\def\UU{{\boldsymbol U}}
\def\VV{{\boldsymbol V}}

\def\SSigma{{\boldsymbol \Sigma}}

\newcommand{\E}{\mathbb{E}}
\newcommand{\EE}{\mathcal{E}}
\newcommand{\K}{\mathcal{K}}

\newcommand{\X}{\mathcal{X}}
\newcommand{\Y}{\mathcal{Y}}

\def\nnu{{\boldsymbol \nu}}

\def\eeta{{\boldsymbol \eta}}

\def\RR{{\mathbb R}}
\def\defas{\stackrel{\text{def}}{=}}

\DeclareMathOperator{\tr}{tr}
\def\defas{\stackrel{\text{def}}{=}}

\DeclareMathOperator*{\diag}{\mathbf{diag}}

\newtheorem{lemma}{Lemma}
\newtheorem{theorem}{Theorem}
\newtheorem{corollary}{Corollary}
\newtheorem{prop}{Proposition}
\newtheorem{definition}{Definition}
\newtheorem{remark}{Remark}
\newtheorem{assumption}{Assumption}[section]

\title{Trajectory of Mini-Batch Momentum\\
{\Large Batch Size Saturation and Convergence in High Dimensions} }

\author{%
  Kiwon Lee\thanks{Department of Mathematics and Statistics, McGill University, Montreal, QC; KL email \url{kiwon.lee@mail.mcgill.ca}; ANC email \url{andrew.cheng@mail.mcgill.ca}; CP is a CIFAR AI chair, MILA and CP was supported by a Discovery Grant from the
Natural Science and Engineering Research Council (NSERC) of Canada; website \url{https://cypaquette.github.io/} and email \url{courtney.paquette@mcgill.ca}. Research by EP was supported by a Discovery Grant from the
Natural Science and Engineering Research Council (NSERC) of Canada; website \url{https://elliotpaquette.github.io/} and email: \url{elliot.paquette@mcgill.ca}. } 
    \and Andrew N. Cheng\footnotemark[1]
    \and Elliot Paquette\footnotemark[1]
    \and Courtney Paquette\footnotemark[1] \thanks{Google Research, Brain Team}
}

\date{}

\DeclareDocumentCommand{\Prto} {o} {
  \IfNoValueTF {#1}
  {\overset{\Pr}{\longrightarrow}}
  { \xrightarrow[ #1 \to \infty]{\Pr }}
}

\DeclareDocumentCommand{\law} {o} {
  \IfNoValueTF {#1}
  {\overset{\text{law}}{=}}
  { \xrightarrow[ #1 \to \infty]{\Pr }}
}

\begin{document}
\maketitle

\begin{abstract}
 We analyze the dynamics of large batch stochastic gradient descent with momentum (SGD+M) on the least squares problem when both the number of samples and dimensions are large. In this setting, we show that the dynamics of SGD+M converge to a deterministic discrete Volterra equation as dimension increases, which we analyze.  We identify a stability measurement, the implicit conditioning ratio (ICR), which regulates the ability of SGD+M to accelerate the algorithm.  When the batch size exceeds this ICR, SGD+M converges linearly at a rate of $\mathcal{O}(1/\sqrt{\kappa})$, matching optimal full-batch momentum (in particular performing as well as a full-batch but with a fraction of the size).  For batch sizes smaller than the ICR, in contrast, SGD+M has rates that scale like a multiple of the single batch SGD rate. We give explicit choices for the learning rate and momentum parameter in terms of the Hessian spectra that achieve this performance.
 \end{abstract}
 
 Stochastic learning algorithms are the methods of choice for optimization of high-dimensional problems. Often stochastic learning algorithms incorporate momentum into their stochastic gradients to improve practical performance. Perhaps the simplest, stochastic gradient descent with momentum (SGD+M) adds a fixed multiple of the backward difference of iterates to its stochastic gradient estimator, see Section~\ref{sec:motivation} for details. In the influential work of \citep{sutskever13}, the authors empirically show augmenting stochastic gradient descent (SGD) with momentum significantly improves training performance of deep neural networks. Despite the wide usage of these stochastic momentum methods in machine learning practice, our understanding of its behaviour is not well--understood. 
 
 It has been hypothesized that stochastic-based momentum algorithms improve training because they are employed on a large batch of a data set \citep{kidambi2018}; thereby emulating the speed-up one sees in full-batch settings. For many learning problems, the ``large batch'' setting is often paired with high-dimensional problems, meaning there are many samples (and likely also many features to have interesting behavior). 
 We know of no theoretical analysis that can justify this claim for standard SGD+M, although for variations of SGD+M and SGD with Nesterov momentum \citep{nesterov2004introductory} there has been some success in proving accelerated rates \citep{jain2018accelerating,Liu2020accelerating,allen2017katyusha}. A reason is that typical approaches for analyzing SGD+M do not distinguish large and small batch sizes. We address this problem in this paper and we introduce a stability measurement that exactly captures the transition of SGD+M to an accelerated method. We comment that in the high--dimensional, \textit{vanishing batch fraction setting} (the mini-batch size is $o(n)$, where $n$ is the number of samples), there is work proving in various simplified settings that SGD+M produces the same iterates as SGD with a larger learning rate, up to a vanishing error \citep{paquette2021dynamics}.


In this paper, we study the dynamics of mini-batch SGD+M (with constant learning rate) on a least squares problem when the number of samples $n$ and features $d$ are large (see Section~\ref{sec:motivation} for details). We are motivated by the setting where the mini-batch size $\beta$ is proportionate to the number of samples $n$ and so we define the ratio $\zeta \defas \beta/n$, which we refer to as the \textit{batch fraction}.  We provide a non-asymptotic comparison for the behavior of the training loss under SGD+M to a deterministic function $\psi$, whose accuracy improves when the number of samples and features are large while the batch fraction is strictly positive (see Figure~\ref{fig:volterra_sgdm_comparison}).
 \begin{wrapfigure}[49]{r}{0.48\textwidth}
\centering
\vspace{-0.4cm}
\includegraphics[width=0.45\textwidth]{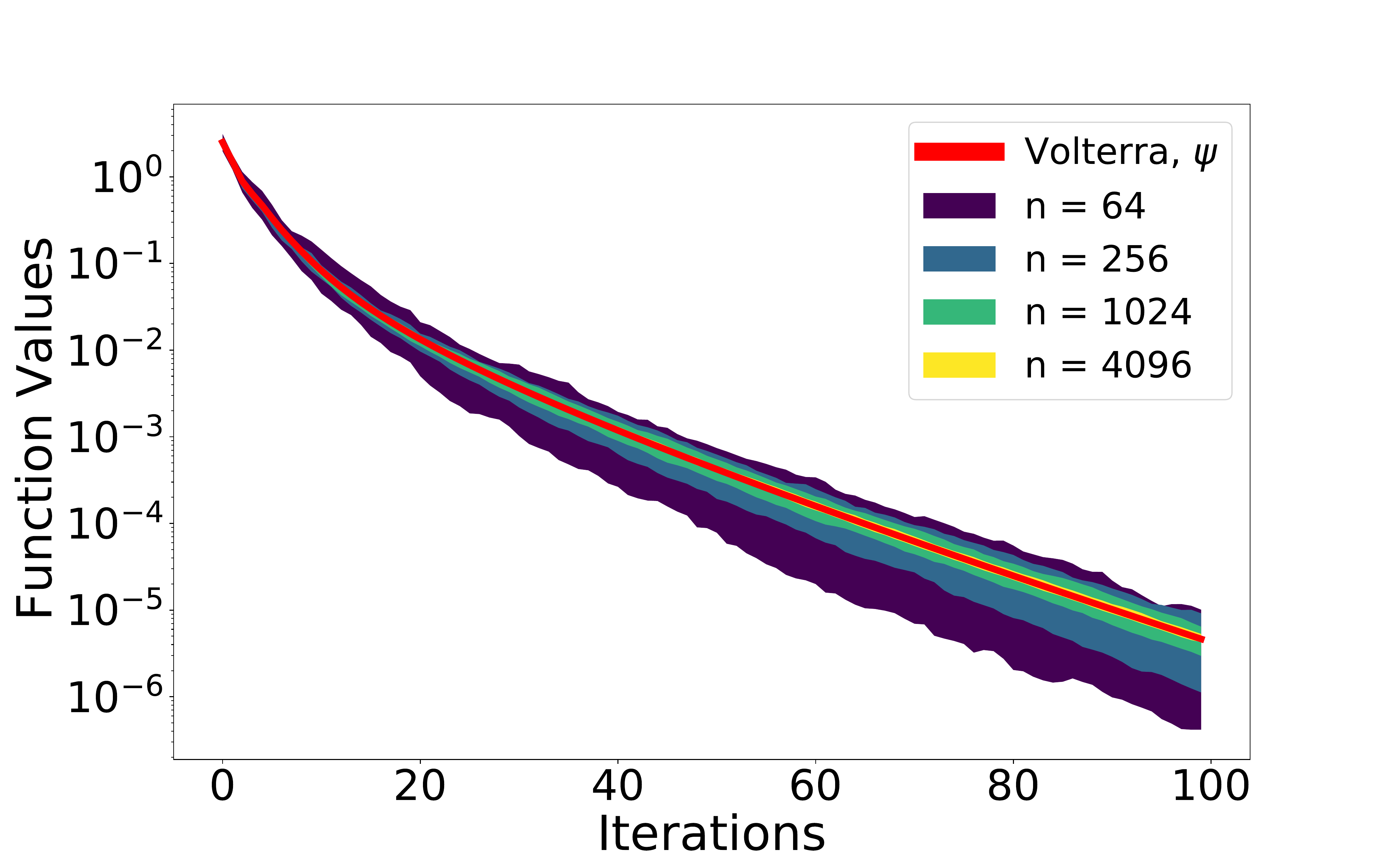}
\caption{\textbf{Concentration of SGD+M on a Gaussian random least squares
problem} such that the ratio $d/n$ is fixed to be $2$; 30 runs of SGD+M and the $80$th percentile confidence intervals recorded (shaded region) for each $n$. The parameters for SGD+M are $\Delta = 0.5$, $\gamma = 0.4, \zeta=0.5$, see Section ~\ref{sec:random_lsp}. The random least squares problem becomes non-random in the large limit and all
runs of SGD+M converge to a deterministic function $\psi(t)$ (red) given by our Volterra equation \eqref{eq:volterra_main}.}
    \label{fig:volterra_sgdm_comparison}
\vspace{0.3cm}
\centering
\includegraphics[width=0.4\textwidth]{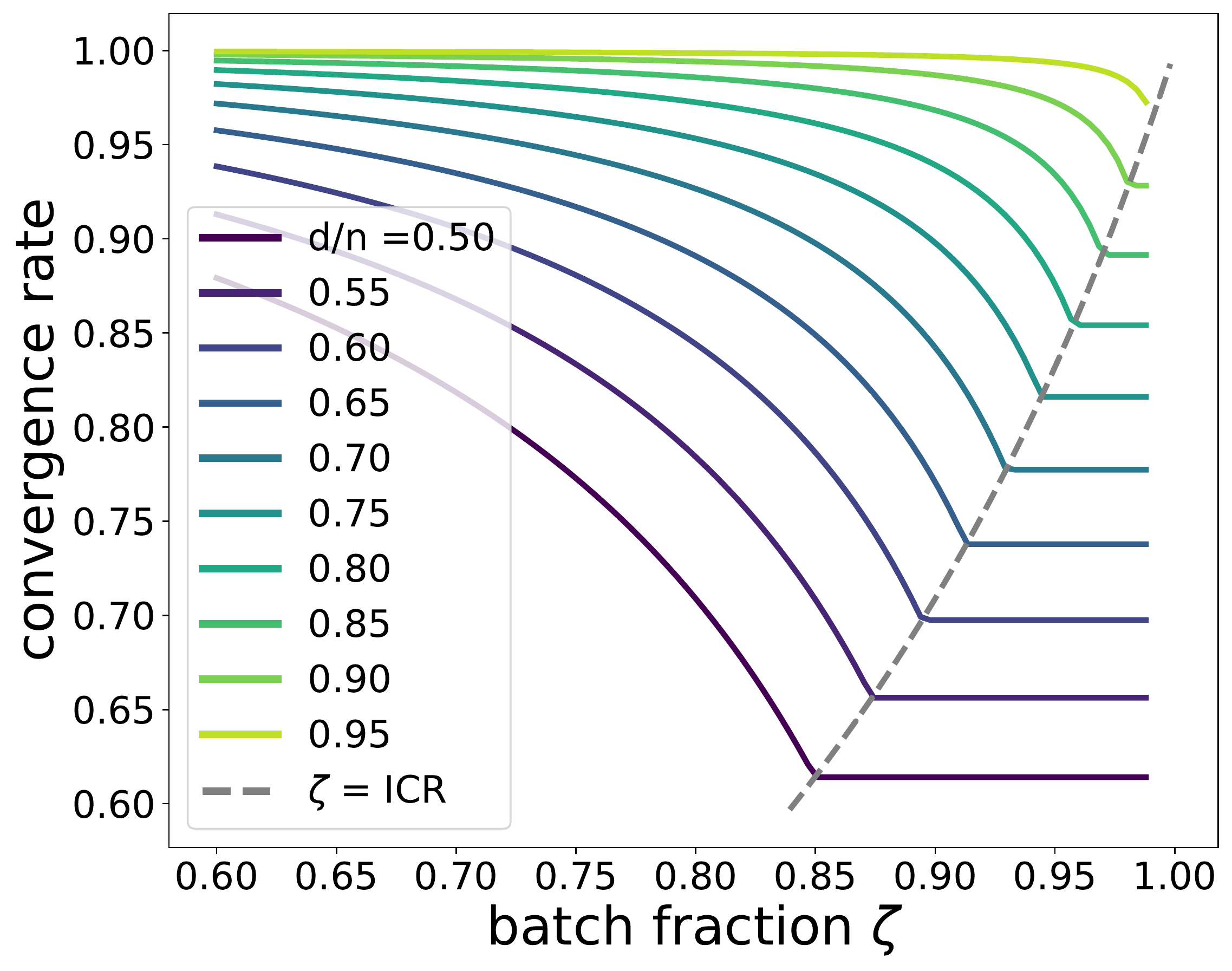}
\caption{\textbf{Effect of the batch fraction on the convergence rate of SGD+M} for a Gaussian random least squares
problem (data matrix $\AA \in \mathbb{R}^{n \times d}$) with ratio $r \defas d/n$ varying. Here, because of the well-known Marchenko-Pastur law (see \citep{marchenko1967} and Appendix~\ref{apx:app_numerical_simulations}),  $\bar{\kappa} = 1 / (1-\sqrt{1/r})^2 $ and $\kappa = (1+\sqrt{1/r})^2 / (1-\sqrt{1/r})^2$ (see \eqref{eq:condition_numbers} and Appendix~\ref{apx:app_numerical_simulations}). When $\zeta \ge \text{ICR}$ (large batch regime), the convergence rate freezes at $1 / \sqrt{\kappa}$ (see Proposition \ref{prop:lambda_delta_choice}). Otherwise the rate is $\zeta / ( (1-\zeta) \bar{\kappa})$ and there is no speed up from SGD+M over SGD. There is a saturation point, or \textit{saturating batch fraction} (dashed gray), after which increasing the batch fraction does not improve convergence. Moreover this point occurs before full batch (i.e., $\zeta = 1$) and it is fully explained by the trace, maximum and minimum (nonzero) eigenvalues of $\AA \AA^T$, that is when $\zeta = \text{ICR}$. 
}
\label{fig:convergence_rates_vs_batch_sizes}
\end{wrapfigure}
This function $\psi$ solves a discrete Volterra equation:
\begin{equation}\label{eq:volterra_main}
\psi(t+1) = F(t+1) + \sum_{k=0}^t \psi(k)\mathcal{K}(t-k).
\end{equation}
The \textit{forcing term} $F(t)$ and \textit{kernel} $\mathcal{K}(t)$ are explicit functions that depend on the hyperparameters and the full Hessian spectra (see Section \ref{sec:dynamics} and Appendix~\ref{apx:proof_main_result}).  They transparently reveal that the dynamics of SGD+M and of SGD are truly non-equivalent in that there is no mapping of the hyperparameters which leads them to have the same training dynamics.  We also note that a similar equation appears in the vanishing batch setting \citep{paquette21a}, although in that setting it is a Volterra integral equation, which can be recovered from \eqref{eq:volterra_main} by sending $\zeta\to0$.

An advantage of the exact loss trajectory is that we give a rigorous definition of the large batch and small batch regimes which reflect a transition in the convergence behavior of SGD+M. To do this we introduce 
 the \textit{condition number} $\kappa$,
the \textit{average condition number} 
 $\bar{\kappa}$, and the \textit{implicit conditioning ratio} (ICR) defined as
\begin{equation}\label{eq:condition_numbers}
\begin{gathered}
    \bar{\kappa} \defas \frac{\frac{1}{n}\sum_{j\in[n]} \sigma_j^2}{ \sigma_{\min}^2} <  \frac{\sigma_{\max}^2}{\sigma_{\min}^2} \defas \kappa \\
    \text{and} \quad \text{ICR} \defas \frac{\bar{\kappa}}{\sqrt{\kappa}}.
\end{gathered}
\end{equation}
Here $\sigma_j^2$ are the eigenvalues of the Hessian of the least squares problem with $\sigma_{\max}^2$ and $\sigma_{\min}^2$ the largest and smallest (non-zero) eigenvalues. We refer to the \textit{large batch} regime where $\zeta \geq \text{ICR}$ and the \textit{small batch} regime where $\zeta \leq \text{ICR}$.  In the large batch regime \eqref{eq:condition_numbers}, SGD+M matches the performance of the heavy-ball algorithm: the convergence is linear with rate given by $\mathcal{O}(1/\sqrt{\kappa})$.  In the small batch regime, the performance matches that of SGD, i.e.\ the convergence rate is $\mathcal{O}(\zeta/\overline{\kappa})$. We give matching lower bounds, and we provide momentum and learning rate choices that achieve the claimed performance.  In addition we show there is a \textit{saturating batch fraction} (see Figure~\ref{fig:convergence_rates_vs_batch_sizes}), after which increasing the batch fraction does not improve the rate. It explicitly occurs when $\zeta =$ ICR. Moreover this saturating batch fraction occurs before full batch, i.e.\ $\zeta = 1$. 

\paragraph{Related work.} Recent works have established convergence guarantees for SGD+M in both strongly convex and non-strongly convex setting \citep{flammarion15, sebbouh21}, including almost sure convergence \citep{gadat18}. In the work of \citep{orvieto20}, they used a stochastic differential equations (SDEs) to obtain convergence of SGD+M. Specializing to the setting of minimizing quadratics, \citep{loizou20} demonstrated that the iterates of SGD+M converge linearly (but not in L2) under an exactness assumption.
    
Determining batch size has been an important issue in determining the convergence rate of SGD and SGD+M. There are instances where (small batch size) SGD+M do not necessarily achieve better performances than small batch size SGD (see \citep{zhang19, kidambi2018, paquette21a}). As for mini-batch SGD without momentum, \citep{ma2018} showed that there is a saturating batch size (roughly $\overline{\kappa}/\kappa$) above which increasing the batch size no longer improves the rate. In \citep{de17}, the authors implemented an adaptive (increasing) batch size schedule and they used it to show linear convergence for SGD. For generalization, \citep{smith18} empirically showed that for SGD and SGD+M, instead of decaying the learning rate, one can increase the batch size during training to obtain a similar learning curve. 

SGD+M has been proven to be useful in practical applications as well, including machine learning. \citep{sutskever13} demonstrated that  SGD+M shows an empirical advantage in training deep and recurrent neural networks (DNNs and RNNs respectively). Many authors have proposed that learning rate warmup enables us to scale training efficiently to larger batch sizes (\citep{goyal17, mccandlish18, smith18}).

\section{Setting}\label{sec:motivation}
We consider the least squares problem when the number of samples ($n$) and features ($d$) are large:
 \begin{equation}\label{eq:lsq}
    \argminA_{\xx\in\mathbb{R}^d} \Big\{ f(\xx) = \frac{1}{n}\sum_{i=1}^n f_i(x) \overset{\mathrm{def}}{=} \frac{1}{2}\sum_{i=1}^n (\aa_i \xx - b_i)^2\Big\}, \quad \text{with $\bb \defas \AA \widetilde{\xx} + \eeta$,}
\end{equation}
where $\AA \in \mathbb{R}^{n \times d}$ is a data matrix whose $i$-th row is denoted by $\aa_i \in \mathbb{R}^d$, $\widetilde{\xx} \in \mathbb{R}^d$ is the signal vector, and $\eeta \in \mathbb{R}^n$ is a source of noise. The target $\bb = \AA \widetilde{\xx} + \eeta$ comes from a generative model corrupted by noise. We let $\sigma_1^2\ge \cdots \ge \sigma_n^2 \ge 0$ be the eigenvalues of the matrix $\AA \AA^T$ with $\sigma_{\max}^2$ and $\sigma_{\min}^2$ the largest and smallest (nonzero) eigenvalues.

We apply SGD with momentum (SGD+M) with mini-batches to the finite sum, quadratic problem \eqref{eq:lsq}. SGD+M iterates by selecting uniformly at random a subset 
$B_k \subseteq \{1,2,\cdots,n\}$ of cardinality $\beta$ 
and makes the update
\begin{equation}\label{eq:oursgd}
    \begin{split}
    \xx_{k+1} 
    &= \xx_k - \gamma\sum_{i \in B_k}\nabla f_{i}(\xx_k) + \Delta (\xx_k - \xx_{k-1})\\
    &= \xx_k - \gamma\AA^T \PP_k (\AA\xx_k - \bb) + \Delta (\xx_k - \xx_{k-1}), \quad
    \text{where} \quad \PP_k \defas \sum_{i \in B_k} \mathbb{e}_{i}\mathbb{e}_{i}^T,
    \end{split}
\end{equation}
with $\PP_k$ a random orthogonal projection matrix and $\mathbb{e}_i$ the $i$-th standard basis vector. Here 
$\gamma > 0$ is the learning rate parameter, $\Delta$ is the momentum parameter, and the function $f_i$ is the $i$-th element of the sum in \eqref{eq:lsq}.

When the stochastic gradient in \eqref{eq:oursgd} is replaced with the full-gradient $\nabla f(\xx)$, the resulting algorithm with learning rate and momentum optimally chosen yields the celebrated algorithm, heavy-ball momentum (a.k.a. Polyak momentum) \citep{Polyak1962Some}. The optimal learning rate and momentum parameters are explicitly given by
\begin{equation} \label{eq:polyak_para}
    \gamma = \frac{4}{(\sqrt{\sigma_{\max}^2} + \sqrt{\sigma_{\min}^2})^2} \quad \text{and} \quad \Delta = \left ( \frac{\sqrt{\sigma_{\max}^2} - \sqrt{\sigma_{\min}^2}}{\sqrt{\sigma_{\max}^2} + \sqrt{\sigma_{\min}^2} } \right )^2.
\end{equation}
 It is well-known that heavy-ball is an optimal algorithm on the least squares problem in that it converges linearly at a rate of $\mathcal{O}(1/\sqrt{\kappa})$ (see \citep{pedregosa2021residual}). 

 In this paper, we adhere whenever possible to the following notation. We denote vectors in lowercase boldface $(\xx)$ and matrices in upper boldface $(\AA)$. The entries of a vector (or matrix) are denoted by subscripts. Unless otherwise specified, the norm $\|\cdot\|_2$ is taken to be the standard Euclidean norm if it is applied to a vector and the operator 2-norm if it is applied to a matrix.

\subsection{Random least squares problem}\label{sec:random_lsp}
\begin{wrapfigure}[40]{r}{0.5\textwidth}
\includegraphics[width = .9\linewidth]{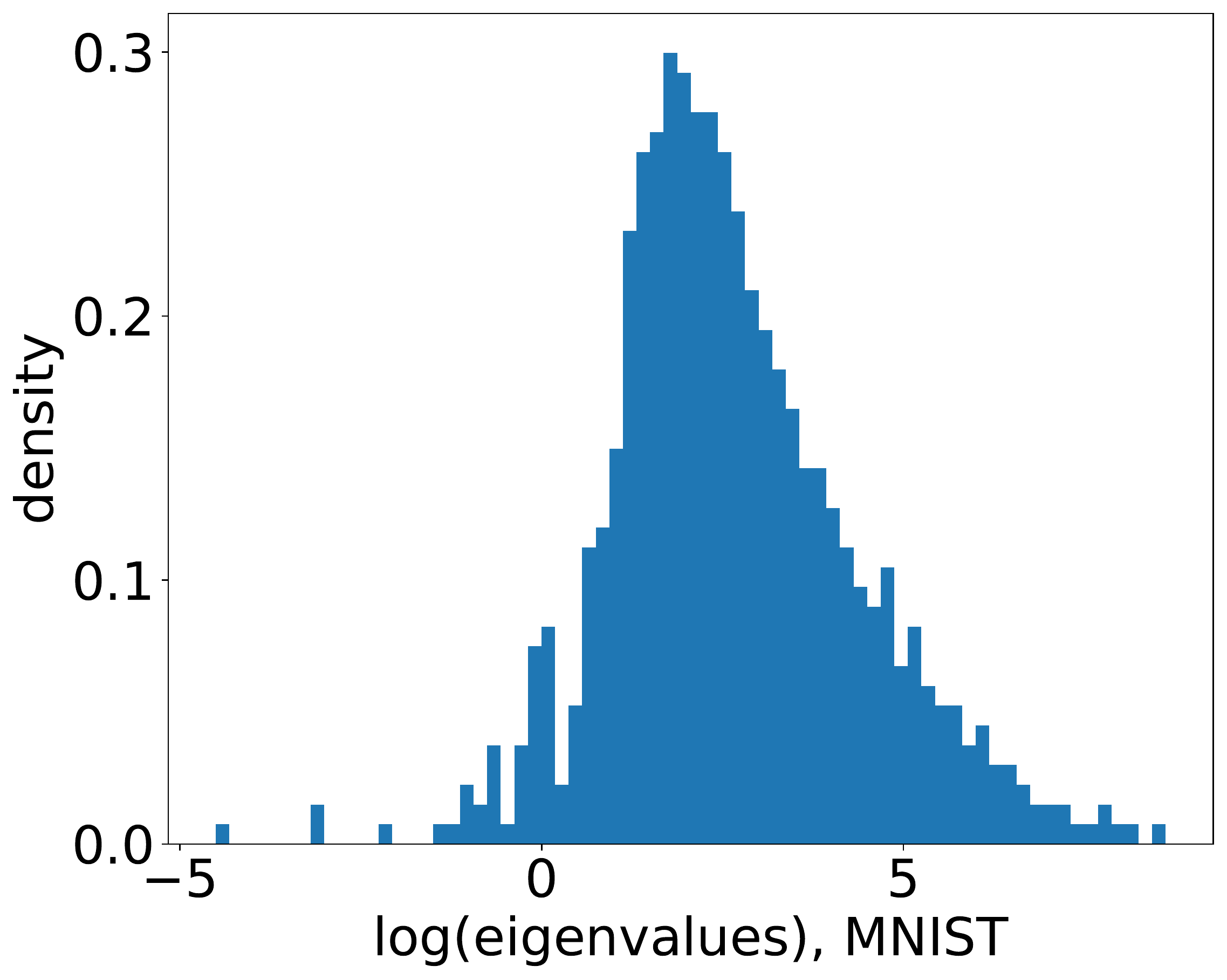}
\includegraphics[width = 1.0\linewidth]{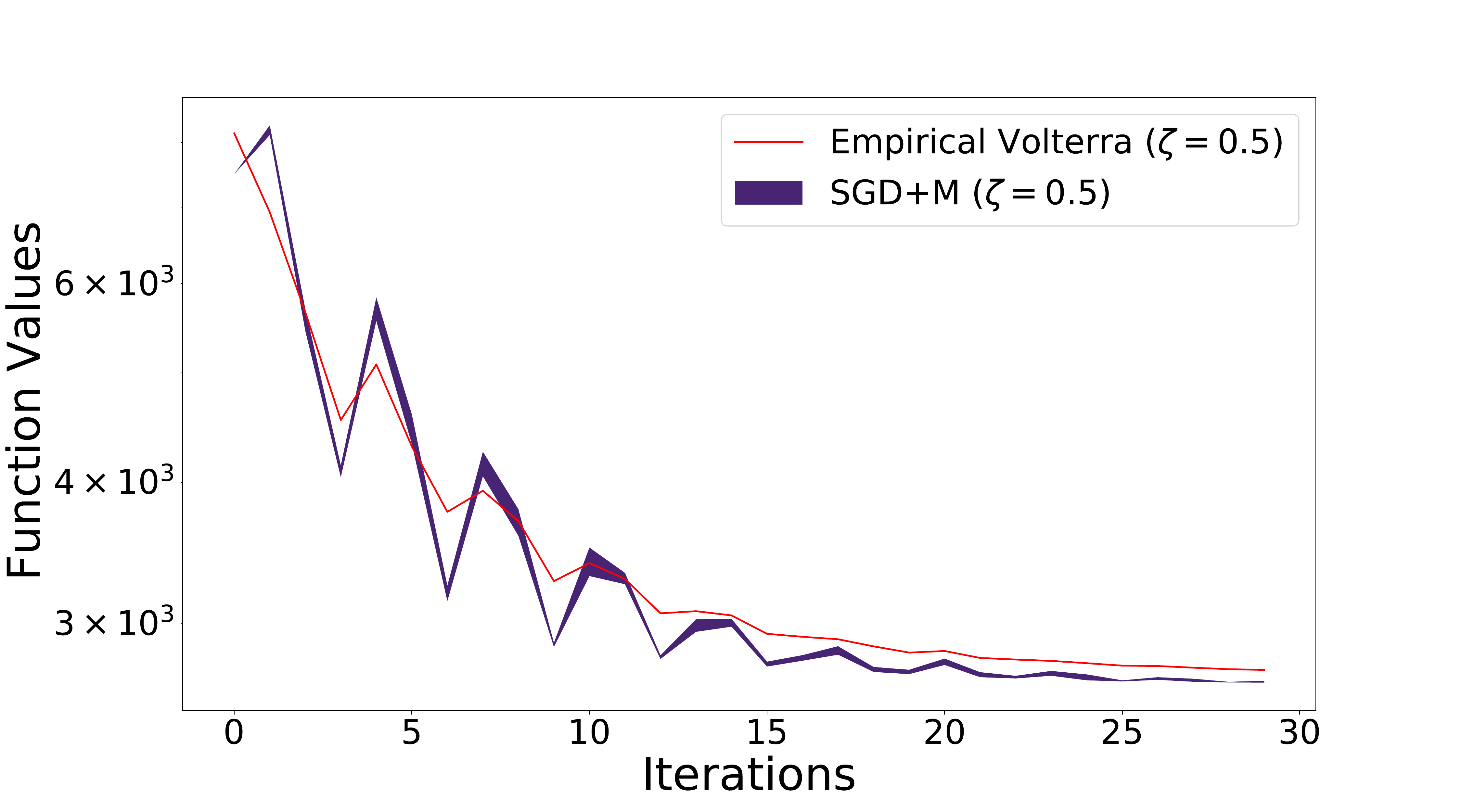}
\vspace{-0.5cm}
 \caption{{\bfseries SGD+M vs. Theory on even/odd MNIST.} 
MNIST ($60,000 \times 28 \times 28$ images) \citep{lecun2010mnist} is reshaped into a single matrix of dimension $60,000 \times 784$ (preconditioned to have centered rows of norm-1), representing 60,000 samples of 10 digits. The target $\bb$ satisfies $b_i = 0.5$ if the $i^{th}$ sample is an odd digit and $b_i = -0.5$ otherwise. SGD+M was run $10$ times with $(\Delta = 0.8, \gamma = 0.001, \zeta = 0.5)$ and the empirical Volterra was run once with $(R=11,000$, $\tilde{R} = 5300)$. The $10^{th}$ to $90^{th}$ percentile interval is displayed for the loss values of 10 runs of SGD+M. While MNIST data set does not satisfy our eigenvalue assumption on the data matrix, the solution to the Volterra equation on MNIST data set captures the dynamics of SGD+M. See App.~\ref{apx:app_numerical_simulations} for more details.}
\label{fig:mnist_motivation}
\end{wrapfigure}

To perform our analysis we make the following explicit assumptions on the signal $\widetilde{\xx}$, the noise $\eeta,$ and the data matrix $\AA.$
\begin{assumption}[Initialization, signal, and noise] \label{assumption: Vector} The initial vector $\xx_0 \in \RR^d$ is chosen so that $\xx_0 - \widetilde{\xx}$ is independent of the matrix $\AA$. The noise vector $\eeta \in \RR^n$ is centered and has i.i.d. entries, independent of $\AA$. The signal and noise are normalized so that 
\[
    \mathbb{E}\|\xx_0 - \widetilde{\xx}\|^2_2 = R\frac{d}{n},\quad \text{and}\quad \E[\|\eeta\|_2^2] = \widetilde{R}.
\]
\end{assumption}

Next we state assumptions on the data matrix $\AA$ as well as its eigenvalue and eigenvector distribution. Each row $\aa_i \in \mathbb{R}^{d\times 1}$ is centered and is normalized so that $\E \|\aa_i\|^2_2 =1$ for all $i$. We suggest as a central example the \textit{Gaussian random least squares} setup where each entry of $\AA$ is sampled independently from a standard normal distribution with variance $\frac{1}{d}$.


\begin{assumption}[Orthogonal invariance]\label{assumption: spectral_density} 
Let $\AA$ be a random $n \times d$ matrix. Suppose these random matrices satisfy a left orthogonal invariance condition: Let $\OO \in \mathbb{R}^{n \times n}$ be an orthogonal matrix. Then the matrix $\AA$ is orthogonally \textit{left invariant} in the sense that 
    \begin{equation} \label{eq:ortho_invariance}
    \OO \AA \law \AA.
    \end{equation}
\end{assumption}
This assumption implies that the left singular vectors of $A$ are uniformly distributed on the sphere which is the strongest form of eigenvector delocalization; many distributions of random matrices including some sparse ones (such as random regular graph adjacency matrices) are known to have some form of eigenvector delocalization. The classic example of a random matrix which has left orthogonal invariance is the sample covariance matrix, $\ZZ \sqrt{\SSigma},$ for an i.i.d. Gaussian matrix $\ZZ$ and any covariance matrix $\SSigma$.  Numerical simulations suggest that \eqref{eq:ortho_invariance} can be weakened in that the theory herein can be applied to other ensembles without this orthogonal invariance property. See Figure \ref{fig:mnist_motivation}.

\section{Deterministic Dynamical Equivalent of SGD+M} \label{sec:dynamics}

With these assumptions, we can give an explicit representation of the loss values on a least squares problem at the iterates generated by SGD+M algorithm. We show in this section (see Theorem~\ref{thm: concentration_main}): for any $T > 0$,
\[ 
\sup_{0 \le t \le T} | f(\xx_t) - \psi(t)| \to 0 \quad \text{in probability,}\]
where $\psi$ solves \eqref{eq:volterra_main}.
We begin by discussing the forcing and the noise terms of $\psi(t)$ and their relationship to SGD+M. 


\paragraph{Forcing term: problem instance information.} 
The forcing term represents the mean (with respect to expectation over the mini-batches) behavior of SGD+M and, in fact, it can be connected directly to the behaviour of heavy ball (Section~\ref{sec:motivation}).
For a \textit{small} learning rate $\gamma$, the forcing term $F(t)$ in \eqref{eq:volterra_main} governs the dynamics of $\psi(t)$. To analyze the forcing term $F(t)$, we need to solve a two step recurrence for the iterations of SGD+M given by $\eqref{eq:oursgd}$. Let $\ww_t \defas \AA\xx_t - \bb$ and $\tilde{\mathcal{X}}_{t,j} \defas (w_{t,j}^2\ w_{t-1,j}^2\ w_{t,j}w_{t-1,j})^T$. To this end, for each $j\in[n]$, we derive a matrix recurrence as follows:
\begin{equation}\label{eq:W_system}
\begin{gathered}
\tilde{\mathcal{X}}_{t+1,j} = \MM_j \tilde{\mathcal{X}}_{t,j} + (\text{Error})
\end{gathered}
\end{equation}
(See Appendix~\ref{subsec: f_dynamics} for more detail). Intuitively, the forcing term at iteration $t$ is given by applying a linear recurrence $\MM_j$ operator $t-1$ times on a vector containing initialization information at each $j\in[n]$ and then summed up for the first coordinate. Explicitly, the forcing term is the first coordinate of the quantity,
\begin{equation}\label{eq: forcingIntuition}
F(t) \approx \left( \sum_{j=1}^n \MM_j^{t-1}
\tilde{\mathcal{X}}_{1,j} \right)_1.
\end{equation}
By \eqref{eq: forcingIntuition}, it is clear that the \textit{maximum} of the eigenvalues of the operator $\MM_j$ is essential to analyze the convergence behavior of $F(t).$  Let 
$\lambda_{2,j}$ be the eigenvalue of $\MM_j$ with the biggest modulus and let
\begin{equation} \label{eq:lambda_2_max}
    \lambda_{2,\max} \defas \max_{j} |\lambda_{2,j}|.
\end{equation}
(See \eqref{eq:lambda_values} for an explicit formula of $\lambda_{2,j}$ and its maximum.) Further analysis (see Appendix~\ref{apx:proof_main_result}) shows that we can rewrite $\eqref{eq: forcingIntuition}$ as 
\begin{equation}\label{eq : forcingDecomposition}
\begin{gathered}
    F(t) = \frac{R}{2}h_1(t) + \frac{\widetilde{R}}{2}h_0(t),
\end{gathered}
\end{equation}
where $h_0, h_1$ are functions depending on the eigenvalues of $\AA\AA^T$ with decaying rate $\lambda_{2,\max}$. Therefore we can conclude that $F(t) = \mathcal{O}(\lambda_{2,\max}^t)$.

\paragraph{Kernel term: noise from the algorithm.} 
The convolution term in \eqref{eq:volterra_main} is due to the inherent stochasticity of SGD+M. More specifically, it is given by  
\begin{equation}\label{eq : kernelDecomposition}
\begin{gathered}
    \gamma^2\zeta(1-\zeta)\sum_{k=0}^t H_2(t-k) \psi(k),
\end{gathered}
\end{equation}
where $H_2$ is another function of eigenvalues of $\AA\AA^T$ with decaying rate $\lambda_{2,\max}$.  The presence of $\psi$ (training loss) is due to the fact that the noise generated by the $k$-th stochastic gradient is proportional to $\psi(k)$ (training loss), and the function $H_2(t-k)$ represents the progress of the algorithm in sending this extra noise to $0$. Observe \eqref{eq : kernelDecomposition} scales quadratically in the learning rate $\gamma.$ Hence for \textit{large} learning rates, \eqref{eq : kernelDecomposition} dominates the decay behaviour of $\psi$. Further details discussed in Section~\ref{subsec:complexity_analysis}.

We now state the main result:
\begin{theorem}[Concentration of SGD+M] \label{thm: concentration_main} Suppose Assumptions~\ref{assumption: Vector} and \ref{assumption: spectral_density} hold with the learning rate  $\gamma < \frac{1+\Delta}{\zeta \sigma_{max}^2}$ and the batch size satisfies $\beta/n = \zeta$ for some $\zeta >0$. Let the constant $T \in \mathbb{N}$. Then there exists $C >0$ such that for any $c >0$, there exists $D>0$ satisfying
\begin{equation} \label{eq:probability_convergence}
    \Pr \bigg [  \sup_{0 \le t \le T, t\in\mathbb{N}} |f(\xx_t)-\psi(t)| > n^{-C} \bigg ] \le Dn^{-c},
\end{equation}
for sufficiently large $n\in \mathbb{N}$. The function $\psi$ is the solution to the Volterra equation 
\begin{equation} \begin{gathered} \label{eq:Volterra_eq_main_1}
    \psi(t+1) = \underbrace{\frac{R}{2} h_1(t+1) + \frac{\widetilde{R}}{2} h_0(t+1) }_{\text{forcing}}+ \underbrace{\sum_{k=0}^t 
    \gamma^2\zeta(1-\zeta)H_2(t-k) \psi(k)}_{\text{noise}},\quad 
    \psi(0) = f(\xx_0).
\end{gathered}
\end{equation}

\end{theorem}
For a more accurate description on $h_0,h_1,H_2$ as well as the proof of Theorem~\ref{thm: concentration_main} and Corollary~\ref{cor: sgdnomomentum}, see Appendix~\ref{apx:proof_main_result}. The expression of $\psi$ highlights how the algorithm, learning rate, batch size, momentum, and noise levels interact with each other to produce different dynamics.
Note that the learning rate assumption will be necessary for the solution to the Volterra equation to be convergent, see Proposition \ref{prop:kernel_norm}. When $\Delta \to 0$, we obtain the Volterra equation for SGD with mini-batching.

\begin{corollary}[Concentration of SGD, no momentum]\label{cor: sgdnomomentum}
Under the same setting as Theorem \ref{thm: concentration_main} and when $\Delta = 0$, the function values $f(\xx_t)$ converge to $\psi(t)$ as in \eqref{eq:probability_convergence} where now the limit $\psi$ is a solution to the Volterra equation
\begin{equation} \begin{gathered} \label{eq:Volterra_eq_nomomentum}
    \psi(t+1) = \frac{R}{2} h_1(t+1) + \frac{\widetilde{R}}{2} h_0(t+1) + \sum_{k=0}^t 
    \gamma^2\zeta(1-\zeta)h_2(t-k) \psi(k).\\
\end{gathered}
\end{equation}
 where for $k=0,1,2$, 
\[
\begin{aligned}
    h_k(t) &= \frac{1}{n}\sum_{j=1}^n \sigma_j^{2k} (1-\gamma\zeta\sigma_j^2)^{2t}.
\end{aligned}
\]
\end{corollary}
\textbf{Remark.} Note that $H_2(t)$ reduces to $h_2(t)$ in $\Delta = 0$ case. Also when the limit $\zeta \to 0$ and when we scale time by $t/\zeta$, we have that $(1-\gamma\zeta\sigma_j^2)^{2t/\zeta} \to e^{-2\gamma\sigma_j^2t}$. This coincides with the result from \citep[Theorem 1]{paquette21a}. Indeed, this shows not only how our dynamics of SGD+M includes the no momentum case (i.e. SGD), but also how the dynamics of SGD+M differ from SGD.

\section{Convolution Volterra analysis} \label{sec:convolution_volterra_analysis}

In this section, we outline how to utilize the Volterra equation \eqref{eq:Volterra_eq_main_1} to a produce complexity analysis of SGD+M.  For additional details and proofs in this section, see Appendix \ref{apx:proof_main_results}. 

We begin by establishing sufficient conditions for the convergence of the solution to the Volterra equation \eqref{eq:Volterra_eq_main_1}.  Our Volterra equation can be seen as the \textit{renewal equation} (\citep{asmussen03}). Let us translate \eqref{eq:Volterra_eq_main_1} into the form of the renewal equation as follows:
\begin{equation}\label{eq: renewal_eq}
    \psi(t+1) = F(t+1) + (\K\ast\psi)(t),
\end{equation}
where $(f \ast g)(t)=\sum_{k=0}^\infty f(t-k)g(k)$.  Let the \textit{kernel norm} be $\|\K\| = \sum_{t=0}^\infty \K(t)$. By \citep[Proposition 7.4]{asmussen03}, we see that $\|\K\| < 1$ is necessary for our solution to the Volterra equation to be convergent. Indeed, we have the following result.

\begin{prop}\label{prop:psi_limit}
If the norm $\|\K\| < 1$, the algorithm is convergent in that 
\begin{equation}\label{eq:psi_limit}
\psi(\infty) \defas \lim_{t\to\infty} \psi(t) = 
\frac{\frac{\widetilde{R}}{2}(\max\{1-\tfrac{d}{n},0\})}{1-\|\K\|}.
\end{equation}
\end{prop}

Note that the noise factor $\widetilde{R}$ and the matrix dimension ratio $d/n$ appear in the limit.  Proposition \ref{prop:psi_limit} formulates the limit behaviour of the objective function in both the over-determined and the under-determined case of least squares. When under-determined, the ratio $d/n \ge 1$ and the limiting $\psi(\infty)$ is $0$; otherwise the limit loss value is strictly positive.  The result \eqref{eq:psi_limit} only makes sense when the noise term $\mathcal{K}$ satisfies $||\mathcal{K}||<1$; the next proposition illustrates the conditions on the learning rate and the trace of the eigenvalues of $\AA \AA^T$ such that the kernel norm is less than 1.
\begin{prop}[Convergence threshold] \label{prop:kernel_norm}
Under the learning rate condition $\gamma < \frac{1+\Delta}{\zeta \sigma_{max}^2}$ and trace condition $\frac{(1-\zeta)\gamma}{1-\Delta} \cdot \frac{1}{n} \tr(\AA \AA^T) < 1$, the kernel norm $\|\K\| < 1$ , i.e., $\sum_{t=0}^\infty \K(t) < 1$. 
\end{prop}
The \textit{learning rate condition} quantifies an upper bound of good learning rates by the largest eigenvalue of the covariance matrix $\sigma_{\max}^2$, batch size $\zeta$, and the momentum parameter $\Delta$. The \textit{trace condition} illustrates a constraint on the growth of $\sigma_{\max}^2.$ Moreover, for a full batch gradient descent model $(\zeta = 1)$, the trace condition can be dropped and we get the classical learning rate condition for gradient descent.

\begin{figure}
    \centering
    \includegraphics[width = 1.0\linewidth]{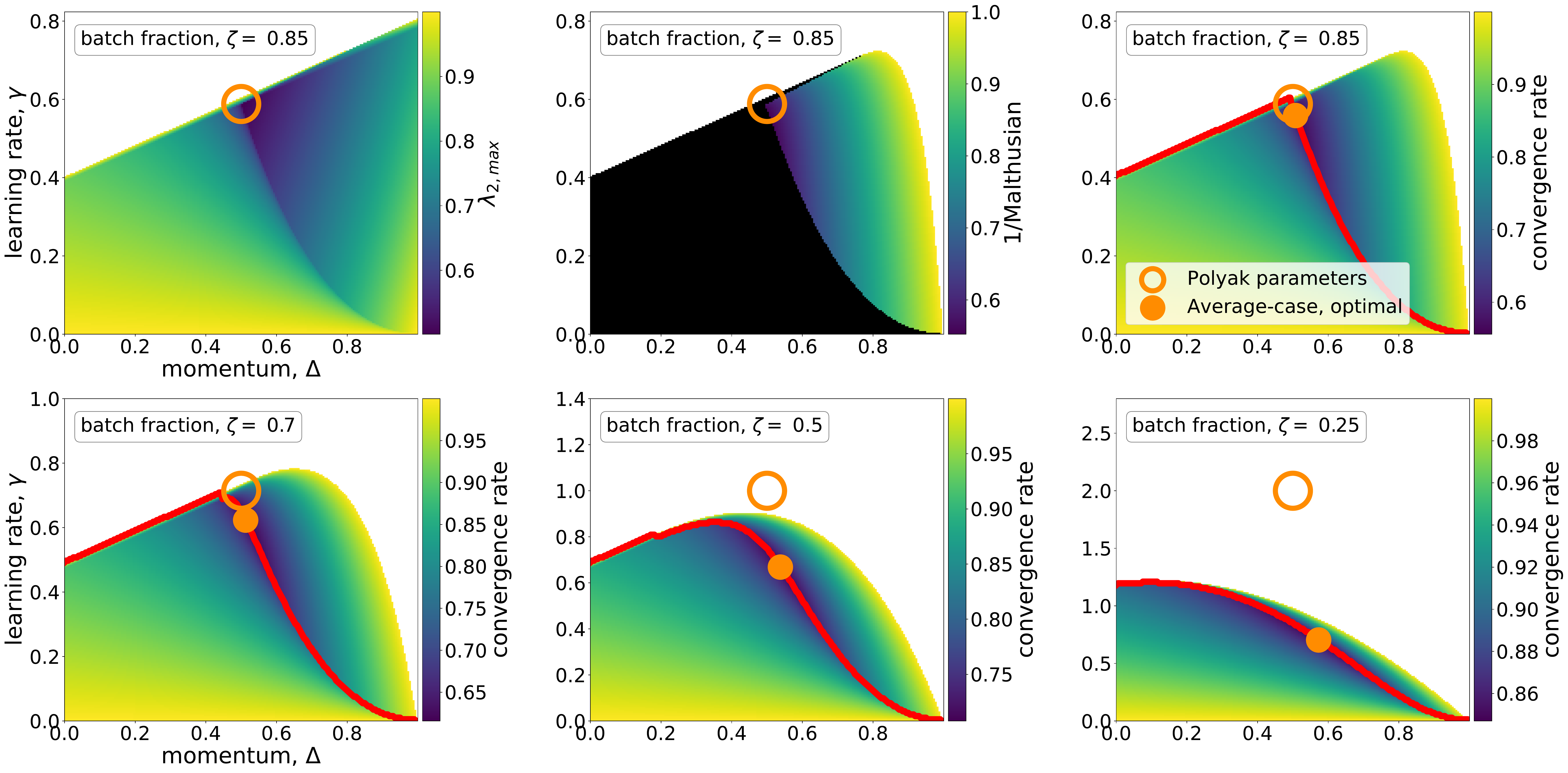}
    \caption{{\bfseries Different convergence rate regions: problem constrained regime versus algorithmically constrained regime } for Gaussian random least squares problem with $(n = 2000 \times d = 1000$). Plots are functions of momentum ($x$-axis) and learning rate ($y$-axis). Analytic expression for $\lambda_{2,\max}$ (see \eqref{eq:lambda_2_max}, \eqref{eq:lambda_values}) -- convergence rate of forcing term $F(t)$ -- given in (top row, column 1) represents the problem constrained region. (top row, column 2) plots 1/(Malthusian exponent) (\eqref{eq: malthusian}, for details see Appendix~\ref{apx:app_numerical_simulations}); black region is where the Malthusian exponent $\Xi$ does not exist. This represents the algorithmically constrained region. Finally, (top row, column 3 and bottom row) plots convergence rate of SGD+M $=\max \{\lambda_{2,\max}, \Xi^{-1}\}$, (see \eqref{eq:convergence_guarantee}), for various batch fractions. When the Malthusian exponent does not exist (black), $\lambda_{2, \max}$ takes over the convergence rate of SGD+M; otherwise the noise in the algorithm (i.e. Malthusian exponent $\Xi$) dominates. Optimal parameters that maximize $\lambda_{2,\max}$ denoted by Polyak parameters (orange circle, \eqref{eq:polyakMomentum}) and the optimal parameters for SGD+M (orange dot); below red line is the problem constrained region; otherwise the algorithmic constrained region. When batch fractions $\zeta = 0.85$ and $\zeta = 0.7$ (top row and bottom row, column 1) (i.e., large batch), the SGD+M convergence rate is the deterministic momentum rate of $1/\sqrt{\kappa}$. As the batch fraction decreases ($\zeta = 0.25$), the convergence rate becomes that of SGD and the optimal parameters of SGD+M and Polyak parameters are quite far from each other. The Malthusian exponent (algorithmically constrained region) starts to control the SGD+M rate as batch fraction $\to 0$.
}
    \label{fig:heatmap}
\end{figure}

\subsection{The Malthusian exponent and complexity}\label{subsec:complexity_analysis}
The rate of convergence of $\psi$ is essentially the worse of two terms -- the forcing term $F(t)$ and a discrete time convolution $\sum_{k=0}^t\psi(k)\mathcal{K}(t-k)$ which depends on the kernel $\mathcal{K}$. Intuitively, the forcing term captures the behavior of the expected value of SGD+M and the discrete time convolution captures the slowdown in training due to noise created by the algorithm.  Note that $F(t)$ is always a lower bound for $\psi(t)$, but it can be that $\psi(t)$ is exponentially (in $t$) larger than $F(t)$ owing to the convolution term.  This occurs when something called the \textit{Malthusian exponent}, denoted $\Xi$, of the convolution Volterra equation exists.   The Malthusian exponent $\Xi$ is given as the unique solution to
\begin{equation}\label{eq: malthusian}
\gamma^2\zeta(1-\zeta)\sum_{t=0}^\infty \Xi^t H_2(t) = 1, \qquad \text{if the solution exists.}
\end{equation}
The Malthusian exponent enters into the complexity analysis in the following way:
\begin{theorem}[Asymptotic rates]\label{thm : Malthusian}
The inverse of the Malthusian exponent always satisfies $\Xi^{-1} > \lambda_{2,\max}$ for finite $n$.  Moreover, for some $C > 0$, the convergence rate for SGD+M is
\begin{equation} \label{eq:convergence_guarantee}
    \psi(t) -\psi(\infty) \le C \max\{ \lambda_{2,\max}, \Xi^{-1} \}^t \quad \text{and} \lim_{t \to \infty} ( \psi(t) -\psi(\infty) )^{1/t} = \max\{ \lambda_{2,\max}, \Xi^{-1} \}.
\end{equation}
\end{theorem}
Thus to understand the rates of convergence, it is necessary to understand the Malthusian exponent as a function of $\gamma$ and $\Delta$. 

\subsection{Two regimes for the Malthusian exponent}
On the one hand, the Malthusian exponent $\Xi$ comes from the stochasticity of the algorithm itself. On the other hand, $\lambda_{2,\max}(\gamma, \Delta, \zeta)$ is determined completely by the problem instance information --- the eigenspectrum of $\AA\AA^T$. (Note we want to emphasize the dependence of $\lambda_{2, \max}$ on learning rate, momentum, and batch fraction.) Let $\sigma_{\max}^2$ and $\sigma_{\min}^2$ denote the maximum and minimum \textit{nonzero} eigenvalues of $\AA\AA^T$, respectively. For a fixed batch size, the optimal parameters $(\gamma_{\lambda}, \Delta_{\lambda})$ of  $\lambda_{2,\max}$ are 
\begin{equation}\label{eq:polyakMomentum}
    \gamma_{\lambda} = \frac{1}{\zeta}\bigg(\frac{2}{\sqrt{\sigma^2_{\max}} + \sqrt{\sigma^2_{\min}}}\bigg)^{2}
    \quad\text{and}\quad \Delta_{\lambda} = \bigg(\frac{\sqrt{\sigma^2_{\max}} - \sqrt{\sigma^2_{\min}}}{\sqrt{\sigma^2_{\max}} + \sqrt{\sigma^2_{\min}}}\bigg)^2.
\end{equation}
 In the full batch setting, i.e. $\zeta =1$, these optimal parameters $\gamma_{\lambda}$ and $\Delta_{\lambda}$ for $\lambda_{2,\max}$ are exactly the Polyak momentum parameters \eqref{eq:polyak_para}. Moreover, in this setting, there is no stochasticity so the Malthusian exponent disappears and the convergence rate \eqref{eq:convergence_guarantee} is $\lambda_{2,\max}$. We observe from \eqref{eq:polyakMomentum} that for all fixed batch sizes, the optimal momentum parameter, $\Delta_{\lambda}$, is independent of batch size.  The only dependence on batch size appears in the learning rate. At first it appears that for small batch fractions, one can take large learning rates, but in that case, the inverse of the Malthusian exponent $\Xi^{-1}$ dominates the convergence rate of SGD+M \eqref{eq:convergence_guarantee} and you cannot take $\gamma$ and $\Delta$ to be as in \eqref{eq:polyakMomentum} (See Figure \ref{fig:heatmap}).

We will define two subsets of parameter space,  the \textit{problem constrained regime} and the \textit{algorithmically constrained regime} (or stochastically constrained regime). The problem constrained regime is for some tolerance $\varepsilon >0$
\begin{equation}\label{eq:low_noise_regime} 
\text{problem constrained regime} \quad \defas \{(\gamma, \Delta) \, : 1- \sqrt{\Xi} < (1- \sqrt{\lambda_{2, \max}^{-1}} ) (1 - \varepsilon)\}.
\end{equation}
The remainder we call the \textit{algorithmically constrained} regime.  To explain the tolerance: for finite $n$, it transpires that we always have $\Xi^{-1} > \lambda_{2,\max}$, but it could be vanishingly close to $\lambda_{2,\max}$ as a function of $n$. Hence we introduce the tolerance to give the correct qualitative behavior in finite $n$.

\begin{prop}\label{prop:low_noise_regime} 
If the learning rate $\gamma \le \min ( \frac{1+\Delta}{\zeta\sigma_{max}^2}, \frac{(1-\sqrt{\Delta})^2}{\zeta \sigma_{\min}^2})$, with the trace condition \\$\frac{8(1-\zeta)\gamma}{1-\Delta} \cdot \frac{1}{n}\tr(\AA^T\AA) < 1$,
then $(\gamma, \Delta)$ is in the problem constrained regime with $\varepsilon = 1/2$. 
\end{prop}

Therefore by \eqref{eq:convergence_guarantee}, we have that 
\begin{equation}
    \psi(t) - \psi(\infty) \le D \left( \frac{4\lambda_{2,\max}}{(1+\sqrt{\lambda_{2,\max}})^2}\right)^t\quad\text{for some}\ D>0;
\end{equation}
we note that the expression in the parenthesis is $1-\tfrac12(1-\lambda_{2,\max}) + \mathcal{O}((1-\lambda_{2,\max})^2)$.

In the problem constrained regime, it is worthwhile to note that the overall convergence rate is the same as full batch momentum with adjusted learning rate, i.e., the batch size does not play an important role as long as we are in the problem constrained regime:
\begin{prop}[Concentration of SGD + M, full batch]\label{prop: concent_fullbatch}
  Suppose $\zeta=1$ and Assumptions~\ref{assumption: Vector} and \ref{assumption: spectral_density} hold with the learning rate  $\gamma < \frac{1+\Delta}{ \sigma_{max}^2}$.  If we let $\xx_t^{\text{full}}$ denote the iterates of full-batch gradient descent with momentum (GD+M), then
    \begin{equation} \label{eq:Volterra_eq_fullbatch}
  \sup_{0 \le t \le T}  \big | f \big ( \xx_{t}^{\text{full}} \big ) -  \psi_{\text{full}}(t) \big | \Prto[n] 0, \quad \text{where} \quad \psi_{\text{full}}(t+1) = \frac{R}{2} h_1(t+1) + \frac{\tilde{R}}{2} h_0(t+1).
\end{equation}
The functions $h_1$ and $h_0$ are defined in Theorem~\ref{thm: concentration_main} with $\zeta = 1$.
In particular, let $\gamma_{\text{full}}$ denote the learning rate for full batch GD+M, and $\gamma, \zeta<1$ for the learning rate and batch fraction in SGD+M with corresponding $\psi$ in Theorem \ref{thm: concentration_main}. Then when $\gamma_{\text{full}} = \gamma\zeta$ is satisfied, $\psi$ and $\psi_{\text{full}}$ share the same convergence rate in the problem constrained regime.
\end{prop}

\section{Performance of SGD+M: implicit conditioning ratio (ICR)} \label{sec:ICR}
Recall from \eqref{eq:condition_numbers} the definition of condition number, average condition number, and the implicit conditioning ratio
\begin{equation}
    \bar{\kappa} \defas \frac{\frac{1}{n}\sum_{j\in[n]} \sigma_j^2}{ \sigma_{\min}^2} <  \frac{\sigma_{\max}^2}{\sigma_{\min}^2} \defas \kappa \qquad \text{and} \qquad 
 \text{ICR} \defas \frac{\bar{\kappa}}{\sqrt{\kappa}}.
\end{equation}
Moreover recall that we refer to the \textit{large batch} regime where $\zeta \geq \text{ICR}$ and the \textit{small batch} regime where $\zeta \leq \text{ICR}$.

We begin by giving a rate guarantee that holds in the problem constrained regime, for a specific choice of $\gamma$ and $\Delta$.
\begin{prop}[Good momentum parameters]\label{prop:lambda_delta_choice} Suppose the learning rate and momentum satisfy
\begin{equation} \label{eq:good_parameters}
   \gamma = \frac{(1-\sqrt{\Delta})^2}{\zeta\sigma_{\min}^2} \,\,  \text{and} \,  \, \Delta = \max \left \{ \left (
\frac{1-\frac{\mathcal{C}}{\bar{\kappa}}}{1+\frac{\mathcal{C}}{\bar{\kappa}}} \right ),  \left ( \frac{1-\frac{1}{\sqrt{2 \kappa}}}{1+\frac{1}{\sqrt{2\kappa}}} \right )\right \}^2, \,  \text{where $\mathcal{C} \defas \zeta/(8(1-\zeta))$.}
\end{equation}
Then $\lambda_{2,\max} = \Delta$ and for some $C > 0$, the convergence rate for SGD+M is
\begin{equation} \label{eq:lambda_choice}
   \psi(t) - \psi(\infty) \le C \cdot \Delta^t = C \cdot \max \left \{
\left(\frac{1-\frac{\mathcal{C}}{\bar{\kappa}}}{1+\frac{\mathcal{C}}{\bar{\kappa}}}\right), \left( \frac{1-\frac{1}{\sqrt{2 \kappa}}}{1+\frac{1}{\sqrt{2\kappa}}} \right ) \right \}^{2t}.
\end{equation}
\end{prop}

\begin{remark}  We note that for all $\Delta$ satisfying $\frac{(1-\sqrt{\Delta})^2}{\zeta\sigma_{\min}^2} \le \frac{(1+\sqrt{\Delta})^2}{2\zeta\sigma_{\max}^2}$ with the learning rate $\gamma$ as in \eqref{eq:good_parameters}, we have that $\lambda_{2, \max} = \Delta$. By minimizing the $\Delta$ (i.e., by finding the fastest convergence rate), we get the formula for the momentum parameter in \eqref{eq:good_parameters}.
\end{remark}

The exact tradeoff in convergence rates \eqref{eq:lambda_choice} occurs when 
\begin{equation}\label{eq:zeta_threshold}
\frac{\mathcal{C}}{\bar{\kappa}} = \frac{1}{\sqrt{2 \kappa}}, \quad \text{or}\quad 
\zeta = 
\frac{\tfrac{8}{\sqrt{2}}\text{ICR}}
{
1+\tfrac{8}{\sqrt{2}}\text{ICR}
}.
\end{equation}
As $\zeta \leq 1,$ this condition is only nontrivial when $\text{ICR} \ll 1$, in which case $\zeta=\tfrac{8}{\sqrt{2}}\text{ICR}$, up to vanishing errors.

\paragraph{Large batch ($\zeta \geq \text{ICR}$).} In this regime SGD+M's performance matches the performance of the heavy-ball algorithm with the Polyak momentum parameters (up to absolute constants).  More specifically with the choices of $\gamma$ and $\Delta$ in Proposition \ref{prop:lambda_delta_choice}, the linear rate of convergence of SGD+M is $1-\frac{c}{\sqrt{\kappa}}$ for an absolute $c$.  Note that $\zeta$ does not appear in the rate, and in particular there is no gain in convergence rate by increasing the batch fraction.


\paragraph{Small batch ($\zeta \leq \text{ICR}$).}
In the small batch regime, the value of $\mathcal{C}$ is relatively small and the first term is dominant in \eqref{eq:lambda_choice}, and so the linear rate of convergence of SGD+M is $1-\frac{c\zeta}{{\overline{\kappa}}}$ for some absolute constant $c>0$.  In this regime, there is still benefit in increasing the batch fraction, and the rate increases linearly with the fraction.  We note that on expanding the choice of constants in small $\zeta$ the choices made in Proposition \ref{prop:lambda_delta_choice} are 
\[
\Delta \approx 1-\frac{\zeta}{8\overline{\kappa}}
\quad\text{and}
\quad 
\gamma \approx \frac{\zeta}{256 \overline{\kappa}^2\sigma_{\min}^2}.
\]
This rate can also achieved by taking $\Delta=0$, i.e.\ mini-batch SGD with no momentum.  Moreover, it is not possible to beat this by using momentum; we show the following lower bound:
\begin{prop}\label{prop:lb}
If $\zeta \leq \min\{\tfrac 12, \text{ICR}\}$ then
there is an absolute constant $C>0$ so that for convergent $(\gamma, \Delta)$ (those satisfying Proposition \ref{prop:kernel_norm}),
\(
\sqrt{\lambda_{2,\max}} \geq 1 - \tfrac{C\zeta}{\overline{\kappa}}.
\)
\end{prop}
This is a lower bound on the rate of convergence by Theorem \ref{thm : Malthusian}.


\section{Conclusion and future work}
 We have shown that the SGD+M method on a least squares problem demonstrates deterministic behavior in the large $n$ and $d$ limit. We described the dynamics of this algorithm through a discrete Volterra equation and for a fixed batch fraction. Moreover we characterized a dichotomy of convergence regimes depending on the learning rate and momentum parameters. Furthermore, we proved that SGD+M shows a distinguishable improvement over SGD in the large batch regime and we provided parameters which achieve acceleration. Our theory is also supported by numerical experiments on the isotropic features model and MNIST data set (see Appendix~\ref{apx:app_numerical_simulations} for details).
 
 While our analysis focuses on SGD+M algorithm applied to the least squares problems with orthogonal invariant data matrix, Figure \ref{fig:mnist_motivation} suggests that the Volterra prediction might hold in even greater generality. Removing these conditions, we leave as future work. Another direction of future work consists in finding the deterministic dynamics for generalization errors.

\bibliographystyle{plainnat}
\bibliography{reference}

\newpage
\appendix
\begin{center}
\LARGE{{Trajectory of Mini-Batch Momentum:}}\\
\Large{Batch Size Saturation and Convergence in High Dimensions}\\
\vspace{0.5em}\Large{Appendix}
\vspace{0.5em}
\end{center}
The appendix is organized into 4 sections as follows:
\begin{enumerate}
    \item Appendix \ref{apx:proof_main_result} derives the Volterra equation and proves the main concentration for the dynamics of SGD+M (Theorem \ref{thm: concentration_main}).
    \item We show in Appendix \ref{apx:error_bounds} that the error terms associated with concentration of measure on the high-dimensional orthogonal group disappear in the large-$n$ limit. 
    \item Appendix~\ref{apx:proof_main_results} derives main results including Proposition \ref{prop:low_noise_regime} and speed up of convergence rate of SGD+M (Proposition \ref{prop:lambda_delta_choice}) in the large batch regime, as well as the lower bound convergence rate in the small batch regime (Proposition \ref{prop:lb}).
    \item Appendix~\ref{apx:app_numerical_simulations} contains details on the numerical simulations.
\end{enumerate}


\paragraph{Notation.} In this paper, we adhere whenever possible to the following notation. We denote vectors in lowercase boldface $(\xx)$ and matrices in upper boldface $(\AA)$. The entries of a vector (or matrix) are denoted by subscripts. Unless otherwise specified, the norm $\|\cdot\|_2$ is taken to be the standard Euclidean norm if it is applied to a vector and the operator 2-norm if it is applied to a matrix. For a matrix $\AA$ and a vector $\bb$, we denote constants depending on $\AA$ and $\bb$, $C(\AA, \bb)$, as those bounded by an absolute constant multiplied by $\|\AA\|$ and $\|\bb\|$. We say an event $B$ holds \textit{with overwhelming probability} (w.o.p.) if, for every fixed $D > 0$, $\Pr(B) \ge 1 -C_D d^{-D}$ for some $C_D$ independent of $d$. Lastly, for $n\in\mathbb{N}$, $[n]$ denotes the set of natural numbers up to $n$, i.e., $[n] \defas \{1,2,\cdots,n\}$.

\section{Derivation of the dynamics of SGD+M} \label{apx:proof_main_result}
In this section, we establish the fundamental of the proof of Theorem \ref{thm: concentration_main}. Let us state the theorem in full detail first.

\begin{theorem}[Theorem \ref{thm: concentration_main}, detailed version]\label{thm:concentraion_main_detailed}
Suppose Assumptions~\ref{assumption: Vector} and \ref{assumption: spectral_density} hold with the learning rate  $\gamma < \frac{1+\Delta}{\zeta \sigma_{max}^2}$ and the batch size satisfies $\beta/n = \zeta$ for some $\zeta >0$. Let the constant $T \in \mathbb{N}$. Then there exists $C >0$ such that for any $c >0$, there exists $D>0$ satisfying
\begin{equation}
    \Pr \bigg [  \sup_{0 \le t \le T, t\in\mathbb{N}} |f(\xx_t)-\psi(t)| > n^{-C} \bigg ] \le Dn^{-c},
\end{equation}
for sufficiently large $n\in \mathbb{N}$. $\xx_t$ are the iterates of SGD+M and the function $\psi$ is the solution to the Volterra equation 
\begin{equation} \begin{gathered} \label{eq:Volterra_eq_main}
    \psi(t+1) = \underbrace{\frac{R}{2} h_1(t+1) + \frac{\tilde{R}}{2} h_0(t+1) }_{\text{forcing}}+ \underbrace{\sum_{k=0}^t 
    \gamma^2\zeta(1-\zeta)H_2(t-k) \psi(k)}_{\text{noise}},\ \text{and}\  
    \psi(0) = f(\xx_0),
\end{gathered}
\end{equation}
 where for $k=0,1$, 
\[
\begin{aligned}
    h_k(t) &= \frac{1}{n}\sum_{j=1}^n \frac{2(\sigma_j^2)^k}{\Omega_j^2-4\Delta} \left( -\Delta\gamma\zeta (\sigma_j^2) \cdot \Delta^{t} + \frac{1}{2}(\kappa_{2,j} - \Delta)^2\cdot (\lambda_{2,j})^{t} + \frac{1}{2}(\kappa_{3,j} - \Delta)^2\cdot (\lambda_{3,j})^{t}\right),
\end{aligned}
\]
and
\[
H_2(t) = \frac{1}{n}\sum_{j=1}^n \frac{2\sigma_j^4}{\Omega_j^2 -4\Delta} \Big(-\Delta^{t+1} + \frac{1}{2}\lambda_{2,j}^{t+1} + \frac{1}{2}\lambda_{3,j}^{t+1}\Big).
\]
Here $\Omega_j, \lambda_{2,j}, \lambda_{3,j}, \kappa_{2,j}, \kappa_{3,j}, j\in[n]$ are defined as
\begin{equation*}
\begin{aligned}
    \Omega_j &= 1-\gamma\zeta \sigma_j^2 + \Delta,\  \kappa_{2,j} = \frac{\lambda_{2,j}\Omega_j}{\lambda_{2,j} + \Delta},\  \kappa_{3,j} = \frac{\lambda_{3,j}\Omega_j}{\lambda_{3,j} + \Delta},\  \text{and}\\ 
    \lambda_{2,j} &= \frac{-2\Delta + \Omega_j^2 + \sqrt{\Omega_j^2(\Omega_j^2 - 4\Delta)}}{2},\  \lambda_{3,j} = \frac{-2\Delta + \Omega_j^2 - \sqrt{\Omega_j^2(\Omega_j^2 - 4\Delta)}}{2}.
\end{aligned}
\end{equation*}
\end{theorem}

\subsection{Change of basis}
Consider the singular value decomposition of $\AA = \UU \SSigma \VV^T$, where $\UU$ and $\VV$ are orthogonal matrices, i.e.\ $\VV \VV^T = \VV^T \VV = \II$ and $\SSigma$ is the $n \times d$ singular value matrix with diagonal entries $\diag(\sigma_j), j=1,\ldots,n$ (in the case $n>d$, we extend the set of singular values so that $\sigma_{d+1} = \cdots = \sigma_{n} = 0$).  We define the spectral weight vector $\nnu_k \defas \VV^T (\xx_k-\widetilde{\xx}),$ which therefore evolves like
\begin{equation}\label{eq:main_nuk}
    \nnu_{k+1} = \nnu_k 
    - \gamma \SSigma^T
    \UU^T \PP_k
    (\UU\SSigma \nnu_k - \eeta) + \Delta (\nnu_k - \nnu_{k-1}).
\end{equation}
Moreover, we can define
\begin{equation}\label{eq:w}
    \ww_{k}:=\SSigma\nnu_{k} - \UU^T\eeta ,
\end{equation}
so that
\begin{equation}\label{eq:f_sum}
f(\xx_t) = \frac{1}{2}\|\SSigma\nnu_t - \UU^T\eeta\|_2^2 = \frac{1}{2}\sum_{j=1}^n \ww_{t,j}^2. 
\end{equation}
Then \eqref{eq:main_nuk} can be translated as
\begin{equation}\label{eq:w_recurrence}
\ww_{k+1} = \ww_k - \gamma\SSigma\SSigma^T\UU^T\PP_k\UU\ww_k + \Delta (\ww_{k} - \ww_{k-1}).
\end{equation}

From this point, we focus on the evolution of $\ww$ rather than the iterates $\xx$.

\subsection{Evolution of $f$}\label{subsec: f_dynamics}

Now we would like to demonstrate the recurrence relation of $\ww_k$ and eventually that of $f(t)$, which will lead to a Volterra equation and error terms in a large scale. First, for $j \in [n]$ and $t\in \mathbb{N}$, \eqref{eq:w_recurrence} implies that
\begin{equation}\label{eq:w_recurrence_coord}
w_{t+1,j} = w_{t,j} - \gamma\sigma_j^2\sum_l w_{t,l} (\sum_{i  \in B_t} U_{ij}U_{il}) + \Delta(w_{t,j}-w_{t-1,j}),
\end{equation}
where $B_t = B$ denotes a randomly chosen mini-batch at the $t$-th iteration, whose size is given by $\beta \le n$. We interchangeably use the notation of $B_t$ and $B$, because it is independently chosen at each iteration. By taking squares on both sides, we have
\[
\begin{aligned}
w_{t+1,j}^2 &= \left( w_{t,j} - \gamma\sigma_j^2\sum_{l \in [n]} w_{t,l} (\sum_{i  \in B_t} U_{ij}U_{il}) + \Delta(w_{t,j}-w_{t-1,j}) \right)^2 \\
&= w_{t,j}^2 + \gamma^2\sigma_j^4 \big(\sum_{l\in [n]} w_{t,l} (\sum_{i\in B_t} U_{ij}U_{il} )\big)^2 -2\gamma\sigma_j^2 w_{t,j}\sum_{l\in [n]} w_{t,l} (\sum_{i\in B_t} U_{ij}U_{il} )\\
 &\quad + \Delta^2 (w_{t,j} - w_{t-1,j})^2
+ 2\Delta w_{t,j }(w_{t,j} - w_{t-1,j})\\
& \quad -2\gamma \sigma_j^2 \Delta \sum_{l\in [n]} w_{t,l} (\sum_{i\in B_t} U_{ij}U_{il})(w_{t,j} - w_{t-1,j}).
\end{aligned}
\]
Now let us denote the following error caused by mini-batching, i.e.,
\begin{equation}\label{eq:error_mini_batching}
\EE_B^{(l,j)} \defas \sum_{i\in B}U_{il}U_{ij}-\frac{\beta}{n}\delta_{l,j}.
\end{equation}
where $\delta_{l,j}$ is the Kronecker-delta symbol, meaning
\[
\text{For}\ l,j\in[n],\ \delta_{l,j} = 1\quad \text{if}\ l=j,\ \text{and}\ 0\ \text{otherwise}.
\]
Then the iteration on $w_{t+1}^2$ reduces to
\begin{align*}
    w_{t+1,j}^2 &= w_{t,j}^2 (1+\Delta^2 + 2\Delta) + w_{t-1,j}^2 \Delta^2 + w_{t,j}w_{t-1,j}(-2\Delta^2 - 2\Delta)\\
     &\qquad -2\gamma\sigma_j^2w_{t,j} \sum_{l\in [n]} w_{t,l} (\EE_B^{(l,j)} + \frac{\beta}{n}\delta_{l,j}) \\
    & \qquad - 2\gamma \sigma_j^2 \Delta \sum_{l\in [n]} w_{t,l} (w_{t,j}-w_{t-1,j}) (\EE_B^{(l,j)} + \frac{\beta}{n}\delta_{l,j})
    + \gamma^2\sigma_j^4 \big(\sum_{l\in [n]} w_{t,l} (\sum_{i\in B} U_{ij}U_{il} )\big)^2\\ 
    &= w_{t,j}^2 (1+\Delta^2 + 2\Delta -2\gamma\sigma_j^2\frac{\beta}{n} - 2\Delta\gamma\sigma_j^2\frac{\beta}{n}) + w_{t-1,j}^2 \Delta^2\\
    & \qquad + w_{t,j}w_{t-1,j}(-2\Delta^2-2\Delta+2\Delta\gamma\sigma_j^2\frac{\beta}{n})\\ 
    & \qquad +  \underbrace{\gamma^2\sigma_j^4\big(\sum_{l\in [n]} w_{t,l}(\sum_{i\in B} U_{ij}U_{il} )\big)^2}_{\defas \text{\textcircled{1}}}  
    +  \underbrace{\left(-2\gamma\sigma_j^2w_{t,j} \sum_{l\in [n]}\EE_B^{(l,j)} w_{t,l} \right)}_{\defas \EE_{B,1}^{(j)}(t)}\\
    & \qquad + \underbrace{\left( - 2\gamma \sigma_j^2 \Delta \sum_{l\in [n]} \EE_B^{(l,j)} w_{t,l} (w_{t,j}-w_{t-1,j}) \right)}_{ \defas \EE_{B,2}^{(j)}(t)}.
\end{align*}

When it comes to \textcircled{1}, we can decompose it into its expectation over the mini-batch $B$ and the error generated by it. By applying the technique from \citep[Lemma 8]{paquette21a}, we have
\begin{align*}
    \mathbb{E}[\text{\textcircled{1}}|\mathcal{F}_t] &= \gamma^2\sigma_j^4 \left[ \frac{\beta(\beta-1)}{n(n-1)}w_{t,j}^2 + \left( \frac{\beta}{n} - \frac{\beta(\beta-1)}{n(n-1)}\right) \sum_{i\in[n]} U_{ij}^2\left(\sum_{l\in[n]} U_{il}w_{t,l}\right)^2 \right]\\
    & = \Gamma_j^2 w_{t,j}^2 + \frac{(1 - \zeta)\gamma\sigma_j^2\Gamma_j}{n} \sum_{l\in[n]} w_{t,l}^2 + \EE_{KL}^{(j)}(t) + \EE_{beta}^{(j)}(t),
\end{align*}
where
\begin{gather*}
\Gamma_j \defas \gamma\zeta\sigma_j^2,\\
\EE_{beta}^{(j)}(t) \defas \gamma^2\sigma_j^4 \left[ \left( \frac{\beta(\beta-1)}{n(n-1)} - \zeta^2 \right) w_{t,j}^2 + \left( - \frac{\beta(\beta-1)}{n(n-1)} + \zeta^2 \right) \sum_{i\in[n]} U_{ij}^2\left(\sum_{l\in[n]} U_{il}w_{t,l}\right)^2\right],\ \\
\text{and} \quad \EE_{KL}^{(j)}(t) \defas \gamma^2\sigma_j^4(\zeta-\zeta^2) \sum_{i\in[n]} (U_{ij}^2 - \frac{1}{n} ) \left(\sum_l U_{il}w_{t,l}\right)^2.
\end{gather*}
Note that $\EE_{beta}^{(j)}(t)$ is generated by the error between $\beta(\beta-1)/(n(n-1))$ and $\zeta^2 = \beta^2/n^2$, whereas $\EE_{KL}^{(j)}(t)$ is generated by the replacement of $U_{ij}^2$ by $1/n$; In Appendix~\ref{sec:martingale_error}, we establish that this error can be bounded by the \textit{Key Lemma} (this is where the acronym ``$KL$'' comes from). Let $\EE_{B^2}^{(j)}(t) \defas \text{\textcircled{1}} - \mathbb{E}[\text{\textcircled{1}}|\mathcal{F}_t]$. 
Then observe
\begin{align*}
    \text{\textcircled{1}}
    &= \Gamma_j^2 w_{t,j}^2  + \frac{(1-\zeta)\gamma\sigma_j^2\Gamma_j}{n}  \sum_{l\in[n]} w_{t,l}^2  + \EE_{B^2}^{(j)}(t)+ \EE_{beta}^{(j)}(t) + \EE_{KL}^{(j)}(t).
\end{align*}
Therefore, we obtain
\begin{equation}\label{eq:w_recurrence_square}
\begin{aligned}
w_{t+1,j}^2 &= \Omega_j^2w_{t,j}^2 + \Delta^2w_{t-1,j}^2 -2\Delta\Omega_j w_{t,j}w_{t-1,j} + \frac{(1-\zeta)\gamma\sigma_j^2\Gamma_j}{n} \sum_{l\in[n]} w_{t,l}^2\\
& \qquad + \EE_{beta}^{(j)}(t) + \EE_{KL}^{(j)}(t) + \EE_{B}^{(j)}(t),
\end{aligned}
\end{equation}
where
\[
\EE_{B}^{(j)}(t) \defas \EE_{B^2}^{(j)}(t)+  \EE_{B,1}^{(j)}(t) + \EE_{B,2}^{(j)}(t).
\]

Similarly, we have
\begin{align}\label{eq:w_recurrence_cross}
\begin{split}
w_{t+1,j}w_{t,j}&= w_{t,j}(w_{t,j} - \gamma\sigma_j^2\sum_{l\in [n]} w_{t,l} (\sum_{i\in B_t} U_{ij}U_{il}) + \Delta(w_{t,j}-w_{t-1,j}))\\
&= w_{t,j}^2 - \gamma\sigma_j^2w_{t,j}\sum_{l \in [n]} w_{t,l}  (\EE_B^{(l,j)} + \frac{\beta}{n}\delta_{l,j}) + \Delta w_{t,j}(w_{t,j}-w_{t-1,j})\\
&= \Omega_j w_{t,j}^2 - \Delta w_{t,j}w_{t-1,j} 
- \underbrace{ \gamma\sigma_j^2w_{t,j}\sum_l \EE_{B}^{(l,j)} w_{t,l}}_{=\frac{1}{2} \EE_{B,1}^{(j)}(t)},
\end{split}
\end{align}
where $\Omega_j \defas 1 - \Gamma_j + \Delta$.

Therefore, \eqref{eq:w_recurrence_square} and \eqref{eq:w_recurrence_cross} imply
\begin{equation}\label{eq:system_iter_w}
\begin{pmatrix}
w_{t+1,j}^2 \\
w_{t,j}^2 \\
w_{t+1,j}w_{t,j}
\end{pmatrix} = \underbrace{\begin{pmatrix}
\Omega_j^2 &\Delta^2 &-2\Delta\Omega_j\\
1 & 0 & 0\\
\Omega_j & 0 & -\Delta 
\end{pmatrix}}_{\defas \MM_j}
\underbrace{\begin{pmatrix}
w_{t,j}^2 \\
w_{t-1,j}^2 \\
w_{t,j}w_{t-1,j}
\end{pmatrix}}_{\defas \tilde{\mathcal{X}}_{t,j}}+
\underbrace{\begin{pmatrix}
\tilde{N}_{t,j} + \EE_1^{(j)}(t)\\
0\\
\EE_2^{(j)}(t)
\end{pmatrix}}_{\defas \tilde{\mathcal{Y}}_{t,j}},
\end{equation}
where
\begin{gather*}
\tilde{N}_{t,j} \defas \frac{(1-\zeta)\gamma\sigma_j^2\Gamma_j}{n} \sum_l w_{t,l}^2 = \varphi_j^{(n)} \sum_l w_{t,l}^2,\  \text{with}\ \varphi_j^{(n)} \defas \frac{(1-\zeta)\gamma\sigma_j^2\Gamma_j}{n},\\
\EE_1^{(j)}(t) \defas \EE_{beta}^{(j)}(t) + \EE_{KL}^{(j)}(t) + \EE_{B}^{(j)}(t),\ \text{and}\\
\EE_2^{(j)}(t) \defas - \frac{1}{2} \EE_{B,1}^{(j)}(t).
\end{gather*}

Let us rewrite \eqref{eq:system_iter_w} as
\[
\begin{aligned}
\tilde{\mathcal{X}}_{t+1,j} &= \MM_j\tilde{\mathcal{X}}_{t,j} + \tilde{\mathcal{Y}}_{t,j}\\
&= \MM_j^2 \tilde{\X}_{t-1,j} + \MM_j\tilde{\Y}_{t-1,j} + \tilde{\Y}_{t,j}\\
&= \MM_j^t\tilde{\X}_{1,j} + \sum_{k=1}^t \MM_j^{t-k}\tilde{\Y}_{k,j}.
\end{aligned}
\]

The eigendecomposition of $\MM_j$ is given by $\MM_j = \XX_j\Lambda_j \XX_j^{-1}$, $\Lambda_j = \diag (\lambda_{1,j},\lambda_{2,j},\lambda_{3,j})$, $\XX_j = (x_{1,j}, x_{2,j}, x_{3,j})^T$ where 

\begin{equation}
\begin{gathered} \label{eq:lambda_values}
\lambda_{1,j} = \Delta, \lambda_{2,j} = \frac{-2\Delta + \Omega_j^2 + \sqrt{(\Omega_j^2)(\Omega_j^2-4\Delta)}}{2},
\lambda_{3,j} = \frac{-2\Delta + \Omega_j^2 - \sqrt{(\Omega_j^2)(\Omega_j^2-4\Delta)}}{2},\\
\text{and}\ \XX_j = \begin{pmatrix}
\Delta & \lambda_{2,j} & \lambda_{3,j}\\
1 & 1 & 1\\
\frac{\Omega_j}{2} & \kappa_{2,j} & \kappa_{3,j}
\end{pmatrix}\ \text{with}\ \kappa_{i,j} = \frac{\lambda_i \Omega_j}{\lambda_i + \Delta}.
\end{gathered}
\end{equation}

Also, its inverse, assuming $\det \XX_j \neq 0$, is given by 
\[
\XX_j^{-1} = \frac{2}{\Omega_j^2-4\Delta}\begin{pmatrix}
- 1 & - \Delta & \Omega_j\\
\frac{1}{2} & \frac{\lambda_{3,j}}{2} & -\kappa_{3,j}\\
\frac{1}{2} & \frac{\lambda_{2,j}}{2} & -\kappa_{2,j}
\end{pmatrix}.
\]

Note that for each $j\in [n]$ and $i\in \{2,3\}$, $\kappa_{i,j}$ satisfies 
\begin{itemize}
    \item $\kappa_{i,j}^2 = \lambda_{i,j}$ and $\kappa_{i,j} = \sqrt{\lambda_{i,j}}$ when $\Omega_j \ge 0$,
    \item $\kappa_{2,j} + \kappa_{3,j} = \Omega_j$, and
    \item $\kappa_{2,j}\kappa_{3,j} = \Delta$.
\end{itemize}
This implies that
\[
\begin{aligned}
\MM_j^{t-k}\tilde{\Y}_{k,j} &= \XX_j\Lambda_j^{t-k}\XX_j^{-1} \begin{pmatrix}
\tilde{N}_{k,j} + \EE_1^{(j)}(k)\\
0\\
\EE_2^{(j)}(k)
\end{pmatrix}\\
&= \XX_j \cdot \frac{2}{\Omega_j^2-4\Delta} \begin{pmatrix}
-\lambda_{1,j}^{t-k} \left(\tilde{\mathcal{N}}_{k,j} +\EE_1^{(j)}(k)\right) + \Omega_j \lambda_{1,j}^{t-k}\EE_2^{(j)}(k) \\
\frac{1}{2}\lambda_{2,j}^{t-k}  \left(\tilde{\mathcal{N}}_{k,j} +\EE_1^{(j)}(k)\right) -\kappa_{3,j} \lambda_{2,j}^{t-k}\EE_2^{(j)}(k) \\
\frac{1}{2}\lambda_{3,j}^{t-k}  \left(\tilde{\mathcal{N}}_{k,j} +\EE_1^{(j)}(k)\right) -\kappa_{2,j} \lambda_{3,j}^{t-k}\EE_2^{(j)}(k)
\end{pmatrix}.
\end{aligned}
\]

In particular, if we just focus on the (first coordinate of $\tilde{\mathcal{X}}_{t+1,j}) = w_{t+1,j}^2$, we have
\[
\begin{aligned}
w_{t+1,j}^2 &= ( \MM_j^t\tilde{\mathcal{X}}_{1,j})_1 + \frac{2}{\Omega_j^2-4\Delta} \sum_{k=1}^t (-\lambda_{1,j}\cdot \lambda_{1,j}^{t-k} + \frac{\lambda_{2,j}}{2}\cdot\lambda_{2,j}^{t-k} + \frac{\lambda_{3,j}}{2} \cdot\lambda_{3,j}^{t-k})\varphi_j^{(n)} \sum_{l\in[n]} w_{k,l}^2\\
&+ \frac{2}{\Omega_j^2-4\Delta} \sum_{k=1}^t (-\lambda_{1,j}\cdot\lambda_{1,j}^{t-k}+ \frac{\lambda_{2,j}}{2}\cdot \lambda_{2,j}^{t-k} + \frac{\lambda_{3,j}}{2}\cdot \lambda_{3,j}^{t-k}) \EE_1^{(j)}(k) \\
&+ \frac{2}{\Omega_j^2-4\Delta} \sum_{k=1}^t (\Omega_j \lambda_{1,j}\cdot\lambda_{1,j}^{t-k}-\kappa_{3,j}\lambda_{2,j}\cdot\lambda_{2,j}^{t-k} -\kappa_{2,j}\lambda_{3,j}\cdot \lambda_{3,j}^{t-k})\EE_2^{(j)}(k)
\end{aligned}
\]
( Here $(\cdot)_1$ denotes the first coordinate of a vector). Summing over $j \in[n]$ and dividing both sides by 2 gives
\begin{align}\label{eq:voltera_w}
\begin{split}   
    \frac{1}{2}&\sum_{j=1}^{n} w_{t+1,j}^2 = \frac{1}{2}\sum_{j=1}^{n} ( \MM_j^t\tilde{\mathcal{X}}_{1,j})_1\\
    & + \sum_{k=1}^t \left(\sum_{j=1}^{n} \frac{2\varphi_j^{(n)}}{\Omega_j^2-4\Delta} (-\Delta\cdot\lambda_{1,j}^{t-k} + \frac{\lambda_{2,j}}{2}\cdot\lambda_{2,j}^{t-k} + \frac{\lambda_{3,j}}{2}\cdot\lambda_{3,j}^{t-k}) \right) f(k)\\
    &+ \sum_{k=1}^t \left(\sum_{j=1}^{n} \frac{1}{\Omega_j^2-4\Delta} (-\Delta\cdot\lambda_{1,j}^{t-k} + \frac{\lambda_{2,j}}{2}\cdot\lambda_{2,j}^{t-k} + \frac{\lambda_{3,j}}{2}\cdot\lambda_{3,j}^{t-k})\EE_1^{(j)}(k) \right)\\
    &+ \sum_{k=1}^t \left(\sum_{j=1}^{n} \frac{1}{\Omega_j^2-4\Delta} (\Omega_j \lambda_{1,j}\cdot\lambda_{1,j}^{t-k}-\kappa_{3,j}\lambda_{2,j}\cdot\lambda_{2,j}^{t-k} -\kappa_{2,j}\lambda_{3,j}\cdot \lambda_{3,j}^{t-k})\EE_2^{(j)}(k) \right).
\end{split}    
\end{align}

Note that $\sum_{j=1}^{n} ( \MM_j^t\tilde{\mathcal{X}}_{1,j})_1$ describes the \textit{forcing} term (see Section \ref{sec:dynamics}). In order to analyze this term, observe
\[
\begin{aligned}
\tilde{\mathcal{X}}_{1,j} &= \begin{pmatrix}
w_{1,j}^2\\
w_{0,j}^2\\
w_{1,j}w_{0,j}
\end{pmatrix} = \begin{pmatrix}
(1-\Gamma_j)^2 w_{0,j}^2 + \varphi_j^{(n)} \sum_l w_{0,l}^2 + \EE_{beta}^{(j)}(0) + \EE_{KL}^{(j)}(0) + \EE_B^{(j)}(0)\\
w_{0,j}^2\\
(1-\Gamma_j)w_{0,j}^2 - \frac{1}{2}\EE_{B,1}^{(j)}(0)
\end{pmatrix}\\ 
&= 
\begin{pmatrix}
(1-\Gamma_j)^2 \Big(\frac{R}{n}\sigma_j^2 + \frac{\tilde{R}}{n} \Big) +2\varphi_j^{(n)} f(0)
+ \EE_{beta}^{(j)}(0) + \EE_{KL}^{(j)}(0) + \EE_B^{(j)}(0) + (1-\Gamma_j)^2\EE_{w_0}^{(j)}\\
\sigma_j^2\frac{R}{n} +\frac{\tilde{R}}{n} + \EE_{w_0}^{(j)}\\
(1-\Gamma_j)\Big( \frac{R}{n}\sigma_j^2 + \frac{\tilde{R}}{n} \Big) + (1-\Gamma_j)\EE_{w_0}^{(j)} - \frac{1}{2}\EE_{B,1}^{(j)}(0)
\end{pmatrix} \\
&= \Big(\frac{R}{n}\sigma_j^2 + \frac{\tilde{R}}{n}\Big) 
\begin{pmatrix}
(1-\Gamma_j)^2\\
1\\
(1-\Gamma_j)
\end{pmatrix} +
\begin{pmatrix}
2\varphi_j^{(n)} f(0) + \EE^{(j)}_1(0) \\
0\\
0
\end{pmatrix} + \EE_{w_0}^{(j)}
\begin{pmatrix}
(1-\Gamma_j)^2\\
1\\
(1-\Gamma_j)
\end{pmatrix} + \EE_2^{(j)}(0) \begin{pmatrix}
0\\
0\\
1
\end{pmatrix},
\end{aligned}
\]
where
\begin{equation}
\begin{gathered}
    \EE_{w_0}^{(j)} \defas w_{0,j}^2 - \mathbb{E}[w_{0,j}^2] = w_{0,j}^2 - \Big(\frac{R}{n}\sigma_j^2 + \frac{\tilde{R}}{n} \Big),\\
\EE^{(j)}_1(0) = \EE_{beta}^{(j)}(0) + \EE_{KL}^{(j)}(0) + \EE_B^{(j)}(0), \quad\text{and}\ \EE_2^{(j)}(0) = -\frac{1}{2}\EE_{B,1}^{(j)}(0).
\end{gathered}
\end{equation}

Therefore, by using the eigendecomposition of $\MM_j$ again, the first coordinate of $\MM_j^t\tilde{\mathcal{X}}_{1,j}$ is given by
\[
\begin{aligned}
&(\MM_j^t\tilde{\mathcal{X}}_{1,j})_1
= \Bigg[ \XX_j\Lambda_j^t \begin{pmatrix}
-(1-\Gamma_j)^2 - \Delta + \Omega_j(1-\Gamma_j)\\
\frac{1}{2}(1-\Gamma_j)^2 + \frac{\lambda_{3,j}}{2} - \kappa_{3,j}(1-\Gamma_j)\\
\frac{1}{2}(1-\Gamma_j)^2 + \frac{\lambda_{2,j}}{2} - \kappa_{2,j}(1-\Gamma_j)
\end{pmatrix}\cdot \frac{2}{\Omega_j^2-4\Delta} \Big(\frac{R}{n}\sigma_j^2 + \frac{\tilde{R}}{n} + \EE_{w_0}^{(j)}\Big)\\
& \qquad + \XX_j\Lambda_j^t \begin{pmatrix}
-1\\
1/2\\
1/2
\end{pmatrix} \cdot \frac{2}{\Omega_j^2-4\Delta} \Big( 2\varphi_j^{(n)} f(0) + \EE^{(j)}_1(0) \Big) + \XX_j\Lambda_j^t \begin{pmatrix}
\Omega_j\\
-\kappa_{3,j}\\
-\kappa_{2,j}
\end{pmatrix}\cdot \frac{2}{\Omega_j^2-4\Delta} \EE_2^{(j)}(0) \Bigg]_1 \\
&= \frac{2(\frac{R}{n}\sigma_j^2 + \frac{\tilde{R}}{n})}{\Omega_j^2-4\Delta} \left( -\Delta\Gamma_j\cdot \lambda_{1,j}^{t+1} + \frac{1}{2}(1-\Gamma_j - \kappa_{3,j})^2\cdot \lambda_{2,j}^{t+1} + \frac{1}{2}(1-\Gamma_j - \kappa_{2,j})^2\cdot \lambda_{3,j}^{t+1}\right)\\
& \qquad + \frac{2\EE_{w_0}^{(j)}}{\Omega_j^2-4\Delta} \left( -\Delta\Gamma_j\cdot \lambda_{1,j}^{t+1} + \frac{1}{2}(1-\Gamma_j - \kappa_{3,j})^2\cdot \lambda_{2,j}^{t+1} + \frac{1}{2}(1-\Gamma_j - \kappa_{2,j})^2\cdot \lambda_{3,j}^{t+1}\right)\\
& \qquad + \left(\frac{2\varphi_j^{(n)}}{\Omega_j^2-4\Delta} (-\lambda_{1,j}^{t+1} + \frac{1}{2}\cdot\lambda_{2,j}^{t+1} + \frac{1}{2}\cdot\lambda_{3,j}^{t+1}) \right) 2f(0)\\
& \qquad + \left(\frac{2}{\Omega_j^2-4\Delta} (-\lambda_{1,j}^{t+1} + \frac{1}{2}\cdot\lambda_{2,j}^{t+1} + \frac{1}{2}\cdot\lambda_{3,j}^{t+1}) \right) \EE^{(j)}_1(0)\\
&\qquad + \left( \frac{2}{\Omega_j^2 - 4\Delta} (\Omega_j\Delta\cdot\Delta^{t} -\kappa_{3,j}\lambda_{2,j}\cdot\lambda_{2,j}^{t} -\kappa_{2,j} \lambda_{3,j}\cdot\lambda_{3,j}^{t})\right)\EE_2^{(j)}(0)
.
\end{aligned}
\]
Simple algebra shows
$
1-\Gamma_j - \kappa_{3,j} = (\Omega_j - \Delta) - (\Omega_j - \kappa_{2,j}) = \Delta - \kappa_{2,j},
$ and similarly, $1-\Gamma_j - \kappa_{2,j} = \Delta - \kappa_{3,j}$. Hence, we conclude that
\[
f(t+1) = \frac{R}{2} h_1(t+1) + \frac{\tilde{R}}{2}h_0(t+1) + \sum_{k=0}^t  \gamma^2 \zeta(1-\zeta)H_2(t-k)f(k) + \EE(t).
\]
Here for $k=0,1$,
\[
\begin{aligned}
    h_k(t) = \frac{1}{n}\sum_{j=1}^n \frac{2(\sigma_j^2)^k}{\Omega_j^2-4\Delta} \left( -\Delta\Gamma_j\cdot \Delta^{t} + \frac{1}{2}(\Delta - \kappa_{2,j})^2\cdot \lambda_{2,j}^{t} + \frac{1}{2}(\Delta - \kappa_{3,j})^2\cdot \lambda_{3,j}^{t}\right),
\end{aligned}
\]
and
\[
H_2(t) = \frac{1}{n}\sum_{j=1}^n  \frac{2\sigma_j^4}{\Omega_j^2-4\Delta} \Big(-\lambda_{1,j}^{t+1} + \frac{1}{2}\lambda_{2,j}^{t+1} + \frac{1}{2}\lambda_{3,j}^{t+1}\Big).
\]
Also, the error term $\EE(t)$ is defined as
\begin{equation}\label{eq:errors}
\EE(t) \defas \EE_{IC}(t) + \EE_{beta}(t) + \EE_{KL}(t) + \EE_M(t),
\end{equation}
where
\[
\begin{aligned}
\EE_{IC}(t) &\defas \sum_{j\in[n]} \frac{1}{\Omega_j^2-4\Delta} \left( -\Delta\Gamma_j\cdot \Delta^{t+1} + \frac{1}{2}(\Delta - \kappa_{2,j})^2\cdot \lambda_{2,j}^{t+1} + \frac{1}{2}(\Delta - \kappa_{3,j})^2\cdot \lambda_{3,j}^{t+1}\right)\EE_{w_0}^{(j)},\\
\EE_{beta}(t) &\defas \sum_{k=0}^t \left(\sum_{j\in[n]} \frac{1}{\Omega_j^2-4\Delta} (-\Delta\cdot\Delta^{t-k} + \frac{\lambda_{2,j}}{2}\cdot\lambda_{2,j}^{t-k} + \frac{\lambda_{3,j}}{2}\cdot\lambda_{3,j}^{t-k})\EE_{beta}^{(j)}(k)
\right),\\
\EE_{KL}(t) &\defas\sum_{k=0}^t \left(\sum_{j\in[n]} \frac{1}{\Omega_j^2-4\Delta} (-\Delta\cdot\Delta^{t-k} + \frac{\lambda_{2,j}}{2}\cdot\lambda_{2,j}^{t-k} + \frac{\lambda_{3,j}}{2}\cdot\lambda_{3,j}^{t-k})\EE_{KL}^{(j)}(k)
\right),\ \text{and}\\
\EE_{M}(t) &\defas \sum_{k=0}^t \left(\sum_{j\in[n]} \frac{1}{\Omega_j^2-4\Delta} (-\Delta\cdot\Delta^{t-k} + \frac{\lambda_{2,j}}{2}\cdot\lambda_{2,j}^{t-k} + \frac{\lambda_{3,j}}{2}\cdot\lambda_{3,j}^{t-k})\EE_{B}^{(j)}(k)
\right)\\
&+  \sum_{k=0}^t \left(\sum_{j\in [n]} \frac{1}{\Omega_j^2-4\Delta}(\Omega_j\Delta\cdot\Delta^{t-k} -\kappa_{3,j}\lambda_{2,j}\cdot\lambda_{2,j}^{t-k} -\kappa_{2,j} \lambda_{3,j}\cdot\lambda_{3,j}^{t-k}) \EE_2^{(j)}(k) \right).
\end{aligned}
\]
A few comments on the naming of errors: $IC$ in $\EE_{IC}(t)$ stands for \textit{initial condition}. This error is generated from the initial bias on $w_{0,j}^2$. On the other hand, $M$ in $\EE_M(t)$ stands for \textit{Martingale}; the error is an accumulation of martingales over each time iteration. We deal with these errors in detail in following sections. And note that Theorem \ref{thm:concentraion_main_detailed} can be proved once we control the error $\EE(t)$ \textit{with overwhelming probability}. 

\section{Concentration of measure for the high--dimensional orthogonal group} \label{sec:martingale_error}
In this section, we give a high-level overview of the errors and how to bound them with overwhelming probability. Recall that we have the following error pieces:
\begin{equation}
\EE(t) \defas \EE_{IC}(t) + \EE_{beta}(t) + \EE_{KL}(t) + \EE_M(t).
\end{equation}

In order to bound the errors, we follow the methods that are used in \citep{paquette21a}: we would like to make an \textit{a priori} estimate that shows the function values remain bounded. Thus, we define the stopping time, for any fixed $\theta >0$ and \textit{large enough} $n\in\mathbb{N}$, by
\[
\vartheta \defas \inf \left\{ t\ge 0:\|\ww_t\|  (= \|\UU\SSigma\nnu_t - \eeta\|) > n^\theta \right\}.
\]
We then need to show:
\begin{lemma}\label{lemma: stoppingtime_norm}
    For any $\theta >0$, and for any $T>0$, $\vartheta >T$ with overwhelming probability.
\end{lemma}

\begin{proof}
From \eqref{eq:w_recurrence}, we have
\[
\ww_{k+1} = \left( (1+\Delta)\II_n - \gamma\SSigma\SSigma^T\UU^T\PP_k\UU \right)\ww_k - \Delta \ww_{k-1},
\]
where $\II_n$ denotes an identity matrix of dimension $n\times n$. Therefore, by taking norm on both sides and applying triangle inequality, we have
\[
\|\ww_{k+1}\| \le \left( 1+\Delta + \gamma\|\SSigma\|_2^2 \right) \|\ww_k\| + \Delta \|\ww_{k-1}\|.
\]
Let $ C:= 1+ 2\Delta + \gamma\|\SSigma\|_2^2$ and $\epsilon>0$ is small enough so that $ C^T \cdot n^\epsilon \le n^\theta$. By induction hypothesis, if we are given $\|\ww_l\| \le C^{l} n^\epsilon$ for $l= 0,\cdots,k < T$, we have
\[
\|\ww_{k+1}\| \le (1+ 2\Delta + \gamma\sigma_{\max}^2) C^{k} n^\epsilon \le C^{k+1} n^{\epsilon},
\]
and this finishes the proof once we check the initial conditions, i.e., $\|\ww_0\|, \|\ww_1\|$ are small enough with overwhelming probability. Observe, for any $\epsilon > 0$ and sufficiently large $n$,
\[
\begin{aligned}
\|\ww_0\|^2 = \sum_{j\in[n]} \Big(\sigma_j\nu_{0,j} - (\UU^T\eeta)_j\Big)^2 \le 2(\sigma_{\max}^2\|\nnu_0\|_2^2 + \|\eeta\|_2^2) = \mathcal{O}(1) \le n^\epsilon,
\end{aligned}
\]
w.o.p. by assumption \ref{assumption: Vector}. Similarly, $\ww_1$ is generated by the following formula
\[
\ww_1 = \left( \II_n - \gamma\SSigma\SSigma^T\UU^T\PP_k\UU \right)\ww_0,
\]
and applying norm on both sides gives
\[
\|\ww_1\| \le (1+ \gamma\sigma_{\max}^2) \|\ww_0\| \le (1+ \gamma\sigma_{\max}^2)n^\epsilon \le C n^\epsilon.
\] 
\end{proof}

We will need the result in what follows. Also, as an input, we work with the stopped process defined for any $t\ge 0$ by $\ww_t^\vartheta \defas \ww_{t\wedge\vartheta}$. Moreover, we condition on $\SSigma$ going forward. 

\label{apx:error_bounds}
\subsection{Control of the errors from the Initial Conditions}
In this section, we focus on controlling the errors generated by the initial conditions:
\[\EE_{IC}(t) = \sum_{j=1}^n \frac{1}{\Omega_j^2-4\Delta} \left( -\Delta\Gamma_j\cdot \lambda_{1,j}^{t+1} + \frac{1}{2}(\Delta - \kappa_{2,j})^2\cdot \lambda_{2,j}^{t+1} + \frac{1}{2}(\Delta - \kappa_{3,j})^2\cdot \lambda_{3,j}^{t+1}\right)\EE_{w_0}^{(j)},
\]
where
\[
  \EE_{w_0}^{(j)} =w_{0,j}^2 - \mathbb{E}[w_{0,j}^2] = w_{0,j}^2 - \Big(\frac{R}{n}\sigma_j^2 + \frac{\tilde{R}}{n} \Big).
\]
The next Proposition shows that the error $\EE_{IC}(t)$ can be bounded w.o.p.
\begin{prop}\label{prop:error_initial}
For any $T>0$ and for any $\epsilon >0$, with overwhelming probability,
\[
\max_{0\le t\le T}|\EE_{IC}(t)| \le n^{\epsilon-1/2}.
\]
\end{prop}

\begin{proof}
The proof is similar to that of \citep[Lemma 10]{paquette21a}. We rely on Chebyshev's inequality and the law of total probability to control the error. Fix $t\in[T]$ and let
\[
C^{(j)}(t) \defas \frac{1}{\Omega_j^2-4\Delta} \left( -\Delta\Gamma_j\cdot \lambda_{1,j}^{t+1} + \frac{1}{2}(\Delta - \kappa_{2,j})^2\cdot \lambda_{2,j}^{t+1} + \frac{1}{2}(\Delta - \kappa_{3,j})^2\cdot \lambda_{3,j}^{t+1}\right),
\]
and
\[
W(t) \defas \sum_{j=1}^n C^{(j)}(t) w_{0,j}^2,
\]
so that $\EE_{IC}(t) = W(t) - \mathbb{E}[W(t)]$. From \citep[Lemma 10]{paquette21a}, we know that the vector $\nnu_0^2$ follows the Dirichlet distribution (recall $\nnu_k = \VV^T (\xx_k-\widetilde{\xx})$), and in particular, $\mathbb{E}(\nu_{0,j}^4) \le \mathcal{O}(n^{-2})$ leads to $\mathbb{E}(w_{0,j}^4) \le \mathcal{O}(n^{-2})$ (also recall $\ww_k = \SSigma\nnu_k - \UU^T\eeta$, \eqref{eq:w}). Therefore, the (conditional) variance of $W(t)$ is bounded by
\[
\begin{aligned}
\text{Var}\Big[W(t)\Big] 
&= \mathbb{E}\left[  \left( \sum_{j=1}^n C^{(j)}(t) w_{0,j}^2 - \sum_{j=1}^n C^{(j)}(t) \Big(\frac{R}{n}\sigma_j^2 + \frac{\tilde{R}}{n} \Big) \right)^2 \right]\\
&= \mathbb{E}\left[ \left( \frac{1}{n}\sum_{j=1}^n C^{(j)}(t) \Big( n w_{0,j}^2 - \big(R\sigma_j^2 + \tilde{R} \big) \Big)\right)^2 \right]\\
&\le \frac{1}{n^2} \mathbb{E}\left[ \sum_{j=1}^n  \big(C^{(j)}(t)\big)^2 \Big( n w_{0,j}^2 - \big(R\sigma_j^2 + \tilde{R} \big) \Big)^2  \right]\\
&= \frac{1}{n} \left[ \frac{1}{n} \sum_{j=1}^n \big(C^{(j)}(t)\big)^2 \Big( n^2 \mathbb{E}[w_{0,j}^4] - \big(R\sigma_j^2 + \tilde{R} \big)^2 \Big) \right] =\mathcal{O}\Big(\frac{1}{n}\Big),
\end{aligned}
\]
where the Cauchy-Schwarz inequality was used in the second last line. Therefore, for $\epsilon >0$, Chebyshev inequality gives
\[
\text{Pr}\left[ \Bigg| \sum_{j=1}^n C^{(j)}(t) w_{0,j}^2 - \sum_{j=1}^n C^{(j)}(t) \Big(\frac{R}{n}\sigma_j^2 + \frac{\tilde{R}}{n} \Big) \Bigg| \ge n^{\epsilon-1/2} \right] \le \frac{1}{n^{2\epsilon-1}} \text{Var}\Big[W(t)\Big] \xrightarrow[n\to\infty]{} 0.
\]
Now applying the law of total probability (over $t=1,\cdots,T$) to this gives the claim.
\end{proof}

\subsection{Control of the beta errors}
In this section, we control the errors generated by the difference of $\frac{\beta(\beta-1)}{n(n-1)}$ and $\zeta^2 = (\frac{\beta}{n})^2$. For $t\in[T\wedge\vartheta]$, recall
\[
\EE_{beta}(t) = \sum_{k=0}^t \left(\sum_{j=1}^{n} \frac{1}{\Omega_j^2-4\Delta} (-\Delta\cdot\lambda_{1,j}^{t-k} + \frac{\lambda_{2,j}}{2}\cdot\lambda_{2,j}^{t-k} + \frac{\lambda_{3,j}}{2}\cdot\lambda_{3,j}^{t-k})\EE_{beta}^{(j)}(k)
\right),
\]
with
\[
\EE_{beta}^{(j)}(t) = \gamma^2\sigma_j^4 \left[ \left( \frac{\beta(\beta-1)}{n(n-1)} - \zeta^2 \right) w_{t,j}^2 + \left( - \frac{\beta(\beta-1)}{n(n-1)} + \zeta^2 \right) \sum_i U_{ij}^2\left(\sum_l U_{il}w_{t,l}\right)^2\right].
\]
First of all, note that
\[
\delta \defas \frac{\beta(\beta-1)}{n(n-1)} - \zeta^2 = \frac{\beta}{n}\cdot \frac{(\beta-1)n - \beta(n-1)}{n(n-1)} = \frac{\zeta(\zeta-1)}{n-1} = \mathcal{O}(n^{-1}).
\]
Then we can show the following:
\begin{prop}\label{prop:betaerror}
For any $T>0$ and for any $\epsilon >0$, with overwhelming probability,
\[
\max_{0\le t\le T\wedge\vartheta}|\EE_{beta}(t)| \le n^{\alpha-1/2},
\]
for some $1/4 > \alpha >\epsilon$.
\end{prop}

\begin{proof}
Let
\[
C^{(j)}(t, k) \defas  \frac{\gamma^2\sigma_j^4}{\Omega_j^2-4\Delta} (-\Delta\cdot\lambda_{1,j}^{t-k} + \frac{\lambda_{2,j}}{2}\cdot\lambda_{2,j}^{t-k} + \frac{\lambda_{3,j}}{2}\cdot\lambda_{3,j}^{t-k}).
\]
Then $C^{(j)}(t, k), j\in[n]$ are uniformly bounded by our assumptions, and we have
\[
\EE_{beta}(t) = \sum_{k=0}^t \sum_{j=1}^{n} C^{(j)}(t,k) \left[ \delta w_{t,j}^2 -\delta \sum_i U_{ij}^2\left(\sum_l U_{il}w_{t,l}\right)^2\right].
\]
Now Lemma \ref{lemma: stoppingtime_norm} (boundedness on the norm of $\ww_t$) and Lemma \ref{lemma:uwt} (uniform boundedness on the coordinates of $\UU\ww_t$) gives
\[
\EE_{beta}(t\wedge \vartheta) \le C\delta \big(\|\ww_t^\vartheta\|^2 + n\cdot \max_i (\UU\ww_t^\vartheta)_i^2 \big) = \mathcal{O}(n^{2\alpha -1}),
\]
for some $C>0$, which shows our claim.
\end{proof}

\subsection{Control of the Key lemma errors}

In this section, we show that $\EE_{KL}(t)$ can be bounded with overwhelming probability. The following Key Lemma from \citep[Lemma 14]{paquette21a} will be useful in the following:

\begin{lemma}[Key Lemma] \label{lemma:errorKH}
For any $T > 0$ and for any $\epsilon > 0,$ for some $\{C^{(j)}(t)\}, j\in[n], 0\le t\le T$ that are uniformly bounded, with overwhelming probability
\[
\max_{1 \leq i \leq n}
\max_{0 \leq t \leq T}
\biggl|
\sum_{j=1}^n 
    C^{(j)}(t)
    \biggl(\bigl(\mathbb{e}_j^T
    \UU^T 
    \mathbb{e}_{i}
    \bigr)^2-\frac{1}{n}\biggr)
\biggr|
\leq n^{\epsilon-1/2}.
\]
\end{lemma}

Given this lemma, combined with the Key Lemma, we can bound the error $\EE_{KL}(t)$ with overwhelming probability.
\begin{prop}\label{prop:error_KL}
For any $T>0$ and for any $\epsilon >0$, with overwhelming probability,
\[
\max_{0\le t\le T\wedge \vartheta}|\EE_{KL}(t)| \le n^{\epsilon-1/2}.
\]
\end{prop}
\begin{proof}
By definition, we have
\[
\EE_{KL}(t) =\sum_{k=0}^t \left(\sum_{j\in[n]} \frac{1}{\Omega_j^2-4\Delta} (-\lambda_{1,j}\cdot\lambda_{1,j}^{t-k} + \frac{\lambda_{2,j}}{2}\cdot\lambda_{2,j}^{t-k} + \frac{\lambda_{3,j}}{2}\cdot\lambda_{3,j}^{t-k})\EE_{KL}^{(j)}(k)
\right),
\]
with
\[
\EE_{KL}^{(j)}(t) = \gamma^2\sigma_j^4(\zeta-\zeta^2) \sum_{i\in[n]} (U_{ij}^2 - \frac{1}{n} ) \left(\sum_{l\in[n]} U_{il}w_{t,l}\right)^2.
\]
Thus for a sufficiently small $\tilde{\epsilon} >0$ and some $C>0$, and by applying Lemma \ref{lemma:errorKH} and Lemma \ref{lemma: stoppingtime_norm},
\[
\begin{aligned}
|\EE_{KL}^{(n)}(t\wedge \vartheta)| \le C\sum_{k=0}^t \sum_{i=1}^n \left(e_i^T\UU\ww_t^\vartheta\right)^2\cdot n^{\tilde{\epsilon}-1/2} \le CTn^{\tilde{\epsilon}-1/2}\cdot \|\ww_t^\vartheta\|^2 \le n^{\epsilon-1/2}.
\end{aligned}
\]
\end{proof}

\subsection{Control of the Martingale error}
In this section, we bound the error caused by Martingale terms. Recall that
\[
\begin{aligned}
\EE_{M}(t) &= \sum_{k=0}^t \left(\sum_{j\in[n]} \frac{1}{\Omega_j^2-4\Delta} (-\Delta\cdot\Delta^{t-k} + \frac{\lambda_{2,j}}{2}\cdot\lambda_{2,j}^{t-k} + \frac{\lambda_{3,j}}{2}\cdot\lambda_{3,j}^{t-k})\EE_{B}^{(j)}(k)
\right)\\
&+  \sum_{k=0}^t \left(\sum_{j\in [n]} \frac{1}{\Omega_j^2-4\Delta}(\Omega_j\Delta\cdot\Delta^{t-k} -\kappa_{3,j}\lambda_{2,j}\cdot\lambda_{2,j}^{t-k} -\kappa_{2,j} \lambda_{3,j}\cdot\lambda_{3,j}^{t-k}) \EE_2^{(j)}(k) \right),
\end{aligned}
\]
where
\[
\EE_{B}^{(j)}(t) = \EE_{B^2}^{(j)}(t)+  \EE_{B,1}^{(j)}(t) + \EE_{B,2}^{(j)}(t),
\]
with
\[
\begin{aligned}
\EE_{B}^{(j)}(t) &= \EE_{B^2}^j(t)+  \EE_{B,1}^{(j)}(t) + \EE_{B,2}^{(j)}(t),\ \text{and}\ \EE_2^{(j)}(t) = -\frac{1}{2} \EE_{B,1}^{(j)}(t),\ \text{with}\\
\EE_{B,1}^{(j)}(t) &= -2\gamma\sigma_j^2w_{t,j} \sum_{l\in [n]}\EE_B^{(l,j)} w_{t,l},\ \EE_B^{(l,j)} = \sum_{i\in B}U_{il}U_{ij}-\zeta\delta_{l,j},\\
\EE_{B,2}^{(j)}(t) &= - 2\gamma \sigma_j^2 \Delta \sum_{l\in [n]} \EE_B^{(l,j)} w_{t,l} (w_{t,j}-w_{t-1,j}),\  \text{and}\\
\EE_{B^2}^{(j)}(t) &= \underbrace{\gamma^2\sigma_j^4\big(\sum_{l\in [n]} w_{t,l}(\sum_{i\in B} U_{ij}U_{il} )\big)^2}_{\defas\text{\textcircled{1}}} - \mathbb{E}[\text{\textcircled{1}}]. 
\end{aligned}
\]

In view of the expression of $\EE_M(t)$, we define
\[
\EE_{B,1}(t) \defas \sum_{k=0}^t \left(\sum_{j\in[n]} \frac{1}{\Omega_j^2-4\Delta} (-\Delta\cdot\lambda_{1,j}^{t-k} + \frac{\lambda_{2,j}}{2}\cdot\lambda_{2,j}^{t-k} + \frac{\lambda_{3,j}}{2}\cdot\lambda_{3,j}^{t-k})\EE_{B,1}^{(j)}(k)
\right),
\]
and
\[
\EE_{B^2}(t) \defas \sum_{k=0}^t \left(\sum_{j\in[n]} \frac{1}{\Omega_j^2-4\Delta} (-\Delta\cdot\lambda_{1,j}^{t-k} + \frac{\lambda_{2,j}}{2}\cdot\lambda_{2,j}^{t-k} + \frac{\lambda_{3,j}}{2}\cdot\lambda_{3,j}^{t-k})\EE_{B^2}^{(j)}(k)
\right).
\]
Then it is easy to see that controlling these two terms will lead to the control of the entire Martingale error. Control of $\EE_{B,2}(t)$, which can be defined similarly to $\EE_{B,1}(t)$, can be done with exactly the same as that of $\EE_{B,1}(t)$. As for the second term of $\EE_M(t)$ which includes $\EE_2^{(j)}(t)$, our analysis will show that the coefficients won't play an important rule in the control of the error; so that term can be controlled for the same reason as $\EE_{B,1}(t)$.

We organize the proof as follows. First, we introduce a proposition from \citep{bardenet15} that gives an overwhelming probability concentration for sampling with replacement. Also, we claim that $\{\UU\ww_t\}, t\in[T\wedge\vartheta]$ is uniformly distributed with overwhelming probability over different coordinates. This lemma will lead to bounding the ``first-order'' error $\EE_{B,1}(t)$ (similarly for $\EE_{B,2}(t)$). As for bounding the ``second-order'' error $\EE_{B^2}(t)$, we will use the Hanson-Wright inequality for sampling without replacement \citep{adamczak15}. 


\subsubsection{Control of $\EE_{B,1}(t)$}
The Martinagle error originates from randomly sampling a mini-batch at every iteration. We begin by presenting the following Bernstein-type concentration result for sampling without replacement so that we see that randomness does not deviate too much from the ``expectation''.

\begin{prop}[Proposition 1.4, \citep{bardenet15}]\label{prop:concent}
Let $\mathcal{X} = (x_1, \cdots, x_n)$ be a finite population of $n$ points and $X_1, \cdots, X_\beta$ be a random sample drawn without replacement from $\mathcal{X}$. Let
\[
a = \min_{1\le i\le n} x_i\ \text{and}\ b=\max_{1\le i\le n} x_i.
\]
Also let
\[
\mu = \frac{1}{n}\sum_{i=1}^n x_i\ \text{and}\ \sigma^2 = \frac{1}{n}\sum_{i=1}^n (x_i-\mu)^2
\]
be the mean and variance of $\mathcal{X}$, respectively. Then for all $\epsilon >0$,
\[
\mathbb{P}\left(\frac{1}{\beta}\sum_{i=1}^\beta X_i - \mu \ge \epsilon \right) \le \exp \left(-\frac{\beta\epsilon^2}{2\sigma^2 + (2/3)(b-a)\epsilon} \right).
\]
\end{prop}

Now we can show that $\UU\ww_t$ is more or less uniformly distributed over coordinates.

\begin{lemma}\label{lemma:uwt}
$\max_{k}|(\UU\ww_t^\vartheta)_k| = \mathcal{O}(n^{\alpha-1/2})$ with overwhelming probability  for some $1/4 > \alpha > \epsilon$. 
\end{lemma}

\begin{proof}
We show a more general result, which is
\begin{equation}
\begin{gathered}\label{eq:maxbm}
MB^{(t)}\defas \max_{1\le k \le n}\max_{1\le m\le n}|B_{k,m}^{(t)}| = \mathcal{O}(n^{\alpha(t)-1/2})\ \text{w.o.p.},\\
\text{where}\ B_{k,m}^{(t)} \defas \sum_{j=1}^m U_{kj}w_{t,j}^\vartheta\ \text{and}\ 
1/4 > \alpha(T\wedge \vartheta) > \alpha((T\wedge\vartheta)-1)> \cdots > \alpha(0) > \epsilon.
\end{gathered}
\end{equation}
Note that $B_{k,n}^{(t)} = (\UU\ww_t^\vartheta)_k$, so $\max_{1\le k\le n}|(\UU\ww_t^\vartheta)_k| \le MB^{(t)}$. One approach is to apply the Proposition \ref{prop:concent} and the induction hypothesis. Note that the initial condition for the induction hypothesis will be treated later. From \eqref{eq:w_recurrence_coord}, we have
\[
w_{t+1,j}^\vartheta = w_{t,j}^\vartheta - \gamma\sigma_j^2\sum_{l\in[n]} w_{t,l}^\vartheta (\sum_{i  \in B_{t+1}} U_{ij}U_{il}) + \Delta(w_{t,j}^\vartheta-w_{t-1,j}^\vartheta).
\]
By multiplying $U_{kj}$ and summing over $j=1,\cdots,m,$ on both sides, we have
\begin{equation}\label{eq:B}
B_{k,m}^{(t+1)} = B_{k,m}^{(t)} - \gamma \underbrace{ \sum_{j=1}^m \sigma_j^2 U_{kj} \sum_{l\in[n]} w_{t,l}^\vartheta (\sum_{i  \in B_{t+1}} U_{ij}U_{il})}_{\defas\text{\textcircled{1}}} + \Delta(B_{k,m}^{(t)} - B_{k,m}^{(t-1)}).
\end{equation}
Let
\[
X_{i,k,m} \defas \sum_{j=1}^m \sigma_j^2 U_{kj}U_{ij} \sum_{l\in[n]} U_{il}w_{t,l}^\vartheta = (\UU\SSigma_m^2\UU^T)_{ik}(\UU\ww_t^\vartheta)_i,
\]
where $\SSigma_m = diag(\sigma_1^2,\cdots,\sigma_m^2,0,\cdots,0),$ so that \textcircled{1} $= \sum_{i\in B_{t+1}}X_{i,k,m}$. Note that we can assume that $k\notin B_{t+1}$ so that $k\neq i$, because we can deal with the term $X_{k,k,m} = (\UU\SSigma_m^2\UU^T)_{kk}(\UU\ww_t^\vartheta)_k$ separately. In order to use Proposition \ref{prop:concent}, we evaluate
\[
\begin{aligned}
\mu = \frac{1}{n}\sum_{i\in[n]} X_{i,k,m} 
= \frac{1}{n} \sum_{j=1}^m \sigma_j^2 U_{kj}\sum_{l\in[n]} w_{t,l}^\vartheta\delta_{j,l}
= \frac{1}{n} \sum_{j=1}^m \sigma_j^2 U_{kj}w_{t,j}^\vartheta,
\end{aligned}
\]
and
\[
\sigma^2 = \frac{1}{n}\sum_i (X_{i,k,m})^2 - \mu^2 = \frac{1}{n}\sum_i (\UU\SSigma_m^2\UU^T)_{ik}^2(\UU\ww_t^\vartheta)_i^2 - \mu^2.
\]
Now observe,
\begin{enumerate}
 \item As for $\mu$, by applying Abel's inequality,
\[
|\mu| \le \frac{1}{n}\sigma_{\max}^2\max_m|B_{k,m}^{(t)}| \le \frac{1}{n}\sigma_{\max}^2 MB^{(t)}.
\]
\item When it comes to controlling $\sigma^2$, by using Lemma \ref{lemma: stoppingtime_norm},
\[
\sigma^2 \le \frac{1}{n} \left( \max_{i\neq k} |(\UU\SSigma_m^2\UU^T)_{ik}|^2 \right) \|\UU\ww_t^\vartheta\|_2^2 + \mu^2 \le n^{-1+2\theta} \left( \max_{i\neq k} |(\UU\SSigma^2\UU^T)_{ik}|^2 \right).
\]
When $i\neq k$, by referring to \citep[Lemma 25]{paquette21a},
\[
|(\UU\SSigma^2 \UU^T)_{ik}| = | \sum_{j=1}^m \sigma_j^2 U_{ij}U_{kj}| = \mathcal{O}(n^{-1/2+\epsilon})\ \text{w.o.p.}
\]
Therefore, we have, with overwhelming probability, 
\[
\sigma^2 = \mathcal{O}(n^{-2+2\theta+2\epsilon}).
\]
\item Observe,
\[
\begin{aligned}
b &= \max_i X_{i,k,m} = \max_i (\UU\SSigma_m^2\UU^T)_{ik}(\UU\ww_t^\vartheta)_i \le \mathcal{O}(n^\epsilon) \max_i|(\UU\ww_t^\vartheta)_i| \le \mathcal{O}(n^\epsilon)\cdot MB^{(t)} \\
&= \mathcal{O}(n^{\epsilon + \alpha(t) -1/2})\ \text{w.o.p.},\ \text{and similar for}\ a = \min_{i} |X_{i,k,m}|.
\end{aligned}
\]
\end{enumerate}

Now applying Proposition \ref{prop:concent} gives 
\[
\mathbb{P}\left(\frac{1}{\beta}\sum_{i=1}^\beta X_{i,k,m} - \mu \ge t \right) \le \exp \left(-\frac{\beta t^2}{2\sigma^2 + (2/3)(b-a) t} \right),
\]
where the concentration with overwhelming probability is attained for $t = n^{-3/2+\alpha'(t)},\alpha'(t) > \alpha(t) > \theta + \epsilon$, and therefore
\[
\mathbb{P}\left(\sum_{i=1}^\beta X_{i,k,m} - \beta\mu \ge \tilde{\epsilon} \right) \to 0\ \text{as}\ n\to\infty\ \text{when}\ \tilde{\epsilon} = n^{-1/2+\alpha'(t)}.
\]
So applying this to \eqref{eq:B} gives
\[
B_{k,m}^{(t+1)} = B_{k,m}^{(t)} - \mathbbm{1}_{i=k}\cdot X_{k,k,m} - \Big( \beta\mu + \mathcal{O}(n^{-1/2+\alpha'(t)}) \Big) + \Delta(B_{k,m}^{(t)} - B_{k,m}^{(t-1)}),
\]
or
\[
\begin{aligned}
|B_{k,m}^{(t+1)}| &\le MB^{(t)} +  \sigma_{\max}^2MB^{(t)} + \Big( \frac{\beta}{n}\sigma_{\max}^2 MB^{(t)} + \mathcal{O}(n^{-1/2+\alpha'(t)}) \Big)\\
& \qquad + \Delta ( MB^{(t)} + MB^{(t-1)})\\
&\le C^{(k,m)}\mathcal{O}(n^{-1/2+\alpha'(t)}),
\end{aligned}
\]
for some $C^{(k,m)}>0$. Now taking maximum on $k$ and $m$ gives
\[
MB^{(t+1)} \le \Big(\max_{k,m} C^{(k,m)}\Big)\mathcal{O}(n^{-1/2+\alpha'(t)}) = \mathcal{O}(n^{-1/2+\alpha(t+1)})\ \text{w.o.p.},
\]
for some $\alpha(t+1) > \alpha'(t)$. Now once we show that the initial value $MB^{(0)}$ is small enough, by the induction hypothesis, we prove the theorem. Note that as $n\to \infty$, we can always make the increment $\alpha(t+1) - \alpha(t), t\in[T\wedge\vartheta-1]$ small enough so that $\alpha(T\wedge\vartheta) < 1/4$.

Now it suffices to check the initial condition, i.e., $MB^{(0)}$ is small enough:

\textbf{Claim.} $MB^{(0)} = \max_k\max_m|B_{k,m}^{(0)}| = \mathcal{O}(n^{\alpha(0)-1/2})\ \text{w.o.p.},\quad \alpha(0)>\theta + \epsilon$.

First note that $w_{0,j} = \sigma_j\nu_{0,j} - (\UU^T\eeta)_j, \nnu_t = \VV^T(\xx_t-\tilde{\xx})$. Therefore
\[
B_{k,m}^{(0)} = \sum_{j=1}^m U_{kj}(\sigma_j\nu_{0,j} - \sum_{l\in[n]} U_{lj}\eta_l) = \underbrace{\sum_{j=1}^m \sigma_jU_{kj}\nu_{0,j}}_{\defas\text{\textcircled{1}}} - \underbrace{\sum_{j=1}^m U_{kj} (\sum_{l\in[n]} U_{lj}\eta_l)}_{\defas\text{\textcircled{2}}}.
\]
We first show that $B_{k,m} = B_{k,m}^{(0)}$ for a fixed $k$ and $m$ attains the desired error order. As for \textcircled{1}, we show that $f_m(\UU_k) \defas$ \textcircled{1} is a Lipschitz function on $S^{n-1}$: observe, for $\UU_k, \UU_k' \in S^{n-1}$,
\[
\begin{aligned}
f_m(\UU_k) - f_m(\UU_k') &= \sum_{j=1}^m \sigma_j(U_{kj} - U_{kj}')\nu_{0,j}\\
&\le \sqrt{\sum_{j=1}^m \sigma_j^2\nu_{0,j}^2} \sqrt{\sum_{j=1}^m(U_{kj}-U_{kj}')^2} \le C\|\UU_k-\UU_k'\|_2,
\end{aligned}
\]
for some $C>0$. Therefore, the concentration result for Lipschitz function (\citep[Ex 5.1.12]{vershynin18}) gives
\[
\text{Pr} \{ |f_m(\UU_k) - \mathbb{E}f_m(\UU_k)| \ge t \} \le 2\exp(-cnt^2),
\]
and the overwhelming probability concentration is attained for $t = n^{-1/2+\epsilon}$, $\epsilon>0$.

As for \textcircled{2}, observe that
\[
\text{\textcircled{2}} = \sum_{j=1}^n g_j(t) (\aa^T \UU)_j (\bb^T \UU)_j,
\]
where $g_j(t) = 1$ for $1\le j \le m$ and 0 otherwise, $\aa = e_k$, and $\bb = \eeta$. Given $\eeta$ fixed, we have $\mathbb{E}_\eta[\text{\textcircled{2}}|\eeta] = \frac{m}{n}\eta_k$. Therefore, by \citep[Lemma 25]{paquette21a}, \textcircled{2} $= \frac{m}{n}\eta_k + \mathcal{O}(n^{\epsilon-1/2})$ w.o.p. As $\max_k |\eta_k| \le n^{\epsilon-1/2}$ w.o.p. ($f(x) = \max_i |x_i|,\ x\in S^{n-1}$ is a Lipschitz function on $S^{n-1}$ with Lipschitz constant 1), we conclude that \textcircled{2} $=\mathcal{O}(n^{\epsilon-1/2})$ w.o.p. Therefore $B_{k,m}^{(0)} = \mathcal{O}(n^{\alpha(0)-1/2})$ w.o.p. for arbitrarily small enough $\epsilon + \theta < \alpha(0) < 1/4$ and taking maximum over $k$ and $m$ shows our claim.
\end{proof}

Above lemma leads to the control of $\EE_{B,1}(t)$. Note that control of $\EE_{B,2}(t)$ can be done very similarly to $\EE_{B,1}(t)$.

\begin{prop}[Error bound for $\EE_{B,1}(t)$]\label{prop:error_EB1}
\[
\max_{0\le t\le T\wedge\vartheta} |\EE_{B,1}(t)| = \mathcal{O}(n^{1/2-\alpha'})\ \text{w.o.p.},
\]
where $1/2 > \alpha' > \alpha$, with $\alpha$ from Lemma \ref{lemma:uwt}. 
\end{prop}

\begin{proof}
Our strategy is to apply Proposition \ref{prop:concent} as well as Lemma \ref{lemma:uwt}. Recall that
\[
\begin{aligned}
\EE_{B,1}(t) &= \sum_{k=0}^t \left(\sum_{j\in[n]} \frac{1}{\Omega_j^2-4\Delta} (-\Delta\cdot\lambda_{1,j}^{t-k} + \frac{\lambda_{2,j}}{2}\cdot\lambda_{2,j}^{t-k} + \frac{\lambda_{3,j}}{2}\cdot\lambda_{3,j}^{t-k})\EE_{B,1}^{(j)}(k)
\right),\\
&= \sum_{k=0}^t \left(\sum_{j\in[n]} C^{(j)}(t,k) w_{t,j}\sum_{l\in[n]}\EE_B^{(l,j)}w_{t,l} \right),  
\end{aligned}
\]
where $C^{(j)}(t,k) \defas \frac{1}{\Omega_j^2-4\Delta} (-\Delta\cdot\lambda_{1,j}^{t-k} + \frac{\lambda_{2,j}}{2}\cdot\lambda_{2,j}^{t-k} + \frac{\lambda_{3,j}}{2}\cdot\lambda_{3,j}^{t-k}) \cdot (-2\gamma\sigma_j^2)$. Let us define
\[
X_i^{(t,k)} \defas \sum_{j\in[n]} C^{(j)}(t,k) U_{ij}w_{t,j}^\vartheta \sum_{l\in[n]} U_{il}w_{t,l}^\vartheta,\ \text{and}\ \mu_{(t,k)} = \frac{1}{n}\sum_{i\in[n]} X_i^{(t,k)} = \frac{1}{n}\sum_{j\in[n]} C^{(j)}(t,k) (w_{t,j}^\vartheta)^2,
\]
so that $\EE_{B,1}(t\wedge\vartheta) = \sum_{k=0}^t \left(\sum_{i\in B} X_i^{(t,k)} - \beta \mu_{(t,k)} \right)$. Let $\sigma_{(t,k)}^2$ be the variance of $X_i^{(t,k)}$:
\[
\sigma_{(t,k)}^2 \defas \frac{1}{n}\sum_{i\in[n]} \left( \sum_{j\in[n]} C^{(j)}(t,k) U_{ij}w_{t,j}^\vartheta \sum_{l} U_{il}w_{t,l}^\vartheta \right)^2 - \left(\frac{1}{n}\sum_{j\in[n]} C^{(j)}(t,k) (w_{t,j}^\vartheta)^2 \right)^2.
\]
In order to determine its order, note that
\[
\frac{1}{n}\sum_{i\in[n]} \left( \sum_{j\in[n]} C^{(j)}(t,k) U_{ij}w_{t,j}^\vartheta \sum_{l} U_{il}w_{t,l}^\vartheta \right)^2 = \frac{1}{n}\sum_{i\in[n]} (\UU\SSigma_{C}\ww_t^\vartheta)_i^2(\UU\ww_t^\vartheta)_i^2,
\]
where $\SSigma_{C} \defas \diag \{C^{(j)}(t,k)\}_{j\in[n]}$. By applying Lemma \ref{lemma:uwt}, we have with overwhelming probability
\[
\sigma^2 \le \max_i \|(\UU\ww_t^\vartheta)_i\|^2 \frac{1}{n}\|\UU\SSigma_{C}\ww_t^\vartheta\|_2^2 \le \|(\UU\ww_t^\vartheta)_i\|^2 \frac{1}{n} \|\SSigma_{C}\|_2^2 \|\ww_t^\vartheta\|_2^2  \le   \mathcal{O}(n^{2\alpha(t)+2\theta-1}). 
\]
Now Proposition \ref{prop:concent} gives
\[
\begin{aligned}
\text{Pr} \left(\frac{1}{\beta}\sum_{i=1}^\beta X_i^{(t,k)} - \mu_{(t,k)} \ge \tilde{\epsilon} \right) &\le \exp \left(-\frac{\beta\tilde{\epsilon}^2}{2\sigma_{(t,k)}^2 + (2/3)(b-a)\tilde{\epsilon}} \right),
\end{aligned}
\]
where, by using Chebyshev's inequality and applying Lemma \ref{lemma:uwt} again,
\[ 
\begin{aligned}
b &= \max_{1\le i\le n} \left( \sum_{j\in[n]} C^{(j)}(t,k) U_{ij}w_{t,j}^\vartheta \sum_{l} U_{il}w_{t,l}^\vartheta \right)\\ 
&\le
\max_i |(\UU\ww_t^\vartheta)_i|\sqrt{\sum_{j\in[n]} (C^{(j)}(t,k))^2(w_{t,j}^\vartheta)^2}\sqrt{\sum_{j\in[n]} U_{ij}^2 } = \mathcal{O}(n^{\alpha(t)+\theta-1/2})\ \text{w.o.p.}
\end{aligned}
\]
So by applying the same argument used in Proposition \ref{prop:concent}, and applying the union bound, we have
\[
\mathbb{P}\left(|\EE_{B,1}(t)| \ge \tilde{\epsilon} \right) \le T \cdot \mathbb{P}\left(\left| \sum_{i=1}^\beta X_i^{(t,k)} - \beta\mu_{(t,k)}\right| \ge c\tilde{\epsilon} \right)  \searrow 0\ \text{as}\ n\to\infty\ \text{when}\ \tilde{\epsilon} = n^{-1/2+\alpha'(t)},
\]
for $c = 1/t$ and any $1/2 > \alpha'(t) > \alpha(t) + \theta$. Note that $\theta$ can be taken as small as possible. Now taking maximum over $t, 0\le t\le T\wedge \vartheta, t\in\mathbb{N}$, gives the claim, with $\alpha' \defas \alpha'(T\wedge \vartheta)$.
\end{proof}

\subsubsection{Control of $\EE_{B^2}^{(j)}(t)$}
This section deals with controlling the error $\EE_{B^2}^{(j)}(t)$. Recall that
\[
\begin{aligned}
\EE_{B^2}(t) &= \sum_{k=0}^t \left(\sum_{j\in[n]} \frac{1}{\Omega_j^2-4\Delta} (-\lambda_{1,j}\cdot\lambda_{1,j}^{t-k} + \frac{\lambda_{2,j}}{2}\cdot\lambda_{2,j}^{t-k} + \frac{\lambda_{3,j}}{2}\cdot\lambda_{3,j}^{t-k})\EE_{B^2}^{(j)}(k)
\right)\\
&= \sum_{k=0}^t \left( \sum_{j\in[n]} C^{(j)}(t,k) \Bigg( \big(\sum_{l\in [n]} w_{t,l}(\sum_{i\in B_k} U_{ij}U_{il})\big)^2 - \mathbb{E}\Big[\big(\sum_{l\in [n]} w_{t,l}(\sum_{i\in B_k} U_{ij}U_{il})\big)^2 \Big| \mathcal{F}_k \Big]\Bigg)  \right),
\end{aligned}
\]
where $ C^{(j)}(t,k) \defas \frac{\gamma^2\sigma_j^4}{\Omega_j^2-4\Delta} (-\lambda_{1,j}\cdot\lambda_{1,j}^{t-k} + \frac{\lambda_{2,j}}{2}\cdot\lambda_{2,j}^{t-k} + \frac{\lambda_{3,j}}{2}\cdot\lambda_{3,j}^{t-k})$. Observe that the expression in the summand of $k$ can be translated as a \textit{quadratic form}:
\[
\begin{aligned}
\sum_{j\in[n]} C^{(j)}(t,k)\big(\sum_{l\in [n]} w_{k,l}(\sum_{i\in B_k} U_{ij}U_{il} )\big)^2 &= \sum_{j\in[n]} C^{(j)}(t,k) \left( e_j^T \UU^T\PP_k\UU\ww_k\right)^2\\
&= (\UU\ww_k)^T\PP_k \UU \SSigma_C \UU^T \PP_k (\UU\ww_k),
\end{aligned}
\]
where $\SSigma_C \defas \diag \{C^{(j)}(t,k)\}_{j\in[n]}$. Let $\XX_k \defas \PP_k (\UU\ww_k)$ and $\DD \defas \UU\SSigma_C \UU^T$. Note that, for a fixed time $t$ and $k$, and conditioned on $\UU$, $\DD$ is a fixed symmetric matrix and $\XX_k$ has a randomness only depending on $\PP_k$. Therefore, our error $\EE_{B^2}(t)$ can be expressed as
\begin{equation}\label{eq:error_squared}
\EE_{B^2}(t) = \sum_{k=0}^t \left(\XX_k^T\DD\XX_k - \mathbb{E}[\XX_k^T\DD\XX_k|\mathcal{F}_k] \right).
\end{equation}
As we did in the previous section, in view of union bounds, it suffices to impose bounds on each summand of \eqref{eq:error_squared} at $k=0,\cdots,t$. In order to have the \textit{Hanson-Wright} type concentration for our expression, we introduce the concept of \textit{Convex concentration property}.

\begin{definition}[Convex concentration property, \citep{adamczak15}]
Let $\XX$ be a random vector in $\mathbb{R}^n$. We will say that $\XX$ has the convex concentration property with constant $K$ if for every $1-$Lipschitz convex function $\varphi: \mathbb{R}^n \to \mathbb{R}$. we have $\mathbb{E}[\varphi(\XX)]<\infty$ and for every $t>0$,
\[
\text{Pr}(|\varphi(\XX)-\mathbb{E}\varphi(\XX)| \ge t) \le 2 \exp(-t^2/K^2).
\]
\end{definition}

\begin{remark}
 By a simple scaling, the previous remark can extend to $x_1,\cdots, x_n \in [a,b]$, in which case $K$ in the definition above will be replaced by $K(b-a)$.
\end{remark}

What is interesting for us is that vectors obtained via sampling without replacement follow the convex concentration property (\citep[Remark 2.3]{adamczak15}). More precisely, if $x_1,\cdots,x_n\in[0,1]$ and for $m\le n$ the random vector $\XX = (X_1,\cdots, X_m)$ is obtained by sampling without replacement $m$ numbers from the set $\{x_1, \cdots, x_n\}$, then $\XX$ satisfies the convex concentration property with an absolute constant $K$. In this sense, the following lemma (\citep[Theorem 2.5]{adamczak15}) will be useful to us.

\begin{lemma}[Hanson-Wright concentration for sampling without replacement]\label{lemma: convexHW}
Let $\XX$ be a mean zero random vector in $\mathbb{R}^n$. If $\XX$ has the convex concentration property with constant $K$, then for any $n\times n$ matrix $\AA$ and every $t>0$,
\[
\mathbb{P}(|\XX^T\AA\XX - \mathbb{E}\XX^T\AA\XX | \ge t) \le 2\exp \left( -\frac{1}{C}\min\left(\frac{t^2}{2K^4\|\AA\|_{HS}^2}, \frac{t}{K^2\|\AA\|} \right)\right),
\]
for some universal constant $C$.
\end{lemma}

\begin{remark}
The assumption that $\XX$ is centered is introduced just to simplify the statement of the theorem. Note that if $\XX$ has the convex concentration property with constant $K$, then so does $\tilde{\XX} = \XX - \mathbb{E}\XX$. Moreover, observe,
\[
\XX^T\AA\XX = (\tilde{\XX} + \E \XX)^T\AA(\tilde{\XX} + \E \XX) = \tilde{\XX}^T\AA\XX + \tilde{\XX}^T\AA(\E \XX) + (\E \XX)^T \AA \XX + (\E \XX)^T\AA(\E \XX),
\]
and this implies
\[
\begin{aligned}
\PP(|&\XX^T\AA\XX - \E \XX^T\AA\XX | \ge t) \\
&\le
\PP(|\tilde{\XX}^T\AA\tilde{\XX} - \E \XX^T\AA\XX | \ge t/3) + \PP(|\tilde{\XX}^TA(\E \XX) - \E \tilde{\XX}^T\AA(\E \XX) | \ge t/3) \\
&+ \PP(|(\E \XX)^T\AA\XX - \E (\E \XX)^T\AA\XX | \ge t/3) \\
&\le 2\exp \left( -\frac{1}{C}\min\left(\frac{t^2}{2\cdot 9K^4\|\AA\|_{HS}^2}, \frac{t}{3K^2\|\AA\|} \right)\right) + 2\cdot 2\exp\left(-\frac{t^2}{9K^2 \|\AA(\E \XX)\|_2^2}\right).
\end{aligned}
\]
\end{remark}

Finally, we can bound the error $\EE_{B^2}(t)$ using Lemma \ref{lemma: convexHW}.

\begin{prop}\label{prop:HWconcent}
For any $\epsilon>0$, we have
\[ \max_{0\le t\le T\wedge \vartheta}|\EE_{B^2}(t)| = \mathcal{O}(n^{-1/2+2\tilde{\alpha}})\ \text{w.o.p.},
\]
where $1/4 > \tilde{\alpha} > \alpha$, with $\alpha$ from Lemma \ref{lemma:uwt}. 
\end{prop}

\begin{proof}
Recall that
\[
\EE_{B^2}(t) = \sum_{k=0}^t \left(\XX_k^T\DD\XX_k - \mathbb{E}[\XX_k^T\DD\XX_k|\mathcal{F}_k] \right),
\]
and we apply Lemma \ref{lemma: convexHW} to each summand of $\EE_{B^2}(t\wedge\vartheta)$. More precisely,

\begin{itemize}
    \item $K$ is replaced by $K\cdot M_k$, where $M_k \defas \max_{l\in[n]} |(\UU\ww_k^\vartheta)_l| = \mathcal{O}(n^{\alpha(k)-1/2}),$ by Lemma \ref{lemma:uwt}.
    \item Observe that 
    \[
    \|\DD\|_{HS}^2 \le \|\SSigma_{C}\|_{HS}^2 = \mathcal{O}(n),
    \]
    and
    \[
     \|\DD\| = \|\SSigma_{C}\| = \mathcal{O}(1).
     \]
    \item $\E \XX_k = (\mu_1, \cdots, \mu_n)$ where $\mu_l = \frac{\beta}{n} (\UU\ww_k^\vartheta)_l, l\in[n]$, so that $\|\DD \E \XX\|_2 \le \|\DD\|_2 \|\E \XX\|_2 \le \mathcal{O}(n^{\theta})$.
\end{itemize}

Therefore, by using Lemma \ref{lemma: convexHW}, we have
\[
\begin{aligned}
\PP(|&\XX_k^T\DD\XX_k - \E \XX_k^T\DD\XX_k | \ge \tilde{\epsilon} | \mathcal{F}_k) \\
&\le 2\exp \left( -\frac{1}{C}\min\left(\frac{\tilde{\epsilon}^2}{2\cdot 9M_k^4K^4\|\DD\|_{HS}^2}, \frac{\tilde{\epsilon}}{3M_k^2K^2\|\DD\|} \right)\right)\\
& \quad + 2\cdot 2\exp\left(-\frac{\tilde{\epsilon}^2}{M_k^2K^2 \|\DD(\E \XX_k)\|_2^2}\right),
\end{aligned}
\]

and for $\tilde{\epsilon} = n^{2\tilde{\alpha}(k) -1/2}, 1/4> \tilde{\alpha}(k) >\alpha(k)$, we obtain the desired concentration result. Now taking union bound over $k=0,\cdots, T\wedge \vartheta$ gives the desired result, with $\tilde{\alpha} \defas \tilde{\alpha}(T\wedge \vartheta)$.
\end{proof}

\subsection{Proof of Theorem \ref{thm:concentraion_main_detailed}}
\begin{proof}[Proof of Theorem \ref{thm:concentraion_main_detailed}]
We have observed that Proposition \ref{prop:error_initial}, Proposition \ref{prop:betaerror}, Proposition \ref{prop:error_KL}, Proposition \ref{prop:error_EB1} and Proposition \ref{prop:HWconcent} imply that there exists $C>0$ such that for any $c>0$, there exists $D>0$ such that
\[
\text{Pr}\left[ \sup_{0\le t\le T\wedge \vartheta, t\in\mathbb{N}}|\EE(t)| > n^{-C} \right] < Dn^{-c}.
\]
Now combining this result with Lemma \ref{lemma: stoppingtime_norm} proves the Theorem.
\end{proof}

\section{Proof of Main Results}\label{apx:proof_main_results}
In this section, we prove various statements from Section \ref{sec:convolution_volterra_analysis}. First, we analyze assumptions on the learning rate $\gamma$ so that the kernel $\mathcal{K}$ is convergent (Proposition \ref{prop:kernel_norm}). Second, we define the Malthusian exponent and show under which conditions the convergence rate of our algorithm is determined by $\lambda_{2,\max}$ (Proposition \ref{prop:low_noise_regime}). Third, We find an optimal set of learning rate and momentum parameter so that the SGD+M outperforms SGD in the large batch regime (Proposition \ref{prop:lambda_delta_choice}). Lastly, we show the lower bound of the convergence rate of SGD+M in the small batch regime (Proposition \ref{prop:lb}).

\subsection{Learning rate assumption and kernel bound}
First, we show that the kernel $\K$ is always a nonnegative function, regardless of whether the eigenvalues $\{\lambda_{2,j}, \lambda_{3,j}\}, j\in[n]$ are real or complex values. 

\begin{lemma}[Positivity of the kernel] \label{lemma: kernel_positive}
The kernel function satisfies $\K(t) \ge 0$ for any $t\ge0$.
\end{lemma}
\begin{proof}
Fix $j\in[n]$ and let
\begin{equation*}
  H_{2,j}(t) \defas \frac{2\sigma_j^4}{\Omega_j^2-4\Delta} \Big(-\Delta\cdot\Delta^{t} + \frac{1}{2}\lambda_{2,j}\cdot\lambda_{2,j}^{t} + \frac{1}{2}\lambda_{3,j}\cdot\lambda_{3,j}^{t}\Big)
\end{equation*}
be the $j$-th summand of $H_2(t)$. We address two cases. In the first case, assume $\Omega_j^2 - 4\Delta \geq 0$. Then $\lambda_{2,j}$ and $\lambda_{3,j}$ are positive real numbers and one can easily verify that $\lambda_{2,j} \ge \Delta \ge \lambda_{3,j}$ and $\lambda_{2,j}\lambda_{3,j} = \Delta^2$. By the arithmetic-geometric inequality, we have
\begin{equation*}
    H_{2,j}(t) \geq \frac{2\sigma_j^4}{\Omega_j^2-4\Delta} \Big(-\Delta^{t+1} + \sqrt{\lambda_{2,j}^{t+1}\lambda_{3,j}^{t+1}}\Big) = \frac{2\sigma_j^4}{\Omega_j^2-4\Delta} \Big(-\Delta^{t+1} + \Delta^{t+1}\Big) = 0.
\end{equation*}

In the second case, we assume $\Omega_j^2 - 4\Delta < 0$. In this case, $\lambda_{2,j}$ and $\lambda_{3,j}$ are complex conjugates with magnitude $\Delta$, and therefore we have the relation 
\begin{equation*}
    \lambda_{2,j}^{t} = \Delta^t e^{i\theta_j t}, \quad\text{and}\ \lambda_{3,j}^t = \Delta^t e^{-i\theta_j t},
\end{equation*}
for some $\theta_j \in \mathbb{R}$. By Euler's formula, we obtain
\begin{equation*}
    -\Delta^{t+1} + \frac{1}{2} \Big(\lambda_{2,j}^{t+1} + \lambda_{3,j}^{t+1}\Big) = -\Delta^{t+1} + \Delta^{t+1} \cos(\theta_j t) \le 0.
\end{equation*}
and combined with the condition $\Omega_j^2-4\Delta <0$ gives $H_{2,j}(t) \ge 0$. Hence these two cases give the claim.
\end{proof}


The next proposition establishes that, under an upper bound on the learning rate, the maximum of the eigenvalues $\{\lambda_{2,j}\}$ for $j\in[n]$ has its magnitude less than one. Let $\lambda_{2,\max} \defas \max_j |\lambda_{2,j}|$. A simple computation shows that when $\lambda_{2,j}$ is complex then $|\lambda_{2,j}| = \Delta$.  
In particular, when all the eigenvalues $\lambda_{2,j}$ are complex numbers, $\lambda_{2,\max} = \Delta$. Otherwise, $\lambda_{2,\max} > \Delta$.
Recall again that $\sigma_{\max}^2$ and $\sigma_{\min}^2$ be the largest and smallest (nonzero) eigenvalue of $\AA\AA^T$, respectively.

\begin{prop}\label{prop:lambda_max}
If $\gamma < \frac{2(1+\Delta)}{\zeta \sigma_{max}^2}$ and $0\le\Delta<1$, then $ \lambda_{2,\max} < 1$.
\end{prop}
\begin{proof}
     First observe that
\begin{align*}
    \begin{split}
        \gamma < \frac{2(1+\Delta)}{\zeta \sigma_{\max}^2} 
        &\iff \Omega_{\min} \defas 1 - \gamma \zeta \sigma_{\max}^2 + \Delta > -1 - \Delta,
    \end{split}
\end{align*}
so we conclude $\Omega_j > -1 -\Delta$ for all $j\in [n]$. Note that $\Omega_j$ increases as $\sigma_j$ decreases. Fix $j \in [n]$. First, when $\Omega_{j}$ is non-positive, i.e. 
\begin{equation*}
    0 \geq \Omega_j > -1 -\Delta,
\end{equation*}
this implies $0 \leq \Omega_j < (1+\Delta)^2$. Second, let $\Omega_j \ge 0$. Then by the definition of $\Omega_j = 1-\gamma\zeta\sigma_j^2 + \Delta$, and as $\sigma_j^2 >0$, we have $\Omega_j \le 1 + \Delta$, or $\Omega_j^2 < (1 + \Delta)^2$. So in both cases, we have 
\begin{equation}\label{bound}
    \Omega_j^2 < (1+\Delta)^2.
\end{equation}
Then plugging in \eqref{bound} into the expression of $\lambda_{2,j}$ gives
\begin{align*}
    \begin{split}
        |\lambda_{2,j}| = \bigg|\frac{-2\Delta + \Omega_j^2 + \sqrt{\Omega_j^2(\Omega_j^2 - 4\Delta)}}{2} \bigg| &< \bigg|\frac{\Delta^2 + 1 \sqrt{(1+\Delta)^2(\Delta^2 - 2\Delta +1)}}{2} \bigg| \\
        &= \frac{\Delta^2 + 1 \sqrt{(1+\Delta)^2(\Delta - 1)^2}}{2} \\
        &= \frac{\Delta^2 + 1 + (1+\Delta)(1-\Delta)}{2} \\
        &= 1,
    \end{split}
\end{align*}
where the second last inequality comes from the constraint $0\le\Delta<1$. 
\end{proof}
Now we are ready to prove Proposition \ref{prop:kernel_norm}. 
\paragraph{Proof of Proposition \ref{prop:kernel_norm}}
\begin{proof}
    Note that $\gamma < \frac{1+\Delta}{\zeta \sigma_{\max}^2}$ implies not only $\lambda_{2,\max} < 1$ from Proposition \ref{prop:lambda_max}, but also $\Omega_j >0$ for all $j\in[n]$. 
    Let $\tilde{C}_j \defas \gamma^2\zeta(1-\zeta)\sigma_j^4/(\Omega_j^2-4\Delta)$ for the following. Using the the fact that $\lambda_{2,j}\lambda_{3,j} = \Delta^2$ and $\lambda_{2,j} + \lambda_{3,j} = -2\Delta + \Omega_j^2,$ we have
    \begin{align*}
    \begin{split}
        \sum_{t=0}^\infty \K(t) 
        &= \sum_{t=0}^\infty \frac{1}{n}\bigg(\sum_{j=1}^{n} \tilde{C}_j(-2\Delta\cdot \Delta^{t} + \lambda_{2,j} \cdot \lambda_{2,j}^{t} + \lambda_{3,j} \cdot \lambda_{3,j}^{t}) \bigg) \\
        &= \frac{1}{n}\sum_{j=1}^{n} \tilde{C}_j\bigg(-2 \frac{\Delta}{1-\Delta} + \frac{\lambda_{2,j}}{1-\lambda_{2,j}} + \frac{\lambda_{3,j}}{1-\lambda_{3,j}} \bigg) \\
        &= \frac{1}{n}\sum_{j=1}^{n}\tilde{C}_j \bigg( \frac{-2 \Delta}{1- \Delta} + \frac{-2\Delta + \Omega_j^2 - 2\Delta^2}{1 + 2\Delta - \Omega_j^2 + \Delta^2} \bigg) \\
        &= \frac{1}{n}\sum_{j=1}^{n}\frac{(1-\zeta)\zeta \gamma^2 \sigma_j^4}{\Omega_j^2 - 4\Delta}\cdot\frac{(1 + \Delta)(\Omega_j^2 - 4\Delta)}{(1-\Delta)(1+\Delta + \Omega_j)(1+\Delta - \Omega_j)}\\
        &= \frac{1}{n}\sum_{j=1}^{n}\frac{(1-\zeta) \gamma \sigma_j^2}{ \Omega_j^2 - 4\Delta}\cdot\frac{(1 + \Delta)(\Omega_j^2 - 4\Delta)}{(1-\Delta)(1+\Delta + \Omega_j)} \\
         &= \frac{1}{n}\sum_{j=1}^{n} \frac{(1-\zeta)\gamma \sigma_j^2 (1+\Delta)}{(1-\Delta)(1+\Delta + \Omega_j)} \\
         &\leq \frac{1}{n}\sum_{j=1}^{n} \frac{(1-\zeta)\gamma \sigma_j^2 }{1-\Delta} = \frac{(1-\zeta)\gamma}{1-\Delta} \cdot \frac{1}{n}\tr(\AA^T\AA) < 1, 
         \end{split}
         \end{align*}
    where $\Omega_j > 0$ was used in the last inequality.
\end{proof}

When the norm of the kernel is less than 1, we can specify the limit of the solution $\psi(t)$ to the Volterra equation when $t\to \infty$, as Proposition \ref{prop:psi_limit} states.

\paragraph{Proof of Proposition \ref{prop:psi_limit}}
\begin{proof}
    This is immediate from \cite[Proposition 7.4]{asmussen03}. In particular, from our expression of the renewal equation \eqref{eq: renewal_eq}, we have
    \[
    \psi(t) \to \frac{F(\infty)}{1-\|\K\|}\quad\text{as}\quad t\to\infty.
    \]
    Now the proof is done once we evaluate the limit of $F(t) = \frac{R}{2}h_1(t) + \frac{\tilde{R}}{2}h_0(t)$. Note that $\lim_{t\to\infty}h_1(t) = 0$. On the other hand, as for $h_0(t)$, if $n>d$, $\sigma_j = 0$ for $j = d+1,\cdots, n$. And for such $j$'s satisfying $\sigma_j = 0$, we can easily verify that $\lambda_{2,j} = 1, \lambda_{3,j} = \Delta^2, \Omega_j = 1+\Delta, \kappa_{2,j} = 1, \kappa_{3,j} = \Delta$. Therefore,
    \[
    \lim_{t\to\infty}h_0(t) = \lim_{t\to\infty} \left\{ \frac{1}{n}\sum_{j=d+1}^n \frac{2}{\Omega_j^2-4\Delta} \big( 0 + \frac{1}{2}(1 - \Delta)^2\cdot 1 +0 \big) \right\} = \frac{n-d}{n} = 1-r,
    \]
    and this proves the claim.
\end{proof}

\subsection{Malthusian exponent and convergence rate}
In this section, we show that the Malthusian exponent $\Xi$ is always smaller than $\lambda_{2,\max}^{-1}$ for a finite dimension $n$. Also, in the problem constrained regime we show that SGD+M shares the same convergence rate with full batch gradient descent with momentum with adjusted learning rate. 

\begin{prop}\label{prop: mal_bound_eigmax}
The Malthusian exponent defined in \eqref{eq: malthusian} satisfies
\[
\Xi < (\lambda_{2,\max})^{-1}
\]
when the dimension $n$ is finite.
\end{prop}

\begin{proof}
It suffices to observe that the convergence rate of $H_2(t)$ is determined by $\lambda_{2,\max}$; if all $\lambda_{2,j}, j\in[n]$, are real numbers, then we can easily show that $\lambda_{2,j} > \Delta > \lambda_{3,j}$. Therefore $\lambda_{2,\max}$ takes over the convergence rate of $H_2(t)$. If, for some $j \in [n]$, $\lambda_{2,j}$ and  $\lambda_{3,j}$ are both complex numbers, observe that $|\lambda_{2,j}| = |\lambda_{3,j}| = \Delta$. In that case, if we let $\lambda_{2,j} = \Delta \exp(i\theta_j)$ for some $\theta_j \in \mathbb{R}$, $\lambda_{3,j} = \Delta \exp(-i\theta_j)$ then
\[
\begin{aligned}
-\Delta^{t+1} + \frac{1}{2}\lambda_{2,j}^{t+1} + \frac{1}{2}\lambda_{3,j}^{t+1} &= -\Delta^{t+1} + \frac{1}{2}\Delta^{t+1}\cdot 2\cos(i(t+1)\theta_j) = \Delta^{t+1} (-1 + \cos(i(t+1)\theta_j)).
\end{aligned}
\]
Therefore, $\Delta$ is the governing convergence rate of such $j$-th summand of $H_2(t)$ and the overall convergence rate of $H_2(t)$ is still determined by $\lambda_{2,\max}$. If all $\lambda_{2,j}, j\in [n]$, are complex numbers then the observation above shows that the governing convergence rate of $H_2(t)$ should be $\Delta = \lambda_{2,\max}$ and this proves our claim. 
\end{proof}

When $\lambda_{2,\max}$ takes over the convergence behavior of SGD+M, we can easily see that its convergence dynamics is nothing but its analogue with full batch size but with adjusted learning rate. This can be easily obtained by $\zeta = 1$ in Theorem \ref{thm: concentration_main}, but we provide a statement for full batch SGD+M and its proof for completeness.

\paragraph{Proof of Proposition \ref{prop: concent_fullbatch}}

\begin{proof} 
    Basically, we follow the same arguments introduced in \ref{subsec: f_dynamics}, but with $\zeta = 1$; so we would not have any errors generated by selecting mini-batches. In other words, $\EE_B^{(l,j)} = 0$. This implies the following, which is an analogue of \eqref{eq:system_iter_w}, 
    \begin{equation}
        \begin{pmatrix}
        w_{t+1,j}^2 \\
        w_{t,j}^2 \\
        w_{t+1,j}w_{t,j}
        \end{pmatrix} = \underbrace{\begin{pmatrix}
        \Omega_j^2 &\Delta^2 &-2\Delta\Omega_j\\
        1 & 0 & 0\\
        \Omega_j & 0 & -\Delta 
        \end{pmatrix}}_{=\MM_j}
        \begin{pmatrix}
        w_{t,j}^2 \\
        w_{t-1,j}^2 \\
        w_{t,j}w_{t-1,j}
        \end{pmatrix}.
    \end{equation}
    This implies $w_{t+1,j}^2 = (\MM_j^t\tilde{\mathcal{X}}_{1,j})_1$ and following the same arguments in \ref{subsec: f_dynamics} gives
\[
\begin{aligned}
(\MM_j^t\tilde{\mathcal{X}}_{1,j})_1
&= \frac{2(\frac{R}{n}\sigma_j^2 + \frac{\tilde{R}}{n})}{\Omega_j^2-4\Delta} \left( -\Delta\Gamma_j\cdot \lambda_{1,j}^{t+1} + \frac{1}{2}(1-\Gamma_j - \kappa_{3,j})^2\cdot \lambda_{2,j}^{t+1} + \frac{1}{2}(1-\Gamma_j - \kappa_{2,j})^2\cdot \lambda_{3,j}^{t+1}\right)\\
&+ \frac{2\EE_{w_0}^{(j)}}{\Omega_j^2-4\Delta} \left( -\Delta\Gamma_j\cdot \lambda_{1,j}^{t+1} + \frac{1}{2}(1-\Gamma_j - \kappa_{3,j})^2\cdot \lambda_{2,j}^{t+1} + \frac{1}{2}(1-\Gamma_j - \kappa_{2,j})^2\cdot \lambda_{3,j}^{t+1}\right).
\end{aligned}
\]
Therefore, this leads to
\[
f(t+1) = \frac{R}{2} h_1(t+1) + \frac{\tilde{R}}{2}h_0(t+1) + \EE(t),
\]
with the error term $\EE(t) = \EE_{IC}(t)$. Now taking $n\to \infty$ combined with Proposition \ref{prop:error_initial} gives \eqref{eq:Volterra_eq_fullbatch}. Note that the convergence rate of $\psi_{\text{full}}(t)$ is determined by $\lambda_{2,\max}^{(\text{full})} := \max_j |\lambda_{2,j}^{(\text{full})}|$, where
\[
\lambda_{2,j}^{(\text{full})} = \frac{-2\Delta + (\Omega_j^{(\text{full})})^2 + \sqrt{(\Omega_j^{(\text{full})})^2((\Omega_j^{(\text{full})})^2 - 4\Delta)}}{2},\ \Omega_j^{(\text{full})} \defas 1-\gamma_\text{full}\sigma_j^2 + \Delta.
\]
And observing that $\lambda_{2,j}^{(\text{full})} = \lambda_{2,j}$ if $\gamma_{\text{full}} = \gamma\zeta$ gives our conclusion.
\end{proof}

\subsection{Choice of optimal learning rate and momentum}\label{subsec:optimal_parameters}
In this section, we prove Proposition \ref{prop:low_noise_regime} which states a sufficient condition for a set of learning rate and momentum parameters to be in the problem constrained regime. We also offer the proof of Proposition \ref{prop:lambda_delta_choice}, which gives an optimal learning rate and momentum so that SGD+M outperforms SGD in terms of convergence rate. Finally, the proof of Proposition \ref{prop:lb} will be given as well.
\paragraph{Proof of Proposition \ref{prop:low_noise_regime}}

\paragraph{Remark on the assumption.} The first assumption on the learning rate, i.e., $\gamma \le \frac{1+\Delta}{\zeta\sigma_{\max}^2}$ implies that $\Omega_j \ge 0$ for all $j\in [n]$. On the other hand, the second condition, i.e., $\gamma \le \frac{(1-\sqrt{\Delta})^2}{\zeta \sigma_{\min}^2}$, implies that $\Omega_{\max} \ge 2\sqrt{\Delta}$. Note that when $\Omega_{\max} = 2\sqrt{\Delta}$, $\lambda_{2,\max} = \frac{1}{2}\big( -2\Delta + \Omega_{\max}^2 + \sqrt{\Omega_{\max}^2(\Omega_{\max}^2-4\Delta)}\big) = \Delta$.

\begin{proof}
First recall that $\varphi_j^{(n)} = \frac{(1-\zeta)\gamma\sigma_j^2\Gamma_j}{n}$ and observe that, for $1<\Upsilon<\lambda_{2,\max}^{-1},$
\begin{equation*}\label{eq:Malthusian_1}
\begin{aligned}
    \tilde{\mathcal{K}}(\Upsilon) \defas \sum_{t=0}^\infty \Upsilon^t \mathcal{K}(t) &= \sum_{t=0}^\infty \bigg(\sum_{j=1}^{n} \frac{\varphi_j^{(n)}}{\Omega_j^2 - 4\Delta}(-2\Delta\cdot (\Upsilon\lambda_{1,j})^{t} + \lambda_{2,j} \cdot (\Upsilon\lambda_{2,j})^{t} + \lambda_{3,j} \cdot (\Upsilon\lambda_{3,j})^{t}) \bigg) \\
    & = \sum_{j=1}^{n}\frac{\varphi_j^{(n)}}{\Omega_j^2 - 4\Delta} \left( \frac{-2\Delta}{1-\Upsilon\Delta} + \frac{\lambda_{2,j}}{1-\Upsilon\lambda_{2,j}} + \frac{\lambda_{3,j}}{1-\Upsilon\lambda_{3,j}} \right) \\
    &= \sum_{j=1}^{n}\frac{\varphi_j^{(n)}}{\Omega_j^2 - 4\Delta} \bigg( \frac{-2 \Delta}{1- \Upsilon\Delta} +
         \frac{-2\Delta + \Omega_j^2 - 2\Upsilon\Delta^2}{1 + \Upsilon(2\Delta - \Omega_j^2) + \Upsilon^2\Delta^2} \bigg) \\
         &= \sum_{j=1}^{n}\frac{(1-\zeta)\zeta \gamma^2 \sigma_j^4}{n }\bigg(\frac{(1+\Upsilon\Delta)}{(1-\Upsilon\Delta)(1 - \Upsilon(-2\Delta + \Omega_j^2) + \Upsilon^2\Delta^2)}\bigg) \\
         & = \sum_{j=1}^{n}\frac{C \zeta \gamma \sigma_j^4}{n }\bigg(\frac{(1-\Delta)(1+\Upsilon\Delta)}{(1-\Upsilon\Delta)(1+\Upsilon\Delta + \sqrt{\Upsilon}\Omega_j)(1+\Upsilon\Delta - \sqrt{\Upsilon}\Omega_j)} \bigg),
\end{aligned}
\end{equation*}
where $C = (1-\zeta)\gamma/(1-\Delta)$. Observe, as $\Omega_j \ge 0$,
\begin{equation}\label{eq:Kappa_Upsilon_ub}
\tilde{\mathcal{K}}(\Upsilon) \le \frac{C \zeta \gamma}{n }\cdot\frac{(1-\Delta)}{(1-\Upsilon\Delta)}\sum_{j=1}^{n}\frac{\sigma_j^4}{1+\Upsilon\Delta - \sqrt{\Upsilon}\Omega_j}.
\end{equation}

Let us analyze the denominator of the summand first. Let $f_j(x) := 1+x^2\Delta - x\Omega_j, 1<x <\sqrt{\Delta^{-1}}$. Then the denominator in the summand is $f_j(\sqrt{\Upsilon})$. Especially, $f_{\min}(x) := \min_j f_j(x) = 1+x^2\Delta - x\Omega_{\max}, \Omega_{\max} = 1-\gamma\zeta{\sigma_{\min}}^2 + \Delta$. Note that $f_{\min}(x)$ is a quadratic function of $x$ and the solution to $f_{\min}(x) = 0$ is $x = \sqrt{\lambda_{2,\max}^{-1}}$ (the other root $\sqrt{\lambda_{3,\max}^{-1}}$ exceeds the valid domain of $x$). Also, observe that this is where the assumption $\Omega_{\max} \ge 2\sqrt{\Delta}$ is used.

Note that $f_j(1) = \gamma\zeta\sigma_j^2$. Simple algebra shows that for $1<x<\alpha<\beta$, $c_1(x-\alpha)^2 \le c_2(x-\alpha)(x-\beta)$ where $c_1,c_2>0$ satisfies $c_1(1-\alpha)^2 = c_2(1-\alpha)(1-\beta)$, i.e. two functions coincide at $x=1$ and $x=\alpha$. If $\lambda_{2,j} \ge 0$, or $\Omega_j^2 - 4\Delta \ge 0$, then the argument above gives
\[
f_{\min}\Big(\frac{1+\sqrt{\lambda_{2,\max}^{-1}}}{2}\Big) \ge \frac{\gamma\zeta\sigma_{\min}^2}{4}.
\]
Now for any $j \in [n]$, note that $f_j(x) - f_{\min}(x) = x\gamma\zeta(\sigma_j^2 - \sigma_{\min}^2)$ is an increasing function of $x\in\mathbb{R}$. So observe,
\[
\begin{aligned}
f_j\Big(\frac{1+\sqrt{\lambda_{2,\max}^{-1}}}{2}\Big) &\ge f_{\min}\Big(\frac{1+\sqrt{\lambda_{2,\max}^{-1}}}{2}\Big) + \gamma\zeta(\sigma_j^2 - \sigma_{\min}^2)\\
&\ge \frac{\gamma\zeta\sigma_{\min}^2}{4} + \frac{1}{4}\gamma\zeta(\sigma_j^2 - \sigma_{\min}^2) = \frac{1}{4}\gamma\zeta\sigma_j^2.
\end{aligned}
\]

Therefore, when $\sqrt{\Upsilon} = \frac{1+\sqrt{\lambda_{2,\max}^{-1}}}{2}$, \eqref{eq:Kappa_Upsilon_ub} gives
\[
\begin{aligned}
\tilde{\mathcal{K}}(\Upsilon) &\le \frac{4C}{n}\cdot\frac{(1-\Delta)}{(1-\Upsilon\Delta)} \sum_{j=1}^{n}\sigma_j^2.
\end{aligned}
\]
Moreover, in order to bound the denominator $(1-\Upsilon\Delta)$ on the right-hand side, if we define $g(x) \defas 1-\Delta x^2$, $g$ is a decreasing function on $[1,\sqrt{\Delta^{-1}}]$ and
\[
g(\frac{1+\sqrt{\lambda_{2,\max}^{-1}}}{2}) \ge g(\frac{1+\sqrt{\Delta^{-1}}}{2}) \ge \frac{1-\Delta}{2},
\]
by considering a linear line passing through $(1,1-\Delta)$ and $(\sqrt{\Delta^{-1}},0)$ that lies below $g$. Therefore,
\[
\begin{aligned}
\tilde{\mathcal{K}}(\Upsilon) &\le \frac{4C}{n}\cdot\frac{(1-\Delta)}{(1-\Upsilon\Delta)}\cdot \sum_{j=1}^{n}\sigma_j^2 \le  \frac{8(1-\zeta)\gamma}{(1-\Delta)}\frac{1}{n}\tr(\AA^T\AA).
\end{aligned}
\]
\end{proof}

\paragraph{Proof of Proposition \ref{prop:lambda_delta_choice}}
\begin{proof}
First, when the assumption $\frac{(1-\sqrt{\Delta})^2}{\zeta\sigma_{\min}^2} \le \frac{(1+\sqrt{\Delta})^2}{2\zeta\sigma_{\max}^2}$ is met, we have
\[
\frac{1-\sqrt{\Delta}}{1+\sqrt{\Delta}} \le \frac{1}{\sqrt{2\kappa}}.
\]
Solving this inequality with respect to $\Delta$ gives
\[
\Delta \ge  \left( \frac{1-\frac{1}{\sqrt{2 \kappa}}}{1+\frac{1}{\sqrt{2\kappa}}} \right )^2.
\]
Furthermore, from Proposition \ref{prop:low_noise_regime}, when $\gamma = \frac{(1-\sqrt{\Delta})^2}{\zeta\sigma_{\min}^2}$, observe that $\lambda_{2,\max} = \Delta$ and 
\[
\frac{8(1-\zeta)\gamma}{(1-\Delta)}\frac{1}{n}\tr(\AA^T\AA) = \frac{8(1-\zeta)(1-\sqrt{\Delta})^2}{\zeta\sigma_{\min}^2(1-\Delta)}\frac{1}{n}\tr(\AA^T\AA) = \frac{8(1-\zeta)}{\zeta}\cdot\frac{1-\sqrt{\Delta}}{1+\sqrt{\Delta}} \bar{\kappa} < 1.
\]
Therefore, this condition implies
\[
\frac{1-\sqrt{\Delta}}{1+\sqrt{\Delta}} < \frac{\mathcal{C}}{\bar{\kappa}},
\]
where $\mathcal{C} = \mathcal{C}(\zeta) \defas \zeta/(8(1-\zeta))$ and solving this inequality gives
\[
\sqrt{\Delta} >  \frac{1-\frac{\mathcal{C}}{\bar{\kappa}}}{1+\frac{\mathcal{C}}{\bar{\kappa}}}.
\]
\end{proof}

Next, we present the proof of Proposition \ref{prop:lb}.
\paragraph{Proof of Proposition \ref{prop:lb}}
\begin{proof}
For brevity and clarity, we define the following quantities:
\[
\gamma_1 \defas \frac{1+\Delta}{\zeta\sigma_{\max}^2},\quad \gamma_2 \defas \frac{(1-\sqrt{\Delta})^2}{\zeta\sigma_{\min}^2},\ \text{and}\ \gamma_3 \defas \frac{1}{\bar{\kappa}\sigma_{\min}^2}\cdot\frac{1-\Delta}{1-\zeta}.
\]
Note that the assumptions on the learning rate $\gamma$ in Proposition \ref{prop:kernel_norm} imply that $\gamma \le \min(\gamma_1, \gamma_3)$.

First, let us assume that $\gamma \ge \gamma_2$. Recall that this condition implies that $\Omega_{\max}^2-4\Delta \le 0$ and therefore $\lambda_{2, \max} = \Delta$. In this case, $\gamma_2 \le \gamma \le \gamma_3$ implies that
\[
\begin{gathered}
   \frac{(1-\sqrt{\Delta})^2}{\zeta\sigma_{\min}^2} \le \frac{1}{\bar{\kappa}\sigma_{\min}^2}\cdot\frac{1-\Delta}{1-\zeta}
   \Rightarrow \frac{1-\sqrt{\Delta}}{1+\sqrt{\Delta}} \le \frac{\zeta}{(1-\zeta)\bar{\kappa}},\ \text{or}\\
   \sqrt{\Delta} \ge \frac{1- \frac{\zeta}{(1-\zeta)\bar{\kappa}}}{1+ \frac{\zeta}{(1-\zeta)\bar{\kappa}}}.
\end{gathered}
\]

So, combining the condition $\zeta \le 1/2$ with the above inequality gives the claim. Therefore, for the following arguments, we assume that $\gamma \le \gamma_2$. It is worthwhile to note that by the definition of $\lambda_{2,\max}$ and $\Omega_{\max} = 1 - \gamma\zeta\sigma_{\min}^2+\Delta$, we know that $\lambda_{2,\max}$ is an increasing function of $\Omega_{\max}$ when $\Omega_{\max}^2 - 4\Delta \ge 0$ and $\Omega_{\max} \ge 0$ and $\Omega_{\max}$ is a decreasing function of $\gamma$. Therefore, $\lambda_{2,\max}$ attains its minimum at the maximum feasible learning rate $\gamma$. 

First, let as assume that $\gamma \le \gamma_3 \le \gamma_1$. Then $\lambda_{2,\max}$ attains its minimum at $\gamma = \gamma_3$ and
\[
\Omega_{\max} \ge 1+\Delta - \zeta\sigma_{\min}^2\cdot\frac{1}{\bar{\kappa}\sigma_{\min}^2}\cdot\frac{1-\Delta}{1-\zeta}
= 1+\Delta - \frac{\zeta}{(1-\zeta)\bar{\kappa}}(1-\Delta).
\]
By observing that
\[
\sqrt{\lambda_{2,\max}} = \frac{\Omega_{\max} + \sqrt{\Omega_{\max}^2-4\Delta}}{2},
\]
we have
\[
\sqrt{\lambda_{2,\max}} \ge \frac{1+\Delta-c_1(1-\Delta) + \sqrt{1+\Delta-c_1(1-\Delta)^2-4\Delta}}{2} =: f_1(\Delta),
\]
where $c_1 \defas \frac{\zeta}{(1-\zeta)\bar{\kappa}} < 1$. One can easily verify that $f_1$ is an increasing function of $\Delta, 0\le \Delta < 1$, so we conclude that
\[
\sqrt{\lambda_{2,\max}} \ge \sqrt{\lambda_{2,\max}}|_{\Delta = 0} = 1-c_1,
\]
and we obtain the claim with the condition $\zeta \le 1/2$.

Second, now we assume that $\gamma \le \gamma_1\le \gamma_3$. Then $\lambda_{2,\max}$ attains its minimum at $\gamma = \gamma_1$ and
\[
\Omega_{\max} \ge 1+\Delta - \frac{1+\Delta}{\sigma_{\max}^2}\cdot\sigma_{\min}^2 = (1+\Delta)(1-\frac{1}{\kappa}).
\]
Therefore, for the same argument as above, we have
\[
\sqrt{\lambda_{2,\max}} \ge \frac{(1+\Delta)(1-c_2) + \sqrt{(1+\Delta)^2(1-c_2)^2-4\Delta}}{2} =: f_2(\Delta),
\]
where $c_2 \defas 1/\kappa$. On the other hand, the condition $\gamma_1 \le \gamma_3$ gives
\begin{equation}\label{eq:gamma13upperbound}
\begin{gathered}
\frac{1+\Delta}{\zeta\sigma_{\max}^2} \le \frac{1}{\bar{\kappa}\sigma_{\min}^2}\cdot\frac{1-\Delta}{1-\zeta} \Rightarrow \frac{1-\Delta}{1+\Delta} \ge \frac{\bar{\kappa}}{\kappa}\cdot\frac{1-\zeta}{\zeta},\ \text{or}\\
\Delta \le \frac{1 - \frac{\bar{\kappa}}{\kappa}\cdot\frac{1-\zeta}{\zeta} }{ 1+ \frac{\bar{\kappa}}{\kappa}\cdot\frac{1-\zeta}{\zeta}} =: \Delta_*.
\end{gathered}
\end{equation}

Let us define $c_3 \defas \frac{\bar{\kappa}}{\kappa}\cdot\frac{1-\zeta}{\zeta} < 1$. Then it suffices to show that $\sqrt{\lambda_{2,\max}} \ge 1-D\frac{c_2}{c_3}$ for some $D>0$. 

Simple algebra shows that $f_2$ is a concave function on $[0, \Delta_u]$ where $\Delta_u \defas \frac{1-\sqrt{2c_2 - c_2^2}}{1-c_2}$ 
makes the radical in the numerator of $f_2$ vanish. Also, one can verify that $f_2(0) = 1-c_2 \ge \frac{1-c_2 + \sqrt{-2c_2 + c_2^2 + c_3^2}}{1+c_3} = f_2(\Delta_*)$ and $\Delta_* \le \Delta_u$, so that $f_2(\Delta) \ge f_2(\Delta_*)$ on $[0,\Delta_*]$. Hence, it suffices to show that $f_2(\Delta_*) \ge  1-D\frac{c_2}{c_3}$ for some $D>0$. Observe,
\[
\begin{aligned}
  f_2(\Delta_*) &= \frac{1-c_2 + \sqrt{-2c_2 + c_2^2 + c_3^2}}{1+c_3}\\
  &= \frac{1-c_2 + c_3\sqrt{1- \frac{2c_2 - c_2^2}{c_3^2}}}{1+c_3}\\
  & \ge \frac{1-c_2 + c_3(1- \frac{2c_2 - c_2^2}{c_3^2})}{1+c_3}\\
  & = 1 - \frac{\frac{c_2}{c_3}(2+c_3-c_2)}{1+c_3}\\
  &\ge 1 - 3\frac{c_2}{c_3},
\end{aligned}
\]
and we finish the proof.
\end{proof}

\section{Numerical Simulations}\label{apx:app_numerical_simulations}
To illustrate our theoretical results, we compare SGD+M's dynamics to \eqref{eq:Volterra_eq_main} on moderately sized problems ($n \approx 1000$) under the setting of section \ref{sec:motivation}. Moreover, the dynamics were also compared using the MNIST data set. Finally, heat maps were displayed to illustrate the interplay between the algorithmic and problem constraints. 
\paragraph{Random least squares.} In all simulations of the Gaussian random least squares problem, the initial weight vector $\xx_0$ is set to zero and the signal and noise vectors $\tilde{\xx}$ and $\eeta$ are set to $N(0, \frac{R}{n}\II)$ and $N(0,\frac{\tilde{R}}{n}\II)$ respectively. Moreover, $\AA$ is constructed by independently sampling its entries $A_{ij} \sim N(0, 1)$ then row-normalized. Similarly, $\bb$ is first sampled $\bb \sim N(0, \frac{\tilde{R}d}{n}\II)$ then the $i$-th entry of $\bb$ is divided by the norm of the $i$-th row of $\AA$. The objective function in which we run SGD+M in all cases is the least squares objective function $f(\xx) = \frac{1}{2}||\AA\xx - \bb||^2$. 

\paragraph{Empirical Volterra equation.} 
We assume that we have access to the eigenvalues of the matrix $\AA \AA^T$. The empirical Volterra equation  \eqref{eq:Volterra_eq_main} were computed using a dynamic programming approach by using as inputs the eigenvalues of $\AA\AA^T$. First, the values of  $h_0(t), h_1(t),$ $H_2(t)$ were computed and stored for values of $t\in[T]$. Then a dynamic programming approach is used to compute $\psi(t)$ for values of $t\in[T]$. The discrete convolution operation in \eqref{eq:Volterra_eq_main} is computed by an array reversal and Numpy dot product.

\begin{center}
\begin{figure}[htp!]
\centering 
    \includegraphics[scale = 0.15]{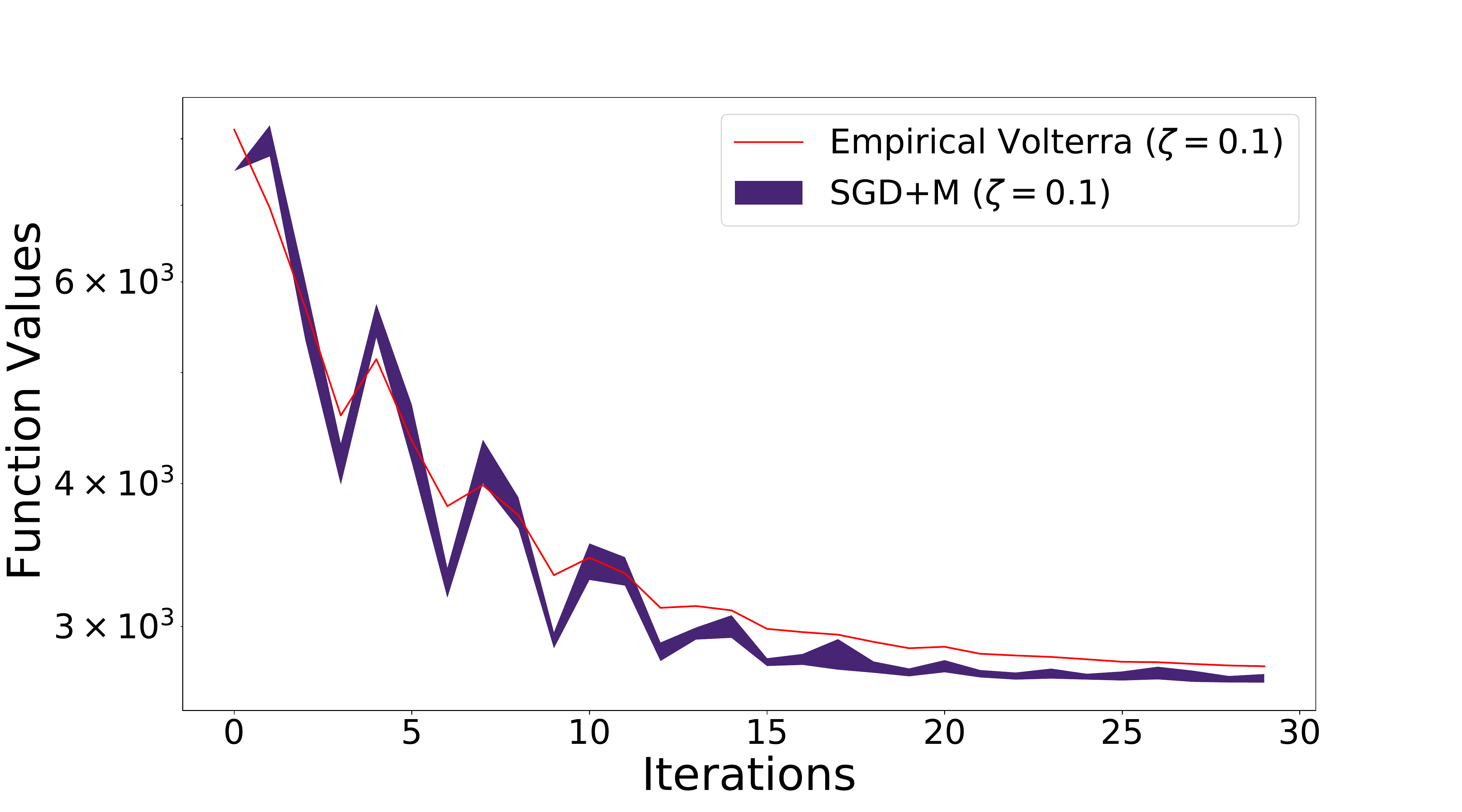}
     \includegraphics[scale =0.15]{figures/MNISTMidBatch.pdf}
      \includegraphics[scale = 0.15]{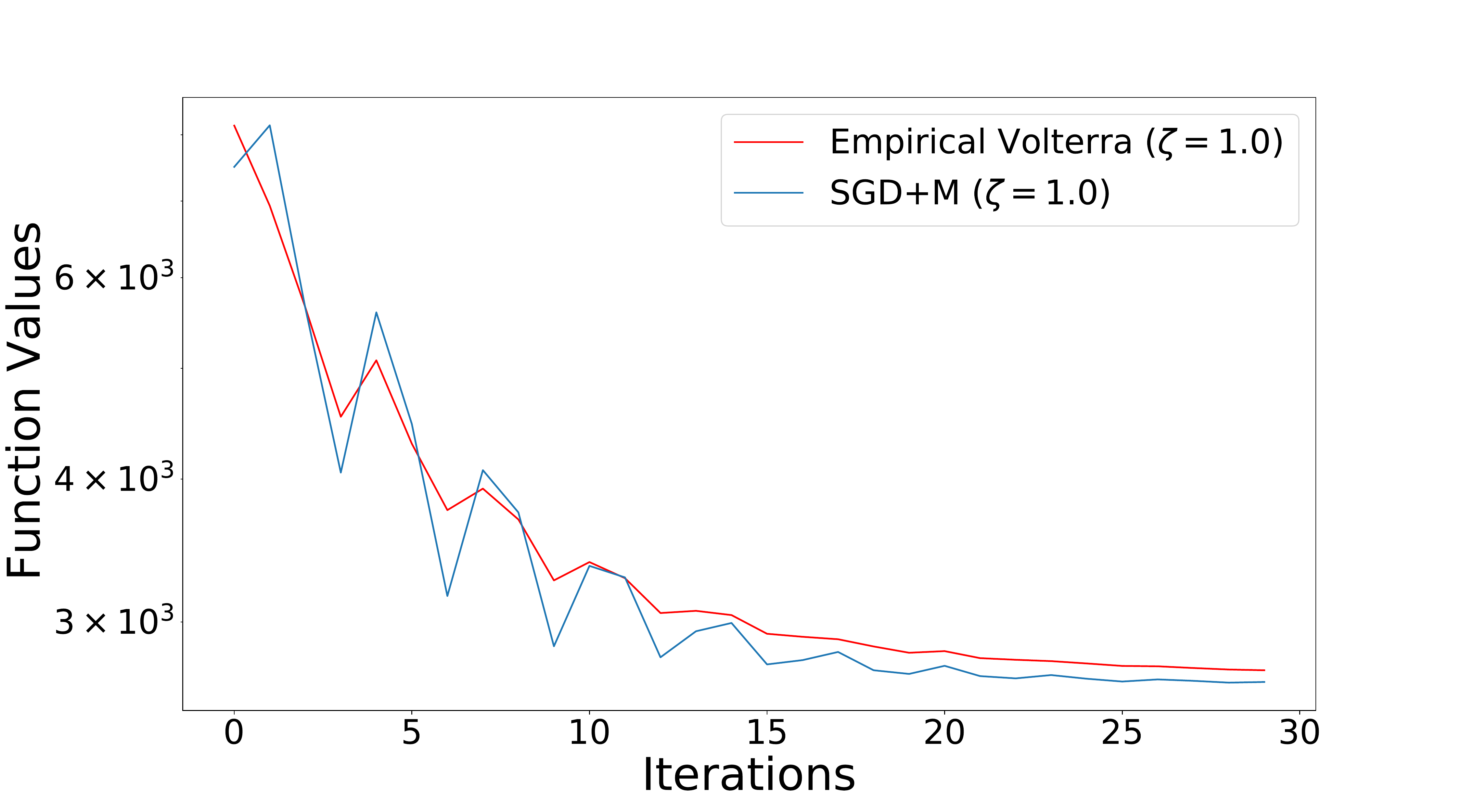}
    \caption{{\bfseries SGD+M vs. Theory on even/odd MNIST.} 
    MNIST ($60,000 \times 28 \times 28$ images) \citep{lecun2010mnist} is reshaped into a single matrix of dimension $60,000 \times 784$ (preconditioned to have centered rows of norm-1), representing 60,000 samples of 10 digits. The target $\bb$ satisfies $\bb_i = 0.5$ if the $i^{th}$ sample is an odd digit and $\bb_i = -0.5$ otherwise. SGD+M was run $10$ times with $\Delta = 0.8$, various values of $\zeta$, and learning rates $\gamma = 0.005, 0.001, 0.0005$ (left to right, top to bottom) and empirical Volterra was run once with $(R=11,000$, $\tilde{R} = 5300)$. The $10^{th}$ to $90^{th}$ percentile interval is displayed for the loss values of 10 runs of SGD+M. Volterra predicts the convergent behavior of SGD+M in this setting.
}
\label{fig:mnistSmallBatch}
\end{figure}
\end{center}
\vspace{-7mm}

\paragraph{Volterra equation with Marchenko-Pastur distribution.} In this setting, we use the theoretical limiting distribution for a large class of random matrices. In a celebrated work by \citep{marchenko1967}, when the entries of $(n \times d)$ matrix $\AA$ are drawn from a common, mean $0$, variance $1/d$ distribution with fourth moment $\mathcal{O}(d^{-2})$ (e.g., Gaussian $N(0,\frac{1}{d})$), it is known that the distribution of eigenvalues of $\AA \AA^T$ converges to the Marchenko-Pastur law


\begin{equation}
\begin{gathered}
    d\mu_{MP}(\lambda) \defas \delta_0(\lambda)\max\{1-r, 0\} + \frac{r\sqrt{(\lambda-\lambda^-)(\lambda^+-\lambda)}}{2\pi\lambda}1_{[\lambda^-, \lambda^+]},\\
    \text{where}\quad \lambda^-\defas (1-\sqrt{\frac{1}{r}})^2\quad\text{and}\quad \lambda^+\defas (1+\sqrt{\frac{1}{r}})^2.
\end{gathered}
\end{equation}

For these experiments, we generated the data matrix $\AA$ with entries $N(0, 1/d)$. Instead of using the eigenvalues of $\AA \AA^T$ in the Volterra equation \eqref{eq:Volterra_eq_main}, we used the Marchenko-Pastur distribution directly.  We used a Chebyshev quadrature rule to approximate the integrals with respect to the Marchenko-Pastur distribution that arise in  \eqref{eq:Volterra_eq_main}. Similar to the finite case, the integrand is computed using dynamic-programming. However, the implementation of the quadrature rule ignores the point mass at $0$ so we manually add this at the end. 

\paragraph{Volterra simulations remarks.}
Despite the numerical approximations to the integral, the resulting solution to the Volterra equation $\psi$
(red line in figure \ref{fig:volterra_sgdm_comparison}) models the true behavior of SGD+M remarkably well. Notably, the fit of the Volterra equation to SGD+M is extremely accurate across various learning rates, batch sizes, and momentum parameters as long as the learning rate condition is satisfied. In Figure 1, the red line corresponds to the Volterra equation with Marchenko-Pastur distribution with values $R = \tilde{R} = 1$. Also, we opted to shade the $10^{th}$ to $90^{th}$  percentile instead of an $\alpha$-confidence interval for an easier read. One can observe the exact same dynamics in either case.


\paragraph{Heat maps.}
The heat maps (Figure \ref{fig:heatmap}) illustrate when the convergence rate is dictated by the problem, ($\lambda_{2,\max} \geq \Xi^{-1}$) or by the algorithm ($\lambda_{2,\max} < \Xi^{-1}$). The white regions of the heat maps represent divergent behaviour ($\lambda_{2,\max} > 1$). The threshold, denoted by the red line, describes the boundary for two different regimes. 
Any non-white point above or to the right of the threshold lies in the algorithmic constraint setting. Conversely, all non-white points lying below or to the left of the threshold lies in the problem constraint setting.

The heat maps are generated by computing $\lambda_{2,\max}$ and $\Xi$ (when it exists) across values of $(\Delta, \gamma)$. Here $\lambda_{2,\max}$ is obtained by calculating
\begin{equation*}
    \lambda_{2,\max} = \frac{-2\Delta + \Omega_{\max}^2 + \sqrt{\Omega_{\max}^2(\Omega_{\max}^2 - 4\Delta)}}{2}, \ \Omega_{\max} =  1-\gamma\zeta \sigma_{\min}^2 + \Delta, \text{ and } \sigma_{\min}^2 = \big(1 - \sqrt{\frac{1}{r}}\big)^2.
\end{equation*}
In order to compute $\Xi$, recall that $\Xi$ is the solution of 
\begin{equation}\label{eq:malthexp}
    \tilde{\mathcal{K}}(\Xi) \defas \sum_{t=0}^{\infty} \Xi^t \mathcal{K}(t) = 1,
\end{equation}
when it exists. One can show (\ref{eq:malthexp}) is equal to (see Appendix \ref{subsec:optimal_parameters} for detail)
\begin{equation}\label{eq:malthsolver}
    \begin{aligned}
   \sum_{j=1}^{n}\frac{ \zeta(1-\zeta) \gamma^2 \sigma_j^4}{n }\bigg(\frac{(1+\Xi\Delta)}{(1-\Xi\Delta)(1+\Xi\Delta + \sqrt{\Xi}\Omega_j)(1+\Xi\Delta - \sqrt{\Xi}\Omega_j)} \bigg) = 1, 
\end{aligned}
\end{equation}
which is computed using the Chebyshev quadrature rule.

For a given $(\Delta, \gamma)$, we are interested in the algorithmic case ($1 \leq \Xi \leq \lambda_{2,\max}^{-1}$) so if $\lambda_{2,\max}^{-1}< 1$ we assign a Nan value to $\Xi$. Otherwise, because of monotonicity of $\tilde{\mathcal{K}}$ in $\eqref{eq:malthexp}$, we perform a binary search starting with initial endpoints $1$ and $\lambda_{2,\max}^{-1}$ to find the solution $\Xi$ satisfying $\eqref{eq:malthexp}$. Finally, with $\Xi^{-1}$ and $\lambda_{2,\max}$ computed for a given $(\Delta, \gamma)$, we plot the maximum of the two.




\begin{figure}[t]
    \centering
    \includegraphics[width=\linewidth]{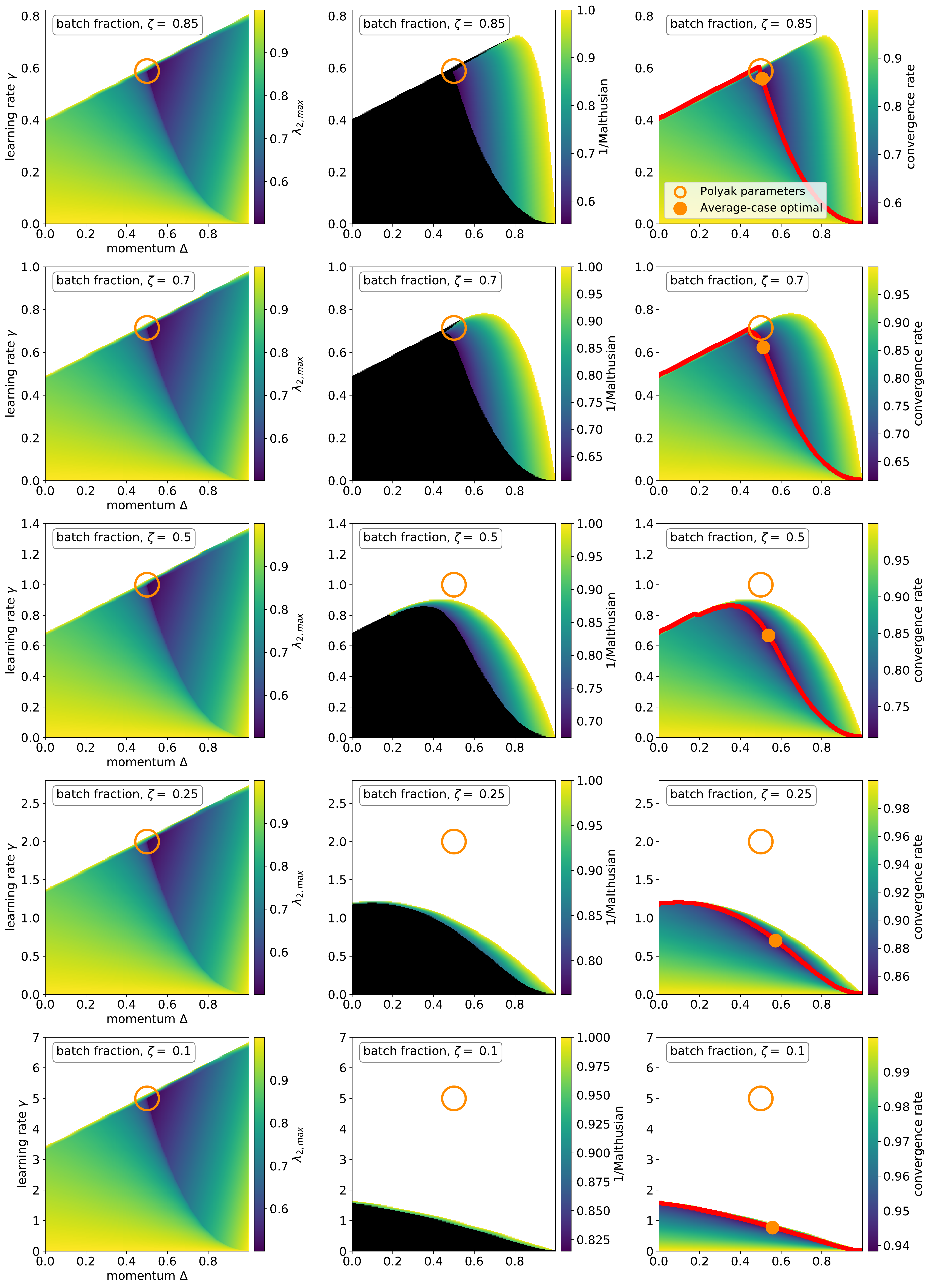}
    \vspace{-0.5cm} \caption{{\bfseries Different convergence rate regions.} Same set-up as in Figure~\ref{fig:heatmap} but for a wider range of batch fractions.
}
    \label{fig:heatmap_full}
\end{figure}




\end{document}